\RequirePackage[l2tabu, orthodox]{nag}
\documentclass[a4paper,11pt,reqno]{amsart}

\usepackage{amssymb,amscd,amsthm}
\usepackage{mathrsfs}
\usepackage{hyperref}

\usepackage[ansinew]{inputenc}
\usepackage[centertags]{amsmath}
\usepackage{epsfig,graphicx,enumerate}

\usepackage{xcolor}

\newcommand{\N}{\mbox{$\mathbb{N}$}}
\newcommand{\Z}{\mbox{$\mathbb{Z}$}}

\newcommand{\R}{\mbox{$\mathbb{R}$}}

\newcommand{\GL}{\mathrm{GL}}
\newcommand{\SL}{\mathrm{SL}}

 \def\NN{{\mathbb N}}  \def\PP{{\mathbb P}}
 \def\RR{{\mathbb R}}  
  
 \def\ZZ{{\mathbb Z}}

  \def\cG{\mathcal{G}}  
    
\def\cC{\mathcal{C}}    \def\cU{\mathcal{U}}
   \def\cP{\mathcal{P}} 
   \def\cQ{\mathcal{Q}}

    \def\sx{\mathsf{x}}
\def\sy{\mathsf{y}}  \def\sz{\mathsf{z}}

\newcommand{\aaa}{\mathfrak{a}}
\newcommand{\semig}{\Lambda}

\newcommand{\ie}{i.e.\ }
\newcommand{\eg}{e.g.\ }
\newcommand{\resp}{resp.\ }

\newtheorem{teo}{Theorem}[section]
\newtheorem{thm}[teo]{Theorem}
\newtheorem{quest}[teo]{Question}
\newtheorem{cor}[teo]{Corollary}
\newtheorem{lema}[teo]{Lemma}
\newtheorem{prop}[teo]{Proposition}
\newtheorem{fact}[teo]{Fact}
\theoremstyle{definition}
\newtheorem{defi}[teo]{Definition}
\newtheorem{ex}[teo]{Example}
\theoremstyle{remark}
\newtheorem{remark}[teo]{Remark}
\newtheorem{remarks}[teo]{Remarks}

\numberwithin{equation}{section}

\DeclareMathOperator{\Hom}{Hom}

\author[F. Kassel]{Fanny Kassel}
\address{CNRS and Laboratoire Alexander Grothendieck, Institut des Hautes \'Etudes Scientifiques, Universit\'e Paris-Saclay, 35 route de Chartres, 91440 Bures-sur-Yvette, France}
\email{kassel@ihes.fr}

\author[R. Potrie]{Rafael Potrie}
\address{CMAT, Facultad de Ciencias, Universidad de la Rep\'ublica, Igu\'a 4225, Montevideo, 11400, Uruguay}
\urladdr{www.cmat.edu.uy/$\sim$rpotrie}
\email{rpotrie@cmat.edu.uy}

\title{Eigenvalue gaps for hyperbolic groups and semigroups}
\thanks{This project received funding from the European Research Council (ERC) under the European Union's Horizon 2020 research and innovation programme (ERC starting grant DiGGeS, grant agreement No.\ 715982).
F.K.\ was partially supported by the Louis D. foundation.
R.P.\ was partially supported by CSIC-618, FCE-135352, FCE-148740 and MathAmSud. This work was completed while R.P. was a von Neumann fellow at IAS, funded by Minerva Research Fundation Membership Fund and NSF DMS-1638352.}

\begin{document}

\maketitle

\begin{abstract}
Given a locally constant linear cocycle over a subshift of finite type, we show that the existence of a uniform gap between the $i$-th and $(i+1)$-th Lyapunov exponents for all invariant measures implies the existence of a dominated splitting of index~$i$.
We establish a similar result for sofic subshifts coming from word hyperbolic groups, in relation with Anosov representations of such groups.
We discuss the case of finitely generated semigroups, and propose a notion of Anosov representation in this setting.
\end{abstract}
  
\medskip

\section{Introduction}

The aim of this paper is to illustrate the idea that certain approaches coming from discrete subgroups of Lie groups may be useful to study linear cocycles in dynamics, while conversely the point of view of cocycles may be interesting to understand discrete subgroups and subsemigroups of Lie groups.
For this we focus on three related but somewhat independent topics, whose common feature is to involve eigenvalue gaps for sequences of matrices.

\subsection{Dominated splittings for locally constant cocycles and Lyapunov exponents} \label{subsec:intro-finite-type}

The first topic of the paper is dominated splittings.
Using ideas from the theory of discrete subgroups of Lie groups, we obtain the following characterization in terms of the Lyapunov exponents of invariant measures (see Section~\ref{sec:prelim} for definitions).

\begin{thm} \label{thm:charact-domin-splitting}
Let $X$ be a subshift of finite type and $(\sigma,\Phi)$ a locally constant cocycle over~$X$, where $\sigma : X\to X$ is the shift and $\Phi : X \to \GL(d,\RR)$ for some $d\in\NN^*$.
For $1\leq i\leq d-1$, the cocycle $(\sigma,\Phi)$ has a dominated splitting of index~$i$ if and only if there exists $c>0$ such that for every ergodic $\sigma$-invariant measure $\nu$ on~$X$ the difference between the $i$-th and~$(i+1)$-th Lyapunov exponents of~$\nu$ is $\geq c$.
\end{thm}

It is easy to show that if $(\sigma,\Phi)$ has a dominated splitting of index~$i$, then there exists $c>0$ such that for every ergodic $\sigma$-invariant measure $\nu$ on~$X$, the difference between the $i$-th and~$(i+1)$-th Lyapunov exponents of~$\nu$ is $\geq c$ (Fact~\ref{fact:domin->Lyap-gap}); the point of Theorem~\ref{thm:charact-domin-splitting} is the converse.

Note that in Theorem~\ref{thm:charact-domin-splitting} we do \emph{not} make any irreducibility assumption on the image of~$\Phi$; the proof involves a reduction to a sum of irreducible actions (\emph{semisimplification}) inspired by \cite{GGKW}.
Theorem~\ref{thm:charact-domin-splitting} actually holds for cocycles with values in any noncompact reductive linear Lie group~$G$ (see Section~\ref{sec:general-reductive}).

In Section~\ref{s.nonunifgap} we give an example showing that in general it is necessary for the gap in Lyapunov exponents to be uniform in order to have a dominated splitting, although in certain specific cases (\eg in dimension $d=2$) this can be relaxed.
We note that our proof of Theorem~\ref{thm:charact-domin-splitting} only uses the a priori weaker hypothesis that there is a uniform gap of Lyapunov exponents for \emph{periodic} measures (see Theorem~\ref{thm:sufficient-domin-splitting}), but this is in fact equivalent by \cite[Th.\,1.4]{Kalinin}. 

In dimension $d=2$, Theorem~\ref{thm:charact-domin-splitting} follows from a stronger result due to Avila--Bochi--Yoccoz \cite{ABY}.
In a recent note, Velozo \cite{Velozo} extended this result to the case of two-dimensional cocycles that are not necessarily locally constant, but satisfy a weaker assumption of \emph{fiber bunching}; this assumption is necessary (see \cite{BochiGaribaldi, Butler, Park} for more details).
Other recent interesting related results can be found in the paper \cite{BS} of Breuillard--Sert or pointed out in some questions in \cite{Bochi-ICM}.
See also the end of Section~\ref{subsec:prelim-domin-splitting} for further discussion.

\subsection{Eigenvalue gaps for representations of finitely generated groups} \label{subsec:intro-hyp-gps}

The second topic of the paper is group representations with a uniform gap in the exponential growth rate of eigenvalues.
Let $\Gamma$ be a group with a finite generating subset~$F$.
Let $|\cdot|_F : \Gamma\to\nolinebreak\NN$ be the word length and $|\cdot|_{F,\infty} : \Gamma\to\NN$ the stable length of~$\Gamma$ with respect to~$F$ (see Section~\ref{subsec:terminology-gaps}).
For $g\in\GL(d,\RR)$, we denote by $\mu_1(g)\geq\dots\geq\mu_d(g)$ (\resp $\lambda_1(g)\geq\dots\geq\lambda_d(g)$) the logarithms of the singular values (\resp of the moduli of the eigenvalues) of~$g$.
For $1\leq i\leq d-1$, we say that a representation $\rho : \Gamma\to\GL(d,\RR)$ has a \emph{uniform $i$-gap in singular values} (\resp \emph{in eigenvalues}) if there exist $c,c'>0$ such that $\mu_i(\rho(\gamma)) - \mu_{i+1}(\rho(\gamma)) \geq c \, |\gamma|_F - c'$ (\resp $\lambda_i(\rho(\gamma)) -\nolinebreak \lambda_{i+1}(\rho(\gamma)) \geq c \, |\gamma|_{F,\infty} - c'$) for all $\gamma\in\Gamma$; this condition does not depend on the choice of~$F$ (see Remark~\ref{rem:change-gen-set}).
We say that $\Gamma$ is \emph{word hyperbolic} if its Cayley graph is Gromov hyperbolic.
With this terminology, we prove the following.

\begin{prop} \label{prop:eig-sing-gap-group}
Let $\Gamma$ be a word hyperbolic group and let $1\leq i<d$ be integers.
A representation $\rho: \Gamma \to \GL(d,\RR)$ has a uniform $i$-gap in eigenvalues if and only if it has a uniform $i$-gap in singular values.
\end{prop}

Proposition~\ref{prop:eig-sing-gap-group} answers a question of \cite{BPS}.
By a result of Kapovich--Leeb--Porti \cite{KLP} (Fact~\ref{fact:Ano-sg-gap}, for which an alternative proof was given in \cite{BPS}), the representations of~$\Gamma$ with a uniform $i$-gap in singular values are exactly the so-called \emph{$P_i$-Anosov representations} of~$\Gamma$, which play an important role in higher Teichm\"uller theory (see \eg \cite[\S\,4.3]{Kassel-ICM} and \cite[\S\,3.3]{wie-icm}).
Their images are discrete subgroups of $\GL(d,\RR)$ with good dynamical, geometric, and topological properties: see \eg \cite{lab06,gw12,BCLS,klp-survey,GGKW,DGK,zim}.
Thus Proposition~\ref{prop:eig-sing-gap-group} yields a new characterization of Anosov representations of~$\Gamma$ (Corollary~\ref{cor:char-Ano-eigenv}), as announced in \cite{Kassel-ICM,Potrie-ICM}.
See Section~\ref{sec:general-reductive} for a generalization to representations to any noncompact reductive linear Lie group.
Note that representations with a uniform $i$-gap in singular values are also called \emph{$i$-dominated representations} in \cite{BPS}.

The equivalence in Proposition~\ref{prop:eig-sing-gap-group} becomes more subtle when $\Gamma$ is a finitely generated group that is not word hyperbolic: see Section~\ref{subsec:eig-gap-nonhyp-gp}.

Using Proposition~\ref{prop:eig-sing-gap-group}, we obtain an analogue of Theorem~\ref{thm:charact-domin-splitting} for locally constant cocycles over certain sofic subshifts which are not necessarily of finite type, but come from word hyperbolic groups.
More precisely, for a group $\Gamma$ with a finite generating subset~$F$ as above, consider the closed subset
\begin{equation} \label{eqn:G-F-group}
\cG_F = \big\{ (f_k)_{k\in\ZZ} \in (F\cup F^{-1})^{\ZZ} \ : \ |f_k \cdots f_{k+\ell}|_F = \ell+1 \quad \forall k\in\ZZ,\ \forall\ell\in\NN\big\}\end{equation}
of~$(F\cup F^{-1})^{\ZZ}$ (where $\cG$ stands for `geodesic'), with the shift $\sigma : \cG_F\to\cG_F$.
For any representation $\rho : \Gamma\to\GL(d,\RR)$, we denote by $\Phi_{\rho} : \cG_F\to\GL(d,\RR)$ the locally constant cocycle over~$\sigma$ sending $(f_k)_{k\in\ZZ}$ to $\rho(f_0)^{-1}$.
We prove the following.

\begin{thm} \label{thm:charact-domin-splitting-hyp-gp}
Let $\Gamma$ be a word hyperbolic group, let $1\leq i<d$ be integers, and let $\rho : \Gamma\to\GL(d,\RR)$ be a representation. 
Then the cocycle $(\sigma,\Phi_{\rho})$ over~$\cG_F$ has a dominated splitting of index~$i$ if and only if there exists $c>0$ such that for every ergodic $\sigma$-invariant measure $\nu$ on~$\cG_F$ the difference between the $i$-th and~$(i+1)$-th Lyapunov exponents of~$\nu$ is $\geq c$.
\end{thm}

\subsection{Anosov representations for finitely generated semigroups}

The third topic of the paper is \emph{semigroup} representations with a uniform gap in the exponential growth rate of eigenvalues and singular values.

Given the importance of Anosov representations in the recent study of discrete subgroups of Lie groups and higher Teichm\"uller theory, it seems interesting to try to develop a theory of Anosov representations for semigroups.
For instance, the theory of random walks has been quite developed for semigroups (see \eg \cite{BenoistQuint-livre}) and it may happen that having a uniform version allows to understand better certain phenomena in a restricted (yet open) class.

It is not completely clear a priori how to adapt the original definition of Anosov representations from \cite{lab06,gw12} to semigroups.
Instead, having the use of dominated splittings for linear cocycles in mind, we propose the following definition (see Section~\ref{subsec:semig-def-Anosov}).

\begin{defi} \label{def:Anosov-semig}
Let $\semig$ be a finitely generated semigroup and $\rho : \semig\to\GL(d,\RR)$ a semigroup homomorphism, where $d\geq 2$.
For $1\leq i\leq d-1$, we say that $\rho$ is \emph{$P_i$-Anosov} if it has a uniform $i$-gap in singular values.
We say that $\rho$ is \emph{Anosov} if it is $P_i$-Anosov for some $1\leq i\leq d-1$.
\end{defi}

Such semigroup homomorphisms have a discrete image (Remark~\ref{rem:Ano-discrete}), but can be far from injective, and the semigroup $\semig$ does not need to have a Gromov hyperbolic Cayley graph.
Moreover, $P_i$-Anosov does not imply $P_{d-i}$-Anosov in general: see Section~\ref{subsec:semig-vs-gp} for differences with the group case.

Using dominated splittings for linear cocycles naturally associated to semigroup homomorphisms, we construct boundary maps for Anosov semigroup homomorphisms (Section~\ref{subsec:semig-bound-map}) and prove that under some condition, which we call property~(D), the space of $P_i$-Anosov semigroup homomorphisms is an open subset of $\Hom(\semig,\GL(d,\RR))$ (Section~\ref{subsec:semig-dom-splitting}).
Similarly to Proposition~\ref{prop:eig-sing-gap-group}, we show that under some additional condition, which we call property~(U), a semigroup homomorphism is $P_i$-Anosov if and only if it has a uniform $i$-gap in eigenvalues (Section~\ref{subsec:gaps-sv-ev-semig}).

In the case of \emph{completely simple} semigroups, \ie semigroups $\semig$ that have no two-sided ideals other than themselves but possess minimal one-sided ideals, we prove that properties (D) and~(U) are satisfied as soon as Anosov representations of~$\semig$ exist, and that in this case there are strong similarities with the theory of Anosov representations for groups.

\begin{prop} \label{prop:main-Rees-semig}
Let $\semig$ be a finitely generated, completely simple semigroup.
If $\semig$ admits an Anosov representation into some $\GL(d',\RR)$ with $d'\geq 2$, then the Cayley graph of~$\semig$ is Gromov hyperbolic, and any infinite-order element $\gamma\in\semig$ has a unique attracting fixed point $\eta_{\gamma}^+$ in the Gromov boundary $\partial\semig$ of~$\semig$.
In this case, for any integers $1\leq k<d$ and any semigroup homomorphism $\rho : \semig\to\GL(d,\RR)$, the following are equivalent:
\begin{enumerate}
  \item\label{item:Pk-Ano-Rees} $\rho$ is $P_k$-Anosov;
  \item\label{item:Pd-k-Ano-Rees} $\rho$ is $P_{d-k}$-Anosov;
  \item\label{item:bound-maps-Rees} there exist continuous $\rho$-equivariant dynamics-preserving boundary maps\linebreak $\xi : \partial\semig\to\mathrm{Gr}_k(\RR^d)$ and $\xi' : \partial\semig\to\mathrm{Gr}_{d-k}(\RR^d)$ which are compatible and transverse, and $\mu_k(\rho(\gamma_n)) - \mu_{k+1}(\rho(\gamma_n)) \to +\infty$ for any sequence $(\gamma_n)_{n\in\NN}$ of pairwise distinct elements of~$\semig$;
  \item\label{item:eig-gap-Rees} $\rho$ has a uniform $k$-gap in eigenvalues.
\end{enumerate}
Moreover, $P_k$-Anosov representations form an open subset of $\Hom(\semig,\GL(d,\RR))$.
\end{prop}

In~\eqref{item:bound-maps-Rees}, following \cite{GGKW}, we say that $\xi$ (\resp $\xi'$) is \emph{dynamics-preserving} if for any infinite-order element $\gamma\in\semig$, the image by $\xi$ (\resp $\xi'$) of the attracting fixed point $\eta_{\gamma}^+$ of $\gamma$ in $\partial\semig$ is an attracting fixed point of $\rho(\gamma)$ in $\mathrm{Gr}_k(\RR^d)$ (\resp $\mathrm{Gr}_{d-k}(\RR^d)$).
We say that $\xi : \partial\semig\to\mathrm{Gr}_k(\RR^d)$ and $\xi' : \partial\semig\to\mathrm{Gr}_{d-k}(\RR^d)$ are \emph{compatible} if for any $\eta\in\partial\semig$, the $k$-plane $\xi(\eta)$ and the $(d-k)$-plane $\xi'(\eta)$ intersect in a plane of dimension $\min(k,d-k)$.
We say that they are \emph{transverse} if for any $\eta,\eta'\in\partial\semig$ such that $\gamma\cdot\eta \neq \gamma\cdot\eta'$ for some $\gamma\in\semig$, the $k$-plane $\xi(\eta)$ and the $(d-k)$-plane $\xi'(\eta')$ intersect trivially.

Fountain--Kambites \cite{fk04} proved that for completely simple semigroups~$\semig$, the Gromov hyperbolicity of the Cayley graph of~$\semig$ is equivalent to a notion of hyperbolicity for~$\semig$ introduced by Gilman \cite{gil02} (see also \cite{dg04}) in language-theoretic terms.

\subsection{Organization of the paper}

In Section~\ref{sec:prelim} we recall basic definitions and useful facts on linear cocycles and their dominated splittings, and on limit cones of singular values and eigenvalues.
In Section~\ref{sec:proof-teoA} we prove Theorem~\ref{thm:charact-domin-splitting} and give an example showing that the uniformity of the Lyapunov exponent gap for periodic orbits is necessary.
In Section~\ref{sec:gaps-in-groups} we establish Proposition~\ref{prop:eig-sing-gap-group} and Theorem~\ref{thm:charact-domin-splitting-hyp-gp}.
In Section~\ref{sec:Ano-semig} we treat the case of finitely generated semigroups, and in particular prove Proposition~\ref{prop:main-Rees-semig}.
Finally, in Section~\ref{sec:general-reductive} we extend some of our results from $\GL(d,\RR)$ to any noncompact reductive linear Lie group~$G$.

\section{Preliminaries} \label{sec:prelim}

In the whole paper, we fix an integer $d\geq 1$ and denote by $\Vert\cdot\Vert$ the standard Euclidean norm on~$\RR^d$.
For a matrix $g \in \GL(d,\RR)$ we denote by $\lambda_1(g) \geq \ldots \geq \lambda_d(g)$ the logarithms of the moduli of the (complex) eigenvalues of~$g$, and by $\mu_1(g) \geq \ldots \geq \mu_d(g)$ the logarithms of the singular values of~$g$, \ie half the logarithms of the eigenvalues of $gg^t$.
This defines continuous maps
\begin{equation} \label{eqn:lambda-mu-def}
\left\{ \begin{array}{l}
\lambda = (\lambda_1,\dots,\lambda_d)\ :\ \GL(d,\RR) \longrightarrow \RR^d,\\
\mu = (\mu_1,\dots,\mu_d)\ :\ \GL(d,\RR) \longrightarrow \RR^d,
\end{array} \right.
\end{equation}
respectively called \emph{Lyapunov} (or \emph{Jordan}) \emph{projection} and \emph{Cartan projection}.

\begin{remark} \label{rem:mu1-norm}
For any $g\in\GL(d,\RR)$, the operator norm $\Vert g\Vert$ is equal to $e^{\mu_1(g)}$.
In particular, $\mu_1(gg')\leq\mu_1(g)+\mu_1(g')$ for all $g,g'\in\GL(d,\RR)$.
\end{remark}

\subsection{Subshifts} \label{subsec:subshifts}

Fix an integer $N>0$.
We denote by $\Sigma$ the \emph{full shift space} on $N$ elements, that is, $\Sigma = \{ 1, \ldots, N \}^\ZZ$.
It is compact for the product topology.
The \emph{shift} on $\Sigma$ is the continuous map $\sigma : \Sigma \to \Sigma$ given by shifting the sequence one position to the left: namely, $\sigma(\sx) = \sy$ where $\sx=(x_k)_{k\in\ZZ} \in \Sigma$ and $\sy=(y_k)_{k\in\ZZ} \in \Sigma$ satisfy $y_k=x_{k+1}$ for all $k\in\ZZ$.

A $\sigma$-invariant closed subset of~$\Sigma$, endowed with the restriction of~$\sigma$, is called a \emph{subshift}.
We shall be mainly interested in subshifts of \emph{finite type}, \ie of the form
$$\Sigma_A = \big\{ \sx=(x_k)_{k\in\ZZ} \in \Sigma \ : \  a_{x_k,x_{k+1}} = 1 \quad \forall k\in\ZZ \big\} $$
where $A =(a_{i,j})_{1\leq i,j\leq N}$ is an $N\times N$ matrix whose entries are $0$'s and $1$'s. These subshifts are sometimes also called \emph{Markov shifts}. 

In this paper, as often in the literature, we shall make a certain irreducibility assumption on the matrix $A$ by imposing the existence of $n_0>0$ such that for every $n\geq n_0$ the matrix $A^n$ has all its entries positive.
Typically, one reduces any subshift of finite type to such irreducible ones by decomposing $A$ into blocks: see \cite{LindMarcus}.

Other interesting subshifts are \emph{sofic subshifts}, which include in particular the subshifts $\cG_F$ associated with word hyperbolic groups, as in Section~\ref{subsec:intro-hyp-gps}.
We refer to \cite[\S\,5]{BPS} for the definition and more details.
Similarly to subshifts of finite type, one important feature of sofic subshifts is to admit plenty of periodic orbits, which in some ways govern their dynamics through a \emph{specification property} (Fact~\ref{fact-specification}).

\subsection{Linear cocycles} \label{subsec:cocycles}

Let $X$ be a compact metric space.
A \emph{linear cocycle} over~$X$ is a pair $(T,\Phi)$ where $T : X \to X$ and $\Phi: X \to \GL(d,\RR)$ are continuous maps; sometimes, by a little abuse of notation, we shall also say that $\Phi$ is a linear cocycle over~$T$.
For any $n\in\NN^*$ we define $\Phi^{(n)} : X \to \GL(d,\RR)$ by
\begin{equation} \label{eqn:Phi-(n)}
\Phi^{(n)}(\sx) := \Phi(T^{n-1}(\sx)) \cdots \Phi(\sx)
\end{equation}
for all $\sx\in X$.
We also set $\Phi^{(0)}(\sx) = \mathrm{id} \in \GL(d,\RR)$ and, when $T$ is invertible,
$$\Phi^{(-n)}(\sx) := (\Phi^{(n)}(T^{-n}(\sx)))^{-1} $$
for all $n\in\NN^*$.
We then have the \emph{cocycle relation}
$$\Phi^{(n+m)}(\sx) = \Phi^{(m)}(T^n(\sx)) \, \Phi^{(n)}(\sx)$$
for all $\sx\in X$ and $n,m\in\NN$ (and for all $\sx\in X$ and $n,m \in \ZZ$ when $T$ is invertible).
We can define a skew-product map $\mathcal{F} = (T,\Phi) : X \times \RR^d \to X \times \RR^d$ by
$$\mathcal{F}(\sx,v) := (T(\sx), \Phi(\sx) v) $$
for all $\sx\in X$ and $v\in\RR^d$.
It is easy to check that $\mathcal{F}^n(\sx,v)= (T^n(\sx), \Phi^{(n)}(\sx) v)$.

One can similarly define cocycles $(T,\Phi)$ where $\Phi$ takes values in any linear Lie group~$G$.
We refer to \cite{Viana} for a broad presentation of linear cocycles.

\subsection{Lyapunov exponents}\label{ss.lyapunovexponents} 

Let $X$ be a compact metric space and $(T,\Phi)$ a linear cocycle over~$X$, where $\Phi : X\to\GL(d,\RR)$.
By Oseledets's theorem (see \eg \cite{Viana}), for any ergodic $T$-invariant measure $\nu$ on~$X$, there exist real numbers $\hat \chi_1(\nu) > \ldots > \hat \chi_k(\nu)$ (with $k \leq d$), called the \emph{Lyapunov exponents} of~$\nu$ with respect to $(T,\Phi)$, with the property that for $\nu$-almost every $\sx\in X$ there exists a (possibly incomplete) flag $\{0\} \subsetneq E_{k}(\sx) \subsetneq \ldots \subsetneq E_1(\sx) = \RR^d $ which is $(T,\Phi)$-equivariant (\ie $\Phi(\sx) E_{\ell}(\sx) = E_{\ell}(T(\sx))$ for all $\sx\in X$ and $1\leq\ell\leq k$) and such that for any $v \in E_{\ell}(\sx) \smallsetminus E_{\ell+1}(\sx)$,
$$ \lim_{n\to +\infty} \frac{1}{n} \log \|\Phi^{(n)}(\sx) v\| = \hat \chi_{\ell}(\nu). $$
Ergodicity implies that the dimension of $E_{\ell}(\sx)$ is constant $\nu$-almost everywhere.
The integer $\dim E_{\ell} - \dim E_{\ell+1} \geq 1$ is called the \emph{multiplicity} of $\hat \chi_{\ell}(\nu)$.
Equivalently, we can count the Lyapunov exponents with multiplicity and define
$$\chi_1(\nu) \geq \chi_2(\nu) \geq \ldots \geq \chi_d(\nu)$$
by $\chi_i(\nu) = \hat \chi_{\ell}(\nu)$ for all $i\in\{\dim E_1-\dim E_{\ell}+1,\dots,\dim E_1-\dim E_{\ell+1}\}$.
We shall use the following terminology.

\begin{defi} \label{def:unif-gap-Lyap-exp}
A linear cocycle $(T,\Phi)$ over~$X$ has a \emph{uniform $i$-gap of Lyapunov exponents} if there exists $c>0$ such that the Lyapunov exponents with respect to $(T,\Phi)$ satisfy $\chi_i(\nu) \geq \chi_{i+1}(\nu) + c$ for all ergodic $T$-invariant measures $\nu$ on~$X$.
\end{defi}

\subsection{Dominated splittings} \label{subsec:prelim-domin-splitting}

Let $X$ be a compact metric space and $(T,\Phi)$ a linear cocycle over~$X$, where $\Phi : X\to\GL(d,\RR)$.
For $1\leq i\leq d-1$, the cocycle $(T,\Phi)$ is said to have a \emph{dominated splitting of index~$i$} (or to be $i$-\emph{dominated}) if there exist a continuous $(T,\Phi)$-equivariant map $E^{cs}: X \to \mathrm{Gr}_{d-i}(\RR^d)$ into the Grassmanian of $(d-i)$-planes of~$\R^d$ and constants $C,C'>0$ such that for any $n\in\NN$, any $\sx\in X$, and any unit vectors $v \in E^{cs}(\sx)$ and $w \in (E^{cs}(\sx))^{\perp}$,
\begin{equation}\label{eqn:domination}
 \log \|\Phi^{(n)}(\sx) w \| - \log \|\Phi^{(n)}(\sx) v \|  \geq Cn - C'. 
 \end{equation}
The map $E^{cs}$ is necessarily unique (see \eg \cite[Prop.\,2.2]{CP}). 

\begin{remark} \label{rem-Tinvertible}
When $T$ is invertible, this implies the existence of a continuous $(T,\Phi)$-equivariant map $E^{cu}: X \to \mathrm{Gr}_i(\RR^d)$ such that  $E^{cu}(\sx) \oplus E^{cs}(\sx)=\RR^d$ for all $\sx\in X$.
\end{remark}

An important property of dominated splittings is given by the following fact, which relies on the cone-field criterion (see \eg \cite[\S\,2.2]{CP} or \cite[Th.\,5.1]{BPS}): 

\begin{fact} \label{fact-dsopen}
Let $(T,\Phi)$ be a linear cocycle over~$X$ with a dominated splitting of index~$i$.
Then there exists a neighborhood $\cU$ of $\Phi$ in $C^0(X,\GL(d,\RR))$ (for the compact open topology) such that $(T,\hat\Phi)$ has a dominated splitting of index~$i$ for all $\hat \Phi \in \cU$. 
\end{fact}

The following is well-known; we give a proof for the reader's convenience.

\begin{fact} \label{fact:domin->Lyap-gap}
Suppose that a linear cocycle $(T, \Phi)$ over~$X$ has a dominated splitting of index~$i$.
Then it has a uniform $i$-gap of Lyapunov exponents (Definition~\ref{def:unif-gap-Lyap-exp}).
\end{fact}

\begin{proof}
According to the decomposition $\RR^d = (E^{cs})^{\perp} \oplus E^{cs}$, we can write $\Phi$ as
$$ \Phi(\sx) =\left( \begin{matrix}  A_{\perp}(\sx) & 0 \\ \ast & A_{cs}(\sx) \end{matrix} \right) $$
for all $\sx\in X$, where $(T,A_{cs})$ and $(T,A_{\perp})$ are cocycles.
Let $\nu$ be an ergodic $T$-invariant measure on~$X$.
The Oseledets theorem gives Lyapunov exponents for each of the cocycles $(T,A_{cs})$ and $(T,A_{\perp})$.
The domination condition \eqref{eqn:domination} implies that the Lyapunov exponents of $A_{\perp}$ are uniformly larger than those of $A_{cs}$. 
It follows that $E^{cs}$ is one of the bundles in the Oseledets flag for $\nu$ and that vectors not in $E^{cs}$ have an exponential growth uniformly larger than those in $E^{cs}$, which proves the fact.
\end{proof}

With some further work one can show that some form of converse holds (see \cite[\S\,2.6]{CP}): if $(T,\Phi)$ has a uniform $i$-gap of Lyapunov exponents and if there exists a \emph{continuous} $(T,\Phi)$-equivariant bundle $E^{cs}: X \to \mathrm{Gr}_{d-i}(\RR^d)$ such that for any ergodic $T$-invariant measure $\nu$ on~$X$ the bundle $E^{cs}(\sx)$ coincides for $\nu$-almost every $\sx\in X$ with one of the bundles of the Oseledets flag, then $(T,\Phi)$ has a dominated splitting of index~$i$.

It is natural to ask the following question: 

\begin{quest} \label{q.gapcocycle}
For a linear cocycle $(T,\Phi)$ over~$X$, does the existence of a uniform $i$-gap of Lyapunov exponents imply the existence of a dominated splitting of index~$i$?
\end{quest} 

When there are few $T$-invariant measures on~$X$, one cannot in general expect a positive answer, see for instance \cite{Herman} (and also \cite{Avila-Bochi-Trieste,Viana}).
The question is more natural for transformations $T : X\to X$ with many invariant measures.
Subshifts of finite type are a prototype of those.

Theorem~\ref{thm:charact-domin-splitting} gives a positive answer to Question~\ref{q.gapcocycle} assuming that $T$ is a subshift of finite type and $\Phi$ is \emph{locally constant} (\ie $\Phi$ depends only on the $0$-coordinate $x_0$ of $\sx =(x_k)_{k\in\ZZ} \in X$).
On the other hand, for cocycles that are not locally constant, even in dimension~2, the answer may be negative: see \cite{Gogolev, Velozo}.
Let us mention here the interesting recent work of Park \cite{Park} which obtains a quasi-multiplicative property (related to Fact~\ref{fact:Benoist} below) for \emph{typical} fibered bunched cocycles.
The approach has its roots in previous work of Feng who established a property similar to Fact~\ref{fact:Benoist} for certain cocycles (see \cite{Feng, BochiMorris}).
We also mention \cite[\S\,3.3]{Butler}, where the existence of a dominated splitting is obtained under some conditions on the Lyapunov spectra of a cocycle over an Anosov flow.

\subsection{A criterion for domination}

For $1\leq j\leq d-1$ and $g\in\GL(d,\RR)$ with $\mu_{j}(g)>\mu_{j+1}(g)$, we denote by $\Xi_{j}(g)$ the sum of the eigenspaces of $gg^t$ corresponding to the eigenvalues $e^{2\mu_1(g)},\dots,e^{2\mu_{j}(g)}$.

The following useful criterion for domination was introduced by Bochi--Gourmelon \cite{BG}, based on a criterion in dimension 2 due to Yoccoz \cite{Yoccoz}.

\begin{fact}[Bochi--Gourmelon \cite{BG}] \label{fact-BG}
Given a compact metric space $X$ and an integer $1\leq i\leq d-1$, the property for a linear cocycle $(T,\Phi)$ over~$X$ with $\Phi: X \to \GL(d,\RR)$ to have a dominated splitting of index~$i$ is equivalent to the existence of $C,C'>0$ such that for any $n\in\NN$ and any $\sx\in X$,
\begin{equation} \label{eq-BG}
(\mu_i - \mu_{i+1})(\Phi^{(n)}(\sx)) \geq Cn - C'.
\end{equation}
In this case the continuous $(T,\Phi)$-equivariant bundle $E^{cs}: X \to \mathrm{Gr}_{d-i}(\RR^d)$ is given, for all $\sx\in X$, by
\begin{equation} \label{eqn:limitbundle}
E^{cs}(\sx) = \lim_{n\to +\infty} \Xi_{d-i}\big((\Phi^{(n)}(\sx))^{-1}\big).
\end{equation}
\end{fact}

\subsection{Further properties of dominated splittings}

In Section~\ref{sec:Ano-semig} we shall use the following observation, where $\SL^{\pm}(d,\RR)$ denotes the group of elements of $\GL(d,\RR)$ with determinant $\pm 1$. 

\begin{lema} \label{lem:domin-CLI}
Let $X$ be a compact metric space and $1\leq i\leq d-1$ an integer.
If a linear cocycle $(T,\Phi)$ over~$X$ with $\Phi: X \to \SL^{\pm}(d,\RR)$ has a dominated splitting of index~$i$, then
\begin{enumerate}
  \item\label{item:domin-angle} there exist $\alpha_0>0$ and $n_0>0$ such that for any $m,n\geq n_0$ and $\sx\in X$, the angle between the subspaces $\Xi_{d-i}\big((\Phi^{(m)}(T^n(\sx)))^{-1}\big)$ and $\Xi_i(\Phi^{(n)}(\sx))$ of~$\RR^d$, measured in $[0,\pi/2]$, is larger than $\alpha_0$;
  \item\label{item:domin-CLI} there exist $C'',C'''>0$ such that for any $n,m\in\NN$ and any $\sx\in X$,
    $$\big| (\mu_1+\dots+\mu_i)\big(\Phi^{(n+m)}(\sx)\big) - (\mu_1+\dots+\mu_i)\big(\Phi^{(n)}(\sx)\big) \big| \geq C''m - C''' ;$$
  in particular,
  $$\sqrt{d} \, \big\Vert\mu\big(\Phi^{(n+m)}(\sx)\big) - \mu\big(\Phi^{(n)}(\sx)\big)\big\Vert \geq C''m - C'''.$$
\end{enumerate}
\end{lema}

When $T : X\to X$ is a homeomorphism, Lemma~\ref{lem:domin-CLI}.\eqref{item:domin-angle} follows from \cite{BG}; we now explain how to reduce to this case by a standard ``inverse limit'' trick.

\begin{proof}[Proof of Lemma~\ref{lem:domin-CLI}.\eqref{item:domin-angle}]
Consider the compact metric space
$$ \hat X := \big\{ (\sx_k)_{k\in\ZZ} \in X^\ZZ \ : \ T(\sx_k)=\sx_{k+1}\ \forall k\in\ZZ \big\}, $$
with the shift $\hat T: \hat X \to \hat X$ sending $(\sx_k)_{k\in\ZZ}$ to $(\sy_k)_{k\in\ZZ}$ with $\sy_k = \sx_{k+1}$, which is a homeomorphism.
There is a natural projection $\pi : \hat X \to X$ given by $\pi((\sx_k)_{k\in\ZZ})=\nolinebreak \sx_0$.
If $T$ is invertible, then $\pi$ is a homeomorphism, conjugating $\hat T$ to~$T$.
In general, consider the cocycle $(\hat T,\hat \Phi)$ over~$\hat X$, where $\hat\Phi := \Phi\circ\pi : \hat X \to \GL(d,\RR)$.

One readily checks that if \eqref{eq-BG} holds for all $n\in\NN$ and $\sx\in X$, then the analogous inequality also holds for $(\hat T,\hat \Phi)$ for all $n\in\NN$ and $\hat \sx\in \hat X$.
Therefore, by Fact~\ref{fact-BG}, if $(T,\Phi)$ admits a dominated splitting, then so does $(\hat T,\hat \Phi)$, and the continuous $(\hat T, \hat \Phi)$-equivariant bundle $\hat E^{cs}: \hat X \to \mathrm{Gr}_{d-i}(\RR^d)$ is given, for all $\hat \sx\in\hat X$, by
$$\hat E^{cs}(\hat \sx) = \lim_{n\to +\infty} \Xi_{d-i}\big(\big(\hat\Phi^{(n)}(\hat \sx)\big)^{-1}\big).$$
By Remark~\ref{rem-Tinvertible}, there is a continuous $(\hat T,\hat\Phi)$-equivariant map $\hat E^{cu}: \hat X \to \mathrm{Gr}_i(\RR^d)$ such that $\hat E^{cu}(\hat \sx) \oplus \hat E^{cs}(\hat \sx)=\RR^d$ for all $\hat \sx \in \hat X$, and by applying \eqref{eqn:limitbundle} to the inverse as in \cite{BG} we see that
$$\hat E^{cu}(\hat \sx) = \lim_{n\to +\infty} \Xi_i\big(\hat \Phi^{(n)}(\hat T^{-n}(\hat \sx))\big)$$
for all $\hat \sx \in \hat X$.
Since $\hat X$ is compact, by continuity there exists $\alpha_0>0$ such that for any $\hat \sx\in\hat X$ the angle between $\hat E^{cs}(\hat \sx)$ and $\hat E^{cu}(\hat \sx)$ is larger than $2\alpha_0$, and further by uniformity of the limit there exists $n_0>0$ such that for any $m,n\geq n_0$ and $\hat \sx\in\hat X$, the angle between $\Xi_{d-i}((\hat \Phi^{(m)}(\hat \sx))^{-1})$ and $\Xi_i(\hat \Phi^{(n)}(\hat T^{-n}(\hat \sx)))$ is larger than $\alpha_0$.
Then for any $m,n\geq n_0$ and $\sx\in X$, by choosing $\hat \sx$ such that $\pi(\hat T^{-n}(\hat \sx)) = \sx$, we see that the angle between $\Xi_{d-i}\big((\Phi^{(m)}(T^n(\sx)))^{-1}\big)$ and $\Xi_i(\Phi^{(n)}(\sx))$ of~$\RR^d$ is larger than~$\alpha_0$.
\end{proof}

\begin{proof}[Proof of Lemma~\ref{lem:domin-CLI}.\eqref{item:domin-CLI}]
We may and shall assume that $i=1$.
Indeed, let\linebreak $\tau_{i} : \GL(d,\RR)\to\GL(\Lambda^{i}\RR^d)$ be the natural representation of $\GL(d,\RR)$ on the $i$-th exterior power $\Lambda^{i}\RR^d$.
If $(T,\Phi)$ has a dominated splitting of index~$i$, then $(T,\tau_{i}\circ\Phi)$ has a dominated splitting of index~$1$, and for any $g\in G$ the logarithm of the first singular value of $\tau_{i}(g)$ is $(\mu_{1}+\dots +\mu_i)(g)$.

We claim that there exists $C_0>0$ such that for any $n,m\in\NN$ and $\sx\in X$,
\begin{equation} \label{eqn:submultmu1}
 \mu_1\big(\Phi^{(n+m)}(\sx)\big) - \mu_1\big(\Phi^{(n)}(\sx)\big) \geq \mu_1\big(\Phi^{(m)}(T^n(\sx))\big) - C_0.
\end{equation} 
Indeed, by Lemma~\ref{lem:domin-CLI}.\eqref{item:domin-angle}, there exist $n_0,\alpha_0>0$ such that for any $n,m \geq n_0$ and $\sx\in X$ the angle between $\Xi_{d-1}((\Phi^{(m)}(T^n(\sx)))^{-1})$ and $\Xi_{1}(\Phi^{(n)}(\sx))$ is larger than $\alpha_0$.
In general, for any $g,h\in\GL(d,\RR)$, we have $\mu_1(gh) \geq \mu_1(g) + \mu_1(h) + \log \sin \alpha$ where $\alpha\in [0,\pi/2]$ is the angle between $\Xi_{d-1}(g^{-1})$ and $\Xi_1(h)$ (see \eg \cite[Lem.\,A.7]{BPS}).
Applying this to $g = \Phi^{(m)}(T^n(\sx))$ and $h = \Phi^{(n)}(\sx)$, so that $gh = \Phi^{(n+m)}(\sx)$ by the cocycle relation, we obtain \eqref{eqn:submultmu1} with $C_0 := |\log \sin \alpha_0|$ for all $n,m \geq n_0$ and $\sx\in X$.
On the other hand, using the fact that $\mu_1(gh) \geq \mu_1(g) - \mu_1(h^{-1})$ and $\mu_1(gh) \geq \mu_1(h) - \mu_1(g^{-1})$ for all $g,h\in\GL(d,\RR)$ (see Remark~\ref{rem:mu1-norm}), we obtain \eqref{eqn:submultmu1} with $C_0 := \max_{0\leq j\leq n_0} \max_{\sy\in X} (\mu_1(\Phi^{(j)}(\sy)) + \mu_1(\Phi^{(j)}(\sy)^{-1}))$ for all $n,m\in\NN$ with $\min(n,m)\leq n_0$ and all $\sx\in X$.

By Fact~\ref{fact-BG}, there exist $C,C'>0$ such that for any $m\in\NN$ and $\sy\in X$, we have $\mu_{1}(\Phi^{(m)}(\sy)) - \mu_2(\Phi^{(m)}(\sy)) \geq Cm - C'$.
Since $\Phi$ takes values in $\SL^{\pm}(d,\RR)$, we have $(\mu_1+\dots+\mu_d)(\Phi^{(m)}(\sy))=0$ for all $m\in\NN$ and $\sy\in X$, hence
\[ \mu_1(\Phi^{(m)}(\sy)) = \frac{1}{d} \, \sum_{i=2}^d (\mu_1-\mu_i)(\Phi^{(m)}(\sy))  \geq \frac{d-1}{d} \, (\mu_{1}-\mu_2)(\Phi^{(m)}(\sy)) \geq  \frac{d-1}{d}(Cm - C'). \]
This, together with \eqref{eqn:submultmu1}, completes the proof by taking $\sy=T^n(\sx)$ and setting $C'' = \frac{d-1}{d}C$ and $C''' = \frac{d-1}{d}C' + C_0$.
\end{proof}

\subsection{Periodic orbits}\label{ss.periodicorbits}

For certain dynamical systems, invariant measures can be approximated by measures associated to periodic orbits, as we now discuss.

Let $X$ be a compact metric space and $(T,\Phi)$ a cocycle over~$X$, where $\Phi : X\to\GL(d,\RR)$.
If a point $\sx\in X$ is $T$-periodic, \ie $T^n(\sx) = \sx$ for some $n\in\NN^*$, then we can view the matrix $\Phi^{(n)}(\sx)$ as a map from $\{\sx\} \times \RR^d$ to itself which is linear in~$\RR^d$ (recall the skew product $\mathcal{F}$ from Section~\ref{subsec:cocycles}), and so it makes sense to consider the eigenvalues of $\Phi^{(n)}(\sx)$.
One can expect that their information is relevant due to the following observation: for $n$ a positive multiple of the period of~$\sx$, the Lyapunov exponents of the ergodic invariant measure
\begin{equation} \label{eqn:nu-x}
\nu_{\sx} := \frac{1}{n}(\delta_{\sx} + \ldots + \delta_{T^{n-1}(\sx)})
\end{equation}
with respect to $(T,\Phi)$ are the logarithms of the moduli of the eigenvalues of $\Phi^{(n)}(\sx)$ divided by~$n$, namely
\begin{equation} \label{eqn:Lyap-exp-periodic-meas}
\big(\chi_1(\nu_{\sx}),\dots,\chi_d(\nu_{\sx})\big) = \frac{1}{n} \, \lambda\big(\Phi^{(n)}(\sx)\big).
\end{equation}
This has the following direct consequence.

\begin{fact} \label{fact:Lyap-gap-periodic-meas}
Suppose the linear cocycle $(T,\Phi)$ has a uniform $i$-gap of Lyapunov exponents (Definition~\ref{def:unif-gap-Lyap-exp}), with constant $c>0$.
Then for any $T$-periodic point $\sx\in X$ and any positive multiple $n$ of the period of~$\sx$, we have $(\lambda_i - \lambda_{i+1})(\Phi^{(n)}(\sx)) \geq cn$.
\end{fact}

Remarkably, in certain situations, it is possible to approach the Lyapunov exponents of all invariant measures by the Lyapunov exponents of periodic measures \cite{Kalinin}, giving a converse to Fact \ref{fact:Lyap-gap-periodic-meas}.
This is the case for locally constant cocycles over subshifts of finite type, using the following \emph{specification} property (see \eg \cite[Prop.\,2.2.12]{LindMarcus} or \cite[Ex.\,18.3.5]{KH}): 

\begin{fact} \label{fact-specification}
Let $\Sigma_A\subset\{1,\dots,N\}^{\ZZ}$ be a subshift of finite type, where $A=(a_{i,j})_{1\leq i,j\leq N}$ is a matrix with entries $0$'s and $1$'s, satisfying the irreducibility assumption of Section \ref{subsec:subshifts}.
Then there exists $n_0>0$ such that for any $k\in\NN$, any $k$-tuple $(x_0, \ldots, x_{k-1}) \in \{1,\ldots, N \}^k$ allowed by~$A$ (\ie such that $a_{x_{\ell},x_{\ell+1}}=1$ for all $0\leq\ell\leq k-2$), and any $j \in \{1, \ldots, N\}$, there exists a periodic point $\sy \in \Sigma_A$ of period exactly $k+2n_0$ with $y_0 = j$ and $y_{n_0+\ell}=x_{\ell}$ for all $0\leq\ell\leq k-1$.
\end{fact}

Thanks to Fact~\ref{fact-specification}, we will be able to prove a version of Theorem~\ref{thm:charact-domin-splitting} involving only Lyapunov exponents for periodic orbits: see Theorem~\ref{thm:sufficient-domin-splitting} below.

The specification property of Fact~\ref{fact-specification} also holds for more general subshifts, such as sofic subshifts (see \eg~\cite[\S\,6.2]{Weiss}).
In general, even in such settings, the existence of a uniform gap of Lyapunov exponents for all ergodic $T$-invariant measures does not imply the existence of a dominated splitting: see \eg \cite{Gogolev,Velozo}.

\subsection{Limit cones of singular values and eigenvalues}\label{subsec:limit-cones}

With the notation \eqref{eqn:lambda-mu-def}, for any $g\in\GL(d,\RR)$ we have
\begin{equation} \label{eqn:lambda-mu-relation}
\|\lambda(g)\| \leq \|\mu(g)\| \quad\quad\text{and}\quad\quad \lambda(g) = \lim_{n\to +\infty} \frac{1}{n} \, \mu(g^n),
\end{equation}
where $\Vert\cdot\Vert$ denotes the standard Euclidean norm on~$\RR^d$.
Indeed, if $\mathbb{X} = \GL(d,\RR)/\mathrm{O}(d)$ is the symmetric space of $\GL(d,\RR)$ with its Riemannian metric $d_{\mathbb{X}}$, and if $x_0 \in \mathbb{X}$ is the point with stabilizer $\mathrm{O}(d)$, then $\Vert\lambda(g)\Vert = \inf_{x\in\mathbb{X}} d_{\mathbb{X}}(x,g\cdot x)$ and $\Vert\mu(g)\Vert = d_{\mathbb{X}}(x_0,g\cdot x_0)$ for all $g \in \GL(d,\RR)$.

Moreover (see \eg \cite[Lem.\,2.3]{kas08}), for any $g,g_1,g_2 \in \GL(d,\RR)$,
\begin{equation} \label{eqn:mu-subadd}
\| \mu(g_1 g g_2) - \mu(g) \| \leq \| \mu(g_1) \| + \| \mu(g_2) \|.
\end{equation}
Note also that for any $g\in\GL(d,\RR)$ and $1\leq i\leq d-1$,
\begin{equation} \label{eqn:mu-lambda-inverse}
\mu_i(g) = \mu_{d-i}(g^{-1}) \quad\quad\text{and}\quad\quad \lambda_i(g) = \lambda_{d-i}(g^{-1}).
\end{equation}
Both $\mu(g)$ and $\lambda(g)$ are elements of the closed Weyl chamber
$$\mathfrak{a}^+ = \{ x\in\RR^d\ :\ x_1\geq\dots\geq x_d\}.$$

Benoist \cite{Benoist97} (see also \cite{Benoist-notes}) associated to any group or semigroup $\Gamma \subset \GL(d,\RR)$ two \emph{limit cones} in~$\mathfrak{a}^+$, namely:
\begin{itemize}
  \item the cone $\cC_{\lambda}(\Gamma)$ spanned by the $\lambda(\gamma)$ for $\gamma \in \Gamma$;
  \item the cone $\cC_{\mu}(\Gamma)$ spanned by all possible limits of sequences $(\mu(\gamma_n)/\|\mu(\gamma_n)\|)_{n\in\NN}$ for $(\gamma_n)\in\Gamma^{\NN}$ with $\Vert\mu(\gamma_n)\Vert \to +\infty$.
\end{itemize}
The inclusion $\cC_{\lambda}(\Gamma) \subset \cC_{\mu}(\Gamma)$ always holds, by \eqref{eqn:lambda-mu-relation}.

The following key fact is due to Benoist \cite{Benoist97}, using a result of Abels--Margulis--Soifer \cite[Th.\,4.1]{AMS}; see \eg \cite[Th.\,4.12]{GGKW} for an explicit statement and proof.\footnote{In the statement and proof of \cite[Th.\,4.12]{GGKW}, the word ``group'' can be replaced everywhere by ``semigroup'' as the results of \cite{AMS,Benoist97} hold for semigroups.}
Recall (see \eg \cite[Lem.\,5.15]{BenoistQuint-livre}) that the Zariski closure of a semigroup is a group.
A linear real algebraic group $G$ is said to be \emph{reductive} if the unipotent radical (\ie the largest connected unipotent normal algebraic subgroup) of the identity component of~$G$ (for the Zariski topology) is trivial.
An important subclass is that of \emph{semisimple} algebraic groups (see \cite{BorelTits,Knapp} for more background).

\begin{fact}[Benoist \cite{Benoist97}] \label{fact:Benoist}
Let $\Gamma \subset \GL(d,\RR)$ be a semigroup whose Zariski closure in $\GL(d,\RR)$ is reductive.
Then there exist a finite subset $S$ of $\Gamma$ and a constant $M>0$ such that for any $\gamma\in\Gamma$,
$$\min_{s\in S} \| \mu(\gamma) - \lambda(s\gamma) \| \leq M.$$
In particular, using \eqref{eqn:lambda-mu-relation}, the limit cones $\cC_{\lambda}(\Gamma)$ and $\cC_{\mu}(\Gamma)$ coincide in this setting.
\end{fact} 

An important property from \cite{Benoist97} is that if $\Gamma$ is Zariski-dense in $\GL(d,\RR)$, then the limit cone $\cC_{\lambda}(\Gamma)=\cC_{\mu}(\Gamma)$ is convex with nonempty interior; we will not use it here.

\subsection{Semisimplification} \label{ss.semisimp}

Let $H$ be an algebraic subgroup of $\GL(d,\RR)$.
Choose a Levi decomposition $H=L\ltimes R_u(H)$, where $L$ is a reductive algebraic group and $R_u(H)$ is the unipotent radical of~$H$ (see \cite{mos56}).
The natural projection $\pi^{ss} : H \to L$ is a group homomorphism which does not depend, up to conjugation by $R_u(H)$, on the choice of the Levi factor~$L$.
There is a sequence $(g_k) \in \GL(d,\RR)^{\NN}$ such that $g_k h g_k^{-1} \to \pi^{ss}(h)$ for all $h\in H$.
Since the map $\lambda : \GL(d,\RR)\to\RR^d$ of \eqref{eqn:lambda-mu-def} is invariant under conjugation and continuous, we have, for any $h\in H$,
\begin{equation} \label{eqn:lambda-pi-ss}
\lambda(h) = \lambda(\pi^{ss}(h)).
\end{equation}

As in \cite[\S\,2.5.4]{GGKW}, if $\Gamma$ is a semigroup and $\rho : \Gamma\to\GL(d,\RR)$ a semigroup homomorphism whose image $\rho(\Gamma)$ has Zariski closure $H$ in $\GL(d,\RR)$, we define the \emph{semisimplification} of~$\rho$ to be the semigroup homomorphism $\rho^{ss} := \pi^{ss} \circ \rho : \Gamma\to L$.
It is again uniquely defined up to conjugation by $R_u(H)$.
The Zariski closure of $\rho^{ss}(\Gamma)$ is~$L$.

By \eqref{eqn:lambda-pi-ss} we have $\cC_\lambda(\rho(\Gamma)) = \cC_{\lambda}(\rho^{ss}(\Gamma))$.
On the other hand, $\cC_\mu(\rho(\Gamma))$ could be larger than $\cC_{\mu}(\rho^{ss}(\Gamma))$ (\eg take $\rho(\Gamma)$ unipotent).

\section{Lyapunov exponents and dominated splittings} \label{sec:proof-teoA}

In this section we give a proof of Theorem~\ref{thm:charact-domin-splitting}.
The direct implication of Theorem~\ref{thm:charact-domin-splitting} is easy: see Fact~\ref{fact:domin->Lyap-gap}.
We prove the following strengthening of the converse, which involves the measures $\nu_{\sx} := (\delta_{\sx} + \ldots + \delta_{\sigma^{\pi(\sx)-1}(\sx)})/\pi(\sx)$ as in \eqref{eqn:nu-x} for periodic $\sx\in X$, of period $\pi(\sx)\in\NN^*$.

\begin{thm} \label{thm:sufficient-domin-splitting}
Let $X$ be a subshift of finite type and $(\sigma,\Phi)$ a locally constant cocycle over~$X$, where $\sigma : X\to X$ is the shift and $\Phi : X \to \GL(d,\RR)$ for some $d\in\NN^*$.
For $1\leq i\leq d-1$, suppose there exists $c>0$ such that the Lyapunov exponents with respect to $(\sigma,\Phi)$ satisfy $\chi_i(\nu_{\sx}) \geq \chi_{i+1}(\nu_{\sx}) + c$ for all ergodic $\sigma$-invariant measures $\nu_{\sx}$ with $\sx\in X$ periodic.
Then the cocycle $(\sigma,\Phi)$ has a dominated splitting of index~$i$.
\end{thm}

By Fact~\ref{fact:Lyap-gap-periodic-meas}, the assumption on the Lyapunov exponents in Theorem~\ref{thm:sufficient-domin-splitting} can be rephrased as $(\lambda_i - \lambda_{i+1})(\Phi^{(\pi(\sx))}(\sx)) \geq c\pi(\sx)$ for all periodic $\sx\in X$.

We prove Theorem~\ref{thm:sufficient-domin-splitting}, first in an important special case (Section~\ref{ss.onereductive}), then in the general case (Section~\ref{subsec:proof-teoA-nonreductive}).
We show that in Theorems \ref{thm:charact-domin-splitting} and~\ref{thm:sufficient-domin-splitting} it is necessary to require the existence of a \emph{uniform} gap between Lyapunov exponents of periodic orbits, not just a gap (Section~\ref{s.nonunifgap}).

In the whole section we fix a subshift $X=\Sigma_A$ of finite type on a finite alphabet $\{1, \ldots, N\}$, defined by a matrix $A =(a_{i,j})_{1\leq i,j\leq N}$ whose entries are $0$'s and $1$'s and which satisfies the irreducibility assumption of Section~\ref{subsec:subshifts}.
We fix a cocycle $(\sigma,\Phi)$ over~$\Sigma_A$, where $\sigma : \Sigma_A\to\Sigma_A$ is the shift and $\Phi : \Sigma_A \to \GL(d,\RR)$, and we assume that it is locally constant, \ie there exists $\varphi : \{1,\dots,N\}\to\GL(d,\RR)$ such that $\Phi(\sx)=\varphi(x_0)$ for all $\sx=(x_k)_{k\in\ZZ}\in\Sigma_A$.
We denote by $\Gamma$ the semigroup generated by $\Phi(\Sigma_A) = \{\varphi(i) \,:\, 1\leq i\leq N\}$, and by $H$ its Zariski closure in $\GL(d,\RR)$, which is a real linear algebraic group (see \eg \cite[Lem.\,4.2]{Benoist-notes}).

\subsection{Preliminary lemmas}

For $1\leq j\leq N$, we denote by $\Pi_{j}$ the set of periodic points $\sx = (x_k)_{k\in\ZZ}$ of the shift $\sigma: \Sigma_A \to \Sigma_A$ such that $x_0 = j$.
For $\sx\in\Pi_{j}$, we denote by $\pi(\sx)$ its period.
Let $\Gamma_j \subset \Gamma \subset \GL(d,\RR)$ be the set of elements of the form $\Phi^{(n\pi(\sx))}(\sx)$ where $\sx \in \Pi_j$ and $n\in\NN$.
The following simple but important fact uses the local constancy of the cocycle.

\begin{lema}\label{l.semigroup}
For $1\leq j\leq N$, the set $\Gamma_j$ is a semigroup.  
\end{lema}  

\begin{proof}
Let $\sx = (x_k)_{k\in\ZZ}$ and $\sy = (y_k)_{k\in\ZZ}$ be elements of $\Pi_j$.
For $m,n\in\NN$, let us check that $\Phi^{(m\pi(\sx))}(\sx) \, \Phi^{(n\pi(\sy))}(\sy) \in \Gamma_j$.
We have $a_{x_{m\pi(\sx)-1},y_0} = a_{x_{m\pi(\sx)-1},j} = a_{x_{m\pi(\sx)-1},x_{m\pi(\sx)}} = 1$ and $a_{y_{n\pi(\sy)-1},x_0} = a_{y_{n\pi(\sy)-1},j} = a_{y_{n\pi(\sy)-1},y_{n\pi(\sy)}} = 1$.
Therefore, we can define an element $\sz = (z_k)_{k\in\ZZ} \in \Sigma_A$ by setting $z_k$ to be equal to $x_k$ for $-m\pi(\sx)\leq k\leq -1$, and to $y_k$ for $0\leq k\leq n\pi(\sy)-1$, and requiring $z_{k+m\pi(\sx)+n\pi(\sy)} = z_k$ for all $k\in\ZZ$.
This element $\sz$ belongs to $\Pi_j$, its period $\pi(\sz)$ is a submultiple of $m\pi(\sx)+n\pi(\sy)$ and, using the fact that the cocycle is locally constant,
\begin{eqnarray*}
\Phi^{(m\pi(\sx))}(\sx) \, \Phi^{(n\pi(\sy))}(\sy) & = & \varphi(x_{m\pi(\sx)-1}) \cdots \varphi(x_0) \varphi(y_{n\pi(\sy)-1}) \cdots \varphi(y_0)\\
& = & \varphi(z_{m\pi(\sx)+n\pi(\sy)-1}) \cdots \varphi(z_0) \ =\ \Phi^{(m\pi(\sx)+n\pi(\sy))}(\sz),
\end{eqnarray*}
hence $\Phi^{(m\pi(\sx))}(\sx) \, \Phi^{(n\pi(\sy))}(\sy) \in \Gamma_j$.
\end{proof}

The following lemma allows to translate the hypothesis on the Lyapunov exponents into a geometric property of the embedding of the semigroup $\Gamma_j$ in $\GL(d,\RR)$.

\begin{lema} \label{lem:lambda-farfromwall}
There exists $M'>0$ with the following property: for any $1\leq i\leq d-1$ and $1\leq j\leq N$, if there exists $c>0$ such that $(\lambda_i - \lambda_{i+1})(\Phi^{(\pi(\sx))}) \geq c \pi(\sx)$ for all $\sx\in\Pi_j$, then for any $\gamma=\Phi^{(n\pi(\sx))}(\sx)\in\Gamma_j$,
\begin{equation} \label{eqn:misc-kappa-lambda-2}
\sqrt{2} \, \|\mu(\gamma)\| \geq (\lambda_i - \lambda_{i+1})(\gamma) \geq c \, n \pi(\sx) \geq \frac{c}{M'}\,\|\mu(\gamma)\| \geq \frac{c}{M'}\,\|\lambda(\gamma)\|.
\end{equation}
In particular, the limit cone $\cC_{\lambda}(\Gamma_j)$ does not meet the $i$-th wall $\{ x_i=x_{i+1}\}$ of~$\mathfrak{a}^+$ outside of~$0$.
\end{lema}

\begin{proof}
The first inequality in \eqref{eqn:misc-kappa-lambda-2} follows from the fact that for any $g\in\GL(d,\RR)$, using \eqref{eqn:lambda-mu-relation}, we have $(\lambda_i - \lambda_{i+1})(g) \leq (\lambda_1 - \lambda_d)(g) \leq \sqrt{2} \, \|\lambda(g)\| \leq \sqrt{2} \, \|\mu(g)\|$.
The second inequality in \eqref{eqn:misc-kappa-lambda-2} is immediate from the assumption and the fact that $\lambda(\Phi^{(n\pi(\sx))}(\sx)) = n\,\lambda(\Phi^{(\pi(\sx))}(\sx))$.
Let $M' := \max_{1\leq j\leq N} \|\mu(\varphi(j))\| > 0$.
By \eqref{eqn:mu-subadd}, for any $n\in\NN^*$ and $\sx=(x_k)_{k\in\ZZ}\in\Pi_j$ we have
$$\|\mu(\Phi^{(n\pi(\sx))}(\sx))\| \leq \|\mu(\varphi(x_{n\pi(\sx)-1}))\| + \dots + \|\mu(\varphi(x_0))\| \leq n\pi(\sx) \, M',$$
which yields the third inequality in \eqref{eqn:misc-kappa-lambda-2}.
For the fourth inequality, see \eqref{eqn:lambda-mu-relation}.
\end{proof}

\subsection{Proof of Theorem~\ref{thm:sufficient-domin-splitting} in the case that $H$ is reductive} \label{ss.onereductive}

By Facts \ref{fact-BG} and~\ref{fact:Lyap-gap-periodic-meas}, it is sufficient to prove the following.

\begin{prop} \label{prop-teoA-reductive}
Suppose the Zariski closure $H$ of $\Gamma$ in $\GL(d,\RR)$ is reductive.
For $1\leq i\leq d-1$, if there exists $c>0$ such that $(\lambda_i - \lambda_{i+1})(\Phi^{(\pi(\sx))}) \geq c \pi(\sx)$ for all periodic $\sx\in\Sigma_A$, then there exist $C,C'>0$ such that for any $n\in\NN$ and $\sx \in \Sigma_A$,
$$(\mu_i - \mu_{i+1})(\Phi^{(n)}(\sx)) \geq Cn - C'. $$
\end{prop}

The first ingredient in the proof is the following.

\begin{lema} \label{l.Hj-reductive}
For any $1\leq j\leq N$, the Zariski closure $H$ of $\Gamma$ and the Zariski closure $H_j$ of $\Gamma_j$ in $\GL(d,\RR)$ have the same identity component (for the Zariski topology).
In particular, $H$ is reductive if and only if $H_j$ is.
\end{lema}

\begin{proof}
Let $H_0$ (\resp $(H_j)_0$) be the identity component of $H$ (\resp $H_j$) for the Zariski topology; it is a finite-index normal subgroup of $H$ (\resp $H_j$) which is irreducible, \ie which cannot be written as the union of two proper Zariski-closed subsets (see \eg \cite[Lem.\,5.21]{BenoistQuint-livre}).
Since $\Gamma_j \subset \Gamma$, we have $H_j \subset H$, hence $(H_j)_0 \subset H_0$.
We can write $H_j = h_1 (H_j)_0 \cup \dots \cup h_r (H_j)_0$ for some $h_1,\dots,h_r\in H_j$.
By Fact~\ref{fact-specification}, there is a finite subset $S$ of~$\Gamma$ such that $\Gamma \subset \bigcup_{s_1,s_2\in S}\, s_1^{-1}\,\Gamma_j\,s_2^{-1}$.
Setting $S' := S^{-1}h_1 \cup \dots \cup S^{-1}h_r \cup S^{-1} \subset H$, we see that $\Gamma$ is contained in the Zariski-closed subset $\bigcup_{s'_1,s'_2\in S'}\, s'_1\,(H_j)_0\,s'_2$ of $\GL(d,\RR)$, and so the same holds for~$H$.
Thus $H_0$ is the union of its Zariski-closed subsets $H_0 \cap (s'_1\,(H_j)_0\,s'_2)$ for $s'_1,s'_2\in S'$.
Since $H_0$ is irreducible, it is equal to one of these subsets: $H_0 \subset s'_1\,(H_j)_0\,s'_2$ for some $s'_1,s'_2\in S'$.
Since $H_0$ is normal in~$H$, we have $H_0 = {s'_1}^{-1} H_0\,s'_1 \subset (H_j)_0\,s'_2 s'_1$.
In particular, $\mathrm{id}\in (H_j)_0\,s'_2 s'_1$, hence $s'_2 s'_1\in (H_j)_0$ and $H_0 \subset (H_j)_0$.
On the other hand, $(H_j)_0 \subset H_0$, hence $H_0 = (H_j)_0$.
\end{proof}

The second ingredient is the following, based on Lemma~\ref{lem:lambda-farfromwall} and Fact~\ref{fact:Benoist}.

\begin{lema} \label{l.farfromwall}
For $1\leq j\leq N$, suppose $H_j$ is reductive.
For $1\leq i\leq d-1$, if there exists $c>0$ such that $(\lambda_i - \lambda_{i+1})(\Phi^{(\pi(\sx))}) \geq c \pi(\sx)$ for all $\sx\in\Pi_j$, then there exist $c_0,c'_0>0$ such that for any $\gamma\in\Gamma_j$,
$$(\mu_i - \mu_{i+1})(\gamma) \geq c_0\,\|\mu(\gamma)\| - c'_0.$$
In other words, the limit cone $\cC_{\mu}(\Gamma_j)$ does not meet the $i$-th wall $\{ x_i=x_{i+1}\}$ of~$\mathfrak{a}^+$ outside of~$0$.
\end{lema}

\begin{proof}
Suppose there exists $c>0$ such that $(\lambda_i - \lambda_{i+1})(\Phi^{(\pi(\sx))}) \geq c \pi(\sx)$ for all $\sx\in\Pi_j$.
By Lemma~\ref{lem:lambda-farfromwall}, there exists $c_0>0$ such that $(\lambda_i - \lambda_{i+1})(\gamma) \geq c_0\,\|\mu(\gamma)\|$ for all $\gamma\in\Gamma_j$.
Since $H_j$ is reductive, by Fact~\ref{fact:Benoist} there exist a finite subset $S$ of $\Gamma_j$ and a constant $M>0$ such that for any $\gamma\in\Gamma$ we can find $s\in S$ with
$$\big|(\mu_i - \mu_{i+1})(\gamma) - (\lambda_i - \lambda_{i+1})(s\gamma)\big| \leq \sqrt{2} \, \| \mu(\gamma) - \lambda(s\gamma) \| \leq \sqrt{2} M.$$
Using \eqref{eqn:mu-subadd}, we then have
$$(\mu_i - \mu_{i+1})(\gamma) \geq c_0 \, \|\mu(s\gamma)\| - \sqrt{2} M \geq c_0 \, \|\mu(\gamma)\| - c'_0,$$
where $c'_0 = \sqrt{2} M + c_0\,\max_{s\in S} \|\mu(s)\|$.
\end{proof}

\begin{proof}[Proof of Proposition~\ref{prop-teoA-reductive}]
Choose an integer $1\leq j\leq N$.
Since $H$ is reductive, Lemma~\ref{l.Hj-reductive} implies that the Zariski closure $H_j$ of~$\Gamma_j$ in $\GL(d,\RR)$ is reductive.
Suppose there exists $c>0$ such that $(\lambda_i - \lambda_{i+1})(\Phi^{(\pi(\sx))}) \geq c \pi(\sx)$ for all periodic $\sx\in\Sigma_A$.
Let $n_0>0$ be given by Fact~\ref{fact-specification}, let
$$M'' := \max_{z_1,\dots,z_{n_0}\in\{1,\dots,N\}} \big\| \mu\big(\varphi(z_{n_0})\cdots\varphi(z_1)\big)\big\|,$$
and let $c_0,c'_0>0$ be given by Lemma~\ref{l.farfromwall}.
We claim that for any $\sx = (x_k)_{k\in\ZZ} \in \Sigma_A$ and $n>0$,
\begin{equation}\label{eqn:allx}
(\mu_i - \mu_{i+1})(\Phi^{(n)}(\sx)) \geq Cn - C',
\end{equation}
where $C := c_0\,c/\sqrt{2}$ and $C':=c'_0+2\sqrt{2}M''$.
Indeed, by Fact~\ref{fact-specification}, there exists $\sy \in \Pi_j$ with $\pi(\sy) = n+2n_0$ such that $y_{n_0+k}=x_k$ for all $0\leq k\leq n-1$.
By Lemma~\ref{l.farfromwall} and the inequality $\sqrt{2} \, \|\mu(\Phi^{(\pi(\sy))}(\sy))\| \geq c \, \pi(\sy)$ in Lemma~\ref{lem:lambda-farfromwall}, we have
$$(\mu_i - \mu_{i+1})(\Phi^{(\pi(\sy))}(\sy)) \geq c_0\,\big\|\mu(\Phi^{(\pi(\sy))}(\sy))\big\| - c'_0 \geq \frac{c_0 c}{\sqrt{2}} \, \pi(\sy) - c'_0 \geq Cn - c'_0.$$
On the other hand, by \eqref{eqn:mu-subadd} we have
$$\big\| \mu(\Phi^{(n)}(\sx)) - \mu(\Phi^{(\pi(\sy))}(\sy)) \big\| \leq 2M''.$$
This proves \eqref{eqn:allx} and completes the proof.
\end{proof}

\subsection{Proof of Theorem~\ref{thm:sufficient-domin-splitting} in the general case} \label{subsec:proof-teoA-nonreductive}

Let $H = L \ltimes R_u(H)$ be a Levi decomposition of~$H$, and $\pi^{ss} : H \to L$ the natural projection.
We can define a new cocycle $\Phi^{ss} := \pi^{ss}\circ\Phi : X=\Sigma_A\to\GL(d,\RR)$ over~$\sigma$, with image in~$L$.
This cocycle is still locally constant since $\Phi^{ss}(\sx) = \pi^{ss}\circ\varphi(x_0)$ for all $\sx=(x_k)_{k\in\ZZ}\in\Sigma_A$.
The Zariski closure in $\GL(d,\RR)$ of the semigroup generated by the family $\{\pi^{ss}\circ\varphi(i)\}_{i \in \{1,\dots,N\}}$ is the reductive group~$L$.

There is a sequence $(g_k) \in \GL(d,\RR)^{\NN}$ such that $g_k h g_k^{-1}\to\pi^{ss}(h)$ for all $h\in H$ (see Section~\ref{ss.semisimp}).
In particular, $g_k\Phi(\sx)g_k^{-1} \to \Phi^{ss}(\sx)$ for all $\sx \in \Sigma_A$, and so $\lambda(\Phi^{(n)}(\sx)) = \lambda((\Phi^{ss})^{(n)}(\sx))$ for all $\sx\in\Sigma_A$ and $n \in \ZZ$ (see \eqref{eqn:lambda-pi-ss}).
By \eqref{eqn:Lyap-exp-periodic-meas}, the assumption on Lyapunov exponents can be rephrased as $(\lambda_i -\nolinebreak \lambda_{i+1})(\Phi^{(\pi(\sx))}(\sx)) \geq c\pi(\sx)$ for all periodic $\sx\in X$, and so the same inequalities hold if we replace $\Phi$ by $\Phi^{ss}$.

By Proposition~\ref{prop-teoA-reductive} the cocycle $\Phi^{ss}$ has a dominated splitting of index~$i$.
Since the cocycles are locally constant, the simple convergence $g_k\Phi(\cdot)g_k^{-1} \to \Phi^{ss}$ is a convergence for the compact open topology.
By openness (Fact~\ref{fact-dsopen}), for large enough~$k$ the cocycle $g_k \Phi(\cdot) g_k^{-1}$ also has a dominated splitting of index~$i$.
The definition of dominated splitting (see Section~\ref{subsec:prelim-domin-splitting}) is clearly invariant under conjugation by $\GL(d,\RR)$; therefore $\Phi$ also has a dominated splitting of index~$i$ and this completes the proof. 

\begin{remark} \label{rem:proof-teoA-awayfromwall}
The fact that $\Phi$ has a dominated splitting of index~$i$ implies (Fact~\ref{fact-BG}) that for any $1\leq j\leq N$ the limit cone $\cC_{\mu}(\Gamma_j)$ does not meet the $i$-th wall $\{ x_i=x_{i+1}\}$ of~$\mathfrak{a}^+$ outside of~$0$.
This is not immediate from Lemma~\ref{lem:lambda-farfromwall} when $H_j$ is not reductive.
\end{remark}

\subsection{Uniformity is necessary in Theorem~\ref{thm:charact-domin-splitting}} \label{s.nonunifgap}

In this section we assume $N=2$ and construct a locally constant cocycle $\Phi : \Sigma = \{1,2\}^{\ZZ} \to \GL(3,\RR)$ over the full 2-shift $\sigma : \{1,2\}^{\ZZ} \to \{1,2\}^{\ZZ}$ for which $\chi_2(\nu) > \chi_3(\nu)$ for any $\sigma$-invariant periodic measure $\nu$ on $\{1,2\}^{\ZZ}$, but for which there is a sequence $(\nu_n)_{n\in\NN}$ of $\sigma$-invariant periodic measures with $\chi_2(\nu_n)-\chi_3(\nu_n) \to 0$; in particular, this cocycle cannot have a dominated splitting of index~$2$.
This shows that in Theorem~\ref{thm:charact-domin-splitting} it is necessary to require the existence of a \emph{uniform} gap between Lyapunov exponents of periodic orbits, not just a gap.

The example here is modeled on one presented by J.~Bochi in \cite[Part\,2]{Bochi-ICMtalk}.
Note that it is not strongly irreducible; it would be interesting to know if a strongly irreducible example exists.
We refer to \cite{Bochi-ICM} and \cite[\S\,5]{BS} for further discussions.

\begin{ex} \label{ex:nonunifgap}
Let $\Phi : \{1,2\}^{\ZZ} \to \GL(3,\RR)$ be the locally constant cocycle over the full 2-shift sending $\sx=(x_k)_{k\in\ZZ}\in\{1,2\}^{\ZZ}$ to $\varphi(x_0)$ where
$$ \varphi(1) = \begin{bmatrix} 2 & 0 & 0 \\ 0 & \frac{1}{8} & 0 \\ 0 & 0 & \frac{1}{2} \end{bmatrix} \quad\quad\text{and}\quad\quad  \varphi(2) = \begin{bmatrix} 0 & -1 & 0 \\ 1 & 0 & 0 \\ 0 & 0 & \frac{1}{e} \end{bmatrix}. $$
A simple calculation shows that for any periodic $\sx=(x_k)_{k\in\ZZ}\in\{1,2\}^{\ZZ}$ of period $\pi(\sx)$, there exist $\ell,\ell',m\in\NN$ with $\ell+\ell'+2m=2\pi(\sx)$ such that
$$ \Phi^{2\pi(\sx)}(\sx) = \begin{bmatrix} 2^{\ell} 2^{-3\ell'} & 0 & 0 \\ 0 & 2^{-3\ell}2^{\ell'} & 0 \\ 0 & 0 & 2^{-(\ell+\ell')} e^{-2m} \end{bmatrix}. $$
(Indeed, $\Phi^{2\pi(\sx)}(\sx)$ is a product of positive powers of $\varphi(0)$ and $\varphi(1)$; looking at its action on~$\RR^3$, we see that the directions of $e_1=(1,0,0)$ and $e_2=(0,1,0)$ are switched each time $\varphi(1)$ is applied; the integer $\ell$ (\resp $\ell'$) counts how many times $e_1$ is multiplied by $2$ (\resp by $1/8$) when applying $\varphi(1)$; the integer $2m$ counts the number of occurrences of $\varphi(2)$ in $\Phi^{2\pi(\sx)}(\sx)$.)
Using \eqref{eqn:Lyap-exp-periodic-meas}, it follows that the (unordered) list of Lyapunov exponents of the measure $\nu_{\sx}$ from \eqref{eqn:nu-x} is
$$ \{ \chi_i(\nu_{\sx})\}_{1\leq i\leq 3} = \bigg\{ \frac{\ell-3\ell'}{\ell+\ell'+2m} \log 2, \ \frac{\ell'-3\ell}{\ell+\ell'+2m} \log 2, \ \frac{-(\ell+\ell')\log 2 - 2m}{\ell+\ell'+2m} \bigg\} . $$
Therefore $\chi_1(\nu_{\sx})>\chi_2(\nu_{\sx})>\chi_3(\nu_{\sx})$ if $\ell\neq\ell'$, and $\chi_1(\nu_{\sx})=\chi_2(\nu_{\sx})>\chi_3(\nu_{\sx})$ if $\ell=\ell'$.
Moreover, if $\ell$ is close to~$\ell'$ and much larger than~$m$, then $\chi_1(\nu_{\sx}),\chi_2(\nu_{\sx}),\chi_3(\nu_{\sx})$ can be made arbitrarily close to each other. 
\end{ex}

\begin{remark}
For locally constant cocycles over a subshift with values in $\GL(2,\RR)$, there is a finer understanding of the relationship between eigenvalues of the cocycle over periodic points and domination: see \cite[Th.\,4.1]{ABY} and also \cite{BR}.
In higher dimensions, the situation is less understood: see \cite{BG,Bochi-ICM,Bochi-ICMtalk}.
\end{remark}

\section{Eigenvalue gaps for representations of finitely generated groups} \label{sec:gaps-in-groups}

In this section we prove Proposition~\ref{prop:eig-sing-gap-group} and Theorem~\ref{thm:charact-domin-splitting-hyp-gp}.
We first introduce some terminology.

\subsection{Terminology} \label{subsec:terminology-gaps}

Let $\Gamma$ be a group with a finite generating subset~$F$, \ie any element of~$\Gamma$ can be written as a product of elements of $F\cup F^{-1}$.
For $\gamma\in\Gamma\smallsetminus\{e\}$, we denote by
\begin{itemize}
  \item $|\gamma|_F$ the \emph{word length} of $\gamma$ with respect to~$F$, \ie
  $$|\gamma|_F = \inf\big\{ k\in\NN^*\ :\ \exists f_0,\dots,f_{k-1}\in F\cup F^{-1}\ \mathrm{s.t.}\ \gamma=f_0\cdots f_{k-1}\big\},$$
  \item $|\gamma|_{F,\infty}$ the \emph{stable length} of $\gamma$ with respect to~$F$, \ie
  $$|\gamma|_{F,\infty} = \lim_{n\to +\infty} \frac{|\gamma^n|_F}{n},$$
  \item $\ell_F(\gamma)$ the \emph{translation length} of $\gamma$ in the Cayley graph $\mathrm{Cay}(\Gamma,F)$, \ie
  $$\ell_F(\gamma) := \inf_{\beta\in\Gamma} |\beta^{-1} \gamma \beta|_F.$$
\end{itemize}
We also set $|e|_F = |e|_{F,\infty} = \ell_F(e) = 0$.
The stable length and the translation length are invariant under conjugation, and satisfy
\begin{equation} \label{eqn:stable-transl-length}
|\gamma|_{F,\infty} \leq \ell_F(\gamma)
\end{equation}
for all $\gamma\in\Gamma$.
The word length defines a $\Gamma$-invariant metric $d_{F}$ on the Cayley graph $\mathrm{Cay}(\Gamma,F)$, given by $d_{F}(\gamma,\gamma'):=|\gamma^{-1}\gamma'|_F$ for all $\gamma,\gamma'\in\Gamma$.
Recall that $\Gamma$ is said to be \emph{word hyperbolic} if the metric space $(\mathrm{Cay}(\Gamma,F),d_F)$ is Gromov hyperbolic.
This is independent of the choice of the finite generating set~$F$ by the following remark.

\begin{remark} \label{rem:change-gen-set}
If $F'$ is another finite generating subset of~$\Gamma$, then by setting $M := \max(\max_{f\in F} |f|_{F'},\max_{f'\in F'} |f'|_F) \geq 1$ we have $M^{-1} \, |\gamma|_{F'} \leq |\gamma|_F \leq M \, |\gamma|_{F'}$ for all $\gamma\in\Gamma$, which implies $M^{-1} \, |\gamma|_{F',\infty} \leq |\gamma|_{F,\infty} \leq M \, |\gamma|_{F',\infty}$ and $M^{-1} \, \ell_{F'}(\gamma) \leq \ell_F(\gamma) \leq M \, \ell_{F'}(\gamma)$.
\end{remark}

We shall adopt the following terminology.

\begin{defi} \label{def:gaps}
For $1\leq i\leq d-1$, a representation $\rho : \Gamma\to\GL(d,\RR)$ has
\begin{itemize}
  \item a \emph{uniform $i$-gap in singular values} if there exist $c,c'>0$ such that
  \[ (\mu_i - \mu_{i+1})(\rho(\gamma)) \geq c \, |\gamma|_F - c'  \ \ \ \ \ \ \forall \gamma \in \Gamma, \]
  \item a \emph{weak uniform $i$-gap in eigenvalues} if there exist $c,c'>0$ such that
  \[ (\lambda_i - \lambda_{i+1})(\rho(\gamma)) \geq c \, |\gamma|_{F,\infty} - c' \ \ \ \ \ \ \forall \gamma \in \Gamma, \]
  \item a \emph{strong uniform $i$-gap in eigenvalues} if there exist $c,c'>0$ such that
  \[ (\lambda_i - \lambda_{i+1})(\rho(\gamma)) \geq c \, \ell_F(\gamma) - c'  \ \ \ \ \ \ \forall \gamma \in \Gamma. \]
\end{itemize}
\end{defi}

These notions do not depend on the choice of finite generating set~$F$, by Remark~\ref{rem:change-gen-set}.
A strong uniform $i$-gap in eigenvalues implies a weak one by \eqref{eqn:stable-transl-length}.

\begin{remarks} \label{rem:gaps}
\begin{enumerate}
  \item\label{item:i-(d-i)-gap} All notions of uniform $i$-gap in Definition~\ref{def:gaps} are equivalent to the corresponding notions of uniform $(d-i)$-gap.
  Indeed, by \eqref{eqn:mu-lambda-inverse} we have $(\bullet_i-\bullet_{i+1})(g)=(\bullet_{d-i}-\bullet_{d-i+1})(g^{-1})$ for $\bullet=\mu$ or~$\lambda$ and for all $g\in\GL(d,\RR)$, and for any $\gamma\in\Gamma$ we have $|\gamma|_F=|\gamma^{-1}|_F$ and similarly for $|\cdot|_{F,\infty}$ and~$\ell_F$.
  \item\label{item:gap-conj} All notions of uniform $i$-gap in Definition~\ref{def:gaps} are invariant under conjugation.
  Indeed, for eigenvalues this follows from the invariance of $\lambda : G\to\mathfrak{a}^+$ under conjugation, and for singular values this follows from \eqref{eqn:mu-subadd}.
  \item\label{item:sing-weak-eig} If $\rho$ has a uniform $i$-gap in singular values, then it also has a weak uniform $i$-gap in eigenvalues, using \eqref{eqn:lambda-mu-relation} and the definition of stable length.
  \item\label{item:gap-displacing} If $\rho$ has a strong uniform $i$-gap in eigenvalues, then it is \emph{displacing} in the sense of \cite{DGLM}: there exist $c,c'>0$ such that $\|\lambda(\rho(\gamma))\| \geq c \, \ell_F(\gamma) - c'$ for all $\gamma \in \Gamma$.
  Similarly, if $\rho$ has a weak uniform $i$-gap in eigenvalues, then there exist $c,c'>0$ such that $\|\lambda(\rho(\gamma))\| \geq c \, |\gamma|_{F,\infty} - c'$ for all $\gamma \in \Gamma$.
  \item\label{item:strong-weak-gap-hyp} If $\Gamma$ is word hyperbolic, then having a weak uniform $i$-gap in eigenvalues is equivalent to having a strong uniform $i$-gap in eigenvalues, since
\begin{equation} \label{eqn:stable-transl-length-hyp}
\ell_F(\gamma) - 16 \delta \leq |\gamma|_{F,\infty} \leq \ell_F(\gamma)
\end{equation}
for all $\gamma\in\Gamma$, where $\delta\geq 0$ is the hyperbolicity constant of~$\Gamma$ (see \cite[Ch.\,10, Prop.\,6.4]{cdp90}).
For such groups we shall just talk about having a uniform $i$-gap in eigenvalues.
\end{enumerate}
\end{remarks}

\subsection{Proof of Proposition~\ref{prop:eig-sing-gap-group}} \label{subsec:proof-eig-sing-gap-group}

Let $\rho : \Gamma\to\GL(d,\RR)$ be a representation.
We have already seen (Remark~\ref{rem:gaps}.\eqref{item:sing-weak-eig}) that if $\rho$ has a uniform $i$-gap in singular values, then it also has a uniform $i$-gap in eigenvalues.
Let us prove the converse implication.
  
We may assume that the Zariski closure $H$ of $\rho(\Gamma)$ in $\GL(d,\RR)$ is reductive: otherwise, as in Sections \ref{ss.semisimp} and~\ref{subsec:proof-teoA-nonreductive}, consider a Levi decomposition $H = L \ltimes R_u(H)$ of~$H$, the natural projection $\pi^{ss}: H \to L$, and the semisimplification $\rho^{ss} := \pi^{ss}\circ\rho : \Gamma\to\GL(d,\RR)$.
The Zariski closure of $\rho^{ss}(\Gamma)$ in $\GL(d,\RR)$ is the reductive group~$L$.
There is a sequence $(g_k) \in \GL(d,\RR)^{\NN}$ such that $g_k h g_k^{-1}\to\pi^{ss}(h)$ for all $h\in H$.
In particular, $g_k\rho(\gamma)g_k^{-1} \to \rho^{ss}(\gamma)$ for all $\gamma \in \Gamma$, and so $\lambda(\rho(\gamma)) = \lambda(\rho^{ss}(\gamma))$ for all $\gamma\in\Gamma$ (see \eqref{eqn:lambda-pi-ss}); therefore, $\rho$ has a uniform $i$-gap in eigenvalues if and only if $\rho^{ss}$ does.  
On the other hand, having a uniform $i$-gap in singular values is an open property (see \cite{KLP,BPS}) which is invariant under conjugation (Remark~\ref{rem:gaps}.\eqref{item:gap-conj}); therefore, if $\rho^{ss}$ has a uniform $i$-gap in singular values, then so does~$\rho$.

We now assume that the Zariski closure of $\rho(\Gamma)$ in $\GL(d,\RR)$ is reductive.

Suppose $\rho$ has a (strong) uniform $i$-gap in eigenvalues.
By Remark~\ref{rem:gaps}.\eqref{item:gap-displacing}, there exist $c_0,c'_0>0$ such that $\|\lambda(\rho(\gamma))\| \geq c_0 \, \ell_F(\gamma) - c'_0$ for all $\gamma \in \Gamma$.
Since $\Gamma$ is word hyperbolic, it satisfies \cite[Prop.\,2.2.1]{DGLM} what Delzant--Guichard--Labourie--Mozes call \emph{property~(U)} and which we call \emph{weak property~(U)} in Definition~\ref{def:prop-U-groups} below: there exist a finite subset $S'$ of~$\Gamma$ and constants $c_1,c'_1>0$ such that $\max_{s\in S'} \ell_F(s\gamma) \geq c_1 \, |\gamma|_F - c'_1$ for all $\gamma \in \Gamma$.
As observed in \cite[Lem.\,2.0.1]{DGLM}, this implies that $\rho$ is a quasi-isometric embedding, \ie there exist $c_2,c'_2>0$ such that for any $\gamma\in\Gamma$,
\begin{equation} \label{eqn:QI-embed}
\|\mu(\rho(\gamma))\| \geq c_2 \, |\gamma|_F - c'_2.
\end{equation}
Indeed, for any $\gamma\in\Gamma$ and $s\in S'$ with $\ell_F(s\gamma) \geq c_1 \, |\gamma|_F - c'_1$, using \eqref{eqn:lambda-mu-relation} we have $\Vert\mu(\rho(s\gamma))\Vert \geq \Vert\lambda(\rho(s\gamma))\Vert \geq c_0 \, \ell_F(s\gamma) - c'_0 \geq c_0 c_1 \, |\gamma|_F - (c_0 c'_1 + c'_0)$, and so we may take $c_2 = c_0 c_1$ and $c'_2 = c_0 c'_1 + c'_0 + \max_{s'\in S'} \Vert\mu(\rho(s'))\Vert$ by \eqref{eqn:mu-subadd}.

Given \eqref{eqn:QI-embed}, in order to conclude that $\rho$ has a uniform $i$-gap in singular values, it is sufficient to prove that there exist $c_3,c'_3>0$ such that for any $\gamma\in\Gamma$,
\begin{equation} \label{eqn:cone-away-from-wall}
(\mu_i - \mu_{i+1})(\rho(\gamma)) \geq c_3 \, \|\mu(\rho(\gamma))\| - c'_3.
\end{equation}
We now prove \eqref{eqn:cone-away-from-wall}.
By \eqref{eqn:mu-subadd} and \eqref{eqn:mu-lambda-inverse}, if we set $M := \max_{f\in F} \|\mu(\rho(f))\|>0$, then for any $\gamma\in\Gamma$ we have $\|\mu(\rho(\gamma))\| \leq M \, |\gamma|_F$.
Applying this to $\gamma^n$, dividing by~$n$, and taking the limit, we find $\|\lambda(\rho(\gamma))\| \leq M \, |\gamma|_{F,\infty}$ for all $\gamma\in\Gamma$ (using \eqref{eqn:lambda-mu-relation} and the definition of the stable length).
Since $\rho$ has a (weak) uniform $i$-gap in eigenvalues, there exist $c,c'>0$ such that for any $\gamma\in\Gamma$,
$$(\lambda_i - \lambda_{i+1})(\rho(\gamma)) \geq c \, |\gamma|_{F,\infty} - c' \geq cM^{-1} \, \|\lambda(\rho(\gamma))\| - c'.$$
Applying this to $\gamma^n$, dividing by~$n$, and taking the limit, we find
$$(\lambda_i - \lambda_{i+1})(\rho(\gamma)) \geq cM^{-1} \, \|\lambda(\rho(\gamma))\|$$
for all $\gamma\in\Gamma$.
In particular, the limit cone $\cC_{\lambda}(\rho(\Gamma))$ does not meet the $i$-th wall $\{ x_i=x_{i+1}\}$ of~$\mathfrak{a}^+$ outside of~$0$.
Now, since the Zariski closure of $\rho(\Gamma)$ in $\GL(d,\RR)$ is reductive, we have $\cC_{\mu}(\rho(\Gamma)) = \cC_{\lambda}(\rho(\Gamma))$ (Fact \ref{fact:Benoist}), and so for any sequence $(\gamma_n)\in\Gamma^{\NN}$ with $\Vert\mu(\rho(\gamma_n))\Vert \to +\infty$ we have
$$\liminf_{n\in\NN} \, \frac{(\mu_i - \mu_{i+1})(\rho(\gamma_n))}{\|\mu(\rho(\gamma_n))\|} \geq cM^{-1}.$$
In particular, using \eqref{eqn:QI-embed}, we see that if we fix $0 < c_3 < cM^{-1}$, then there is a finite subset $S''$ of~$\Gamma$ such that $(\mu_i - \mu_{i+1})(\rho(\gamma)) \geq c_3 \, \|\mu(\rho(\gamma))\|$ for all $\gamma\in\Gamma\smallsetminus S''$.
Therefore \eqref{eqn:cone-away-from-wall} holds with $c'_3 = c_3 \max_{s''\in S''} \|\mu(\rho(s''))\|$.
This completes the proof of Proposition~\ref{prop:eig-sing-gap-group}.

\begin{remark}
The short proof of Proposition~\ref{prop:eig-sing-gap-group} given here, based on \cite{DGLM}, was pointed out to us by Konstantinos Tsouvalas.
In a previous version of this paper, we had given a direct proof of Proposition~\ref{prop:eig-sing-gap-group} that did not use \cite{DGLM}, by establishing an analogue of \cite[Th.\,5.17]{AMS} (see Fact~\ref{fact:Benoist}) simultaneously for the linear group $\rho(\Gamma)$ and the abstract group $\Gamma$, namely:
\textit{for any word hyperbolic group $\Gamma$ and any representation $\rho: \Gamma \to \mathrm{GL}(d,\RR)$ such that the Zariski closure of $\rho(\Gamma)$ in $\GL(d,\RR)$ is reductive, there exist a finite subset $S$ of~$\Gamma$ and a constant $M>0$ such that for any $\gamma \in \Gamma$ we can find $s\in S$ satisfying simultaneously the following two properties:}
$$\left\{ \begin{array}{l}
\|\lambda(\rho(s\gamma)) - \mu(\rho(\gamma))\| \leq M , \\
| |s\gamma|_{F,\infty} - |\gamma|_F| \leq M .
\end{array} \right.$$
We could then conclude directly that a uniform $i$-gap in eigenvalues with constants $(c,c'$) implies a uniform $i$-gap in singular values with constants $(c, cM + c' + \sqrt{2}M)$:
\begin{eqnarray*}
(\mu_i - \mu_{i+1})(\rho(\gamma)) & \geq & (\lambda_i - \lambda_{i+1})(\rho(s\gamma)) - \sqrt{2}M\\
& \geq & c \, |s\gamma|_{F,\infty} - (c' + \sqrt{2}M) \geq c \, |\gamma|_F - (cM + c' + \sqrt{2}M).
\end{eqnarray*}
This result will appear in separate upcoming work.
\end{remark}

\subsection{Link with Anosov representations of word hyperbolic groups}

The following was proved by Kapovich--Leeb--Porti \cite{KLP}, and an alternative proof was given in \cite{BPS}.

\begin{fact}[{\cite{KLP,BPS}}] \label{fact:Ano-sg-gap}
Let $\Gamma$ be a finitely generated group and $1\leq i\leq d-1$ an integer.
A representation $\rho : \Gamma\to\GL(d,\RR)$ has a uniform $i$-gap in singular values if and only if $\Gamma$ is word hyperbolic and $\rho$ is $P_i$-Anosov.
\end{fact}

Anosov representations are representations of word hyperbolic groups into semisimple or reductive Lie groups with good dynamical properties.
They were introduced by Labourie \cite{lab06} for fundamental groups of closed negatively curved manifolds and generalized by Guichard--Wienhard \cite{gw12} to all word hyperbolic groups.
They have been very much studied in the past fifteen years, and play an important role in so-called \emph{higher Teichm\"uller theory}.
We refer \eg to \cite[\S\,4]{Kassel-ICM} for the precise definition of a $P_i$-Anosov representation.

The fact that $P_i$-Anosov representations have a uniform $i$-gap in singular values is relatively easy (see \eg \cite[\S\,4.2]{GGKW}); the point of Fact~\ref{fact:Ano-sg-gap} is the converse implication, together with the fact that the $i$-gap implies the hyperbolicity of~$\Gamma$.

Proposition~\ref{prop:eig-sing-gap-group}, which is an answer to \cite[Question\,4.10]{BPS} for word hyperbolic groups, yields, together with Fact~\ref{fact:Ano-sg-gap}, the following characterization of Anosov representations.

\begin{cor} \label{cor:char-Ano-eigenv}
Let $\Gamma$ be a word hyperbolic group and $1\leq i\leq d-1$ an integer.
A representation $\rho: \Gamma \to \mathrm{GL}(d,\RR)$ is $P_i$-Anosov if and only if it has a uniform $i$-gap in eigenvalues.
\end{cor}

(See \cite[Th.\,1.7]{GGKW} for previous characterizations involving growth of eigenvalues.)

\subsection{Eigenvalue gaps in nonhyperbolic groups} \label{subsec:eig-gap-nonhyp-gp}

When $\Gamma$ is \emph{not} word hyperbolic, the link between gaps in eigenvalues and gaps in singular values is more subtle.
For instance, a nonfaithful representation $\rho : \Gamma\to\GL(d,\RR)$ may have a strong uniform $i$-gap in eigenvalues without having a uniform $i$-gap in singular values.

\begin{ex}[{see \cite[\S\,5]{DGLM}}]
There exist (nonhyperbolic) finitely generated groups $\Gamma$ with only finitely many conjugacy classes: see \cite{osi10}.
For such~$\Gamma$, any representation $\rho : \Gamma\to\GL(d,\RR)$ has a strong uniform $i$-gap in eigenvalues.
However,~the~constant representation $\rho : \gamma\mapsto\mathrm{id}\in\GL(d,\RR)$ does not have an $i$-gap in~singular~values.
\end{ex}

It may also happen that a faithful representation has a weak uniform $i$-gap in eigenvalues without having a strong uniform $i$-gap in eigenvalues nor a uniform $i$-gap in singular values.

\begin{ex} \label{ex:BS}
Let $\Gamma$ be the Baumslag--Solitar group $BS(1,2) = \langle a,b \ : \ bab^{-1}=~a^2 \rangle$.
The faithful representation $\rho : \Gamma\to\GL(2,\RR)$ sending $a$ and~$b$ respectively to 
$$ A = \begin{pmatrix} 1 & 1 \\ 0 & 1\end{pmatrix}  \ \text{ and } \ B = \begin{pmatrix} 2 & 0 \\ 0 & 1 \end{pmatrix}$$
has a weak uniform $1$-gap in eigenvalues, but does \emph{not} have a strong uniform $1$-gap in eigenvalues, nor a uniform $1$-gap in singular values.
\end{ex}

\begin{proof}
We consider the standard generating set $F=\{a,b\}$.
Any element $\gamma$ of $BS(1,2)$ can be written uniquely as $\gamma=b^{-m}a^Nb^n$ where $m,n\in\N$ and $N\in\Z$, with $mn=0$ as soon as $N$ is even; we call this the \emph{normal form} of~$\gamma$.
For any $m,n,N\in\ZZ$ we have
$$(\lambda_1-\lambda_2)(\rho(b^{-m}a^Nb^n)) = (\lambda_1-\lambda_2)(B^{-m}A^NB^n) = |n-m|\,\log 2.$$
On the other hand, there are uniform constants $C,C'>0$ such that for any $\gamma=b^{-m}a^Nb^n$ in normal form,
\begin{equation} \label{eqn:BS-length}
C^{-1}\big(m+n+\log |N|\big) - C' \leq |\gamma|_F \leq C\big(m+n+\log |N|\big) + C'
\end{equation}
(see \eg \cite[Prop.\,2.1]{be15}).
Furthermore, one easily checks by induction that for any $k\geq 1$,
$$(b^{-m}a^Nb^n)^k =
\left\{ \begin{array}{ll}
b^{-m} \, a^{N(2^{k(n-m)}-1)/(2^{n-m}-1)} \, b^{k(n-m)+m} & \text{ if }n-m\geq 0,\\
b^{-k(m-n)-n} \, a^{N(2^{k(m-n)}-1)/(2^{m-n}-1)} \, b^n & \text{ if }m-n\geq 0,
\end{array} \right.$$
hence $|(b^{-m}a^Nb^n)^k|_F \leq k \, C (1+\log 2) \, |n-m| + (C\log |N| + C (m+n) + C')$.
We deduce that the stable length satisfies, for any $\gamma=b^{-m}a^Nb^n$ in normal form,
$$|\gamma|_{F,\infty} \leq C (1+\log 2) \, |n-m|.$$
Therefore, for any $\gamma\in BS(1,2)$,
$$(\lambda_1-\lambda_2)(\rho(\gamma)) \geq \frac{\log 2}{C (1+\log 2)} \, |\gamma|_{F,\infty}.$$
This shows that $\rho$ has a weak uniform $1$-gap in eigenvalues.

On the other hand, $\rho$ does \emph{not} have a strong uniform $1$-gap in eigenvalues.
Indeed, considering normal forms and using \eqref{eqn:BS-length} we see that for any $N\in\ZZ$,
$$\ell_F(a^N) = \inf_{\beta\in\Gamma} |\beta^{-1}a^N\beta|_F = \inf_{n\in\NN} |b^{-n}a^Nb^n|_F\geq C^{-1}\,\log |N| - C',$$
which goes to infinity with~$N$, whereas $(\lambda_1-\lambda_2)(\rho(a^N))=(\lambda_1-\lambda_2)(A^N)=0$ for all $N\in\ZZ$.

Since $\Gamma$ is not word hyperbolic, Fact \ref{fact:Ano-sg-gap} implies that $\rho$ does not have a uniform $1$-gap in singular values.
\end{proof}

However, we ask the following question:

\begin{quest} \label{q:eig-gap-hyperb}
Let $\Gamma$ be a finitely generated group, $\rho : \Gamma\to\GL(d,\RR)$ a \emph{faithful and discrete} representation, and $1\leq i\leq d-1$ an integer.
If $\rho$ has a (strong) uniform $i$-gap in eigenvalues, must $\Gamma$ be word hyperbolic?
\end{quest}

This question appears to be related to the following property.

\begin{defi} \label{def:prop-U-groups}
A group~$\Gamma$ with finite generating subset~$F$ has \emph{strong} (\resp \emph{weak}) \emph{property~(U)} (for `undistorted in its conjugacy classes') if there exist a finite subset $S'$ of~$\Gamma$ and constants $c,c'>0$ such that for any $\gamma \in \Gamma$,
$$\max_{s\in S'} |s\gamma|_{F,\infty} \geq c \, |\gamma|_F - c' \quad\quad\text{(\resp }\quad \max_{s\in S'} \ell_F(s\gamma) \geq c \, |\gamma|_F - c'\text{)}.$$
\end{defi}

This does not depend on the choice of~$F$, by Remark~\ref{rem:change-gen-set}.
Strong property~(U) implies weak property~(U) by \eqref{eqn:stable-transl-length}.
Weak property~(U) was introduced in \cite{DGLM}, where it was simply called property~(U); we already mentioned it in the proof of Proposition~\ref{prop:eig-sing-gap-group}. 

\begin{ex} \label{ex:prop-U-groups}
A finitely generated group~$\Gamma$ has strong property~(U) as soon as one of the following holds:
\begin{enumerate}
  \item $\Gamma$ is abelian (in this case $|\gamma|_{F,\infty} = \ell_F(\gamma) = |\gamma|_F$ for all $\gamma\in\Gamma$),
  \item $\Gamma$ is word hyperbolic (see \cite[Prop.\,2.2.1]{DGLM} and \eqref{eqn:stable-transl-length-hyp}),
  \item\label{item:prop-U-QI-ss} there exists a representation $\rho : \Gamma\to\GL(d,\RR)$ which is a quasi-isometric embedding and such that the Zariski closure of $\rho(\Gamma)$ in $\GL(d,\RR)$ is reductive.
\end{enumerate}
\end{ex}

\begin{proof}[Proof of Example~\ref{ex:prop-U-groups}.\eqref{item:prop-U-QI-ss}]
Since $\rho$ is a quasi-isometric embedding, there exist\linebreak $c_0,c_0'>0$ such that $c_0 |\gamma|_F - c_0' \leq \Vert\mu(\rho(\gamma))\Vert \leq c_0^{-1} |\gamma|_F + c_0'$ for all $\gamma\in\Gamma$.
In particular, $|\gamma|_{F,\infty} = \lim_n |\gamma^n|_F/n \geq c_0 \, \lim_n \Vert\mu(\rho(\gamma^n))\Vert/n = c_0\,\Vert\lambda(\rho(\gamma))\Vert$ (see \eqref{eqn:lambda-mu-relation}).
By Fact~\ref{fact:Benoist}, there exist a finite subset $S$ of $\Gamma$ and $M>0$ such that for any $\gamma\in\Gamma$ we can find $s\in S$ with $\| \mu(\rho(\gamma)) - \lambda(\rho(s\gamma)) \| \leq M$; we then have
$$c_0^{-1} |s\gamma|_{F,\infty} \geq \Vert\lambda(\rho(s\gamma))\Vert \geq \Vert\mu(\rho(\gamma))\Vert - M \geq c_0 |\gamma|_F - (c'_0 + M). \qedhere$$
\end{proof}

A similar proof was given in \cite[\S\,2.1.4]{DGLM} to establish weak property~(U) in Example~\ref{ex:prop-U-groups}.\eqref{item:prop-U-QI-ss}, in the general setting of reductive Lie groups over any local field.

\begin{prop} \label{prop:strong-U-gap-hyp}
Let $\Gamma$ be a finitely generated group with weak (\resp strong) property~(U).
If some representation $\rho : \Gamma \to \GL(d,\RR)$ has a strong (\resp weak) uniform $i$-gap in eigenvalues for some $1\leq i\leq d-1$, then $\Gamma$ is word hyperbolic.
\end{prop}

\begin{proof}
Suppose $\Gamma$ has weak property~(U) and $\rho : \Gamma \to \GL(d,\RR)$ has a strong uniform $i$-gap in eigenvalues.
The same proof as that of Proposition~\ref{prop:eig-sing-gap-group} in Section~\ref{subsec:proof-eig-sing-gap-group} shows that $\rho$ has a uniform $i$-gap in singular values.
Therefore, $\Gamma$ is word hyperbolic by Fact~\ref{fact:Ano-sg-gap}.

Suppose $\Gamma$ has strong property~(U) and $\rho : \Gamma \to \GL(d,\RR)$ has a weak uniform $i$-gap in eigenvalues.
Arguing exactly as in the proof of Proposition~\ref{prop:eig-sing-gap-group} in Section~\ref{subsec:proof-eig-sing-gap-group}, but replacing the translation length $\ell_F$ by the stable length $|\cdot|_{F,\infty}$ everywhere in the fourth paragraph of the proof, we see that $\rho$ has a uniform $i$-gap in singular values.
Therefore, $\Gamma$ is word hyperbolic by Fact~\ref{fact:Ano-sg-gap}.
\end{proof}

In other words, a representation $\rho: \Gamma \to \GL(d,\RR)$ can never have a strong (\resp weak) uniform $i$-gap in eigenvalues if $\Gamma$ has weak (\resp strong) property~(U) but is not word hyperbolic.

\subsection{Proof of Theorem~\ref{thm:charact-domin-splitting-hyp-gp}} \label{sec:proof-teoC}

For the direct implication, see Fact~\ref{fact:domin->Lyap-gap}.
We now prove the converse implication.
Suppose $(\sigma,\Phi_{\rho})$ has a uniform $i$-gap of Lyapunov exponents (Definition~\ref{def:unif-gap-Lyap-exp}).
Since $\Gamma$ is word hyperbolic, it is well known (see \eg \cite[Th.\,5.1]{ef97}) that there exists $N>0$ such that for any $\gamma\in\Gamma$ of infinite order, the element $\gamma^N$ is conjugate to some element $\beta\in\Gamma$ with $|\beta|_{F,\infty} = |\beta|_F$.
Let $\sx_{\beta}\in\cG_F$ be a corresponding periodic element: namely, if $\beta=f_0\cdots f_{n-1}$ with $f_0,\dots,f_{n-1}\in F\cup F^{-1}$ and $n=|\beta|_F$, then $\sx_{\beta}=(f'_k)_{k\in\ZZ}$ with $f'_{\ell n+m}=f_m$ for all $\ell\in\ZZ$ and $0\leq m\leq n-1$.
We have
$$\Phi_{\rho}^{(n)}(\sx_{\beta}) = \Phi_{\rho}(\sigma^{n-1}(\sx_{\beta})) \cdots \Phi_{\rho}(\sx_{\beta}) = \rho(f_0 \cdots f_{n-1})^{-1} = \rho(\beta)^{-1}.$$

Suppose $(\sigma,\Phi_{\rho})$ has a uniform $i$-gap of Lyapunov exponents, \ie there exists $c>0$ such that $\chi_i(\nu) \geq \chi_{i+1}(\nu) + c$ for all ergodic $\sigma$-invariant measures $\nu$ on~$\cG_F$.
We claim that $\rho$ has a uniform $i$-gap in eigenvalues (Definition~\ref{def:gaps} and Remark \ref{rem:gaps}.\eqref{item:strong-weak-gap-hyp}).
Indeed, for $\gamma\in\Gamma$ of finite order, we have $|\gamma|_{F,\infty}=0$, and so we may restrict to $\gamma\in\Gamma$ of infinite order.
For such an element~$\gamma$, let $\beta=f_0\cdots f_{n-1}\in\Gamma$ and $\sx_{\beta}\in\cG_F$ be as above.
We first note that, using \eqref{eqn:mu-lambda-inverse},
\begin{align*}
N \, (\lambda_{d-i} - \lambda_{d-i+1})(\rho(\gamma)) & = (\lambda_{d-i} - \lambda_{d-i+1})(\rho(\gamma^N))\\
& = (\lambda_i - \lambda_{i+1})(\rho(\beta)^{-1}) = (\lambda_i - \lambda_{i+1})(\Phi_{\rho}^{(n)}(\sx_{\beta}))
\end{align*}
and $n = |\beta|_F = |\beta|_{F,\infty} = |\gamma^N|_{F,\infty} = N \, |\gamma|_{F,\infty}$.
On the other hand, applying the assumption on ergodic invariant measures to the measure $\nu_{\sx_{\beta}}$ of \eqref{eqn:nu-x} and using \eqref{eqn:Lyap-exp-periodic-meas}, we obtain
$$(\lambda_i - \lambda_{i+1})(\Phi_{\rho}^{(n)}(\sx_{\beta})) = n \, (\chi_i - \chi_{i+1})(\nu_{\sx_{\beta}}) \geq c n,$$
where $c>0$ is independent of~$\gamma$, hence $(\lambda_{d-i} - \lambda_{d-i+1})(\rho(\gamma)) \geq c n/N = c  |\gamma|_{F,\infty}$.
This proves that $\rho$ has a uniform $(d-i)$-gap in eigenvalues, hence a uniform $i$-gap in eigenvalues by Remark~\ref{rem:gaps}.\eqref{item:i-(d-i)-gap}.

By Proposition~\ref{prop:eig-sing-gap-group}, the representation $\rho$ has a uniform $i$-gap in singular values.
By Fact~\ref{fact-BG}, the linear cocycle $(\sigma,\Phi_{\rho})$ has a dominated splitting of index~$i$.

This completes the proof of Theorem~\ref{thm:charact-domin-splitting-hyp-gp}.

\begin{remark}
The subshifts $\cG_F$ associated to word hyperbolic groups as in \eqref{eqn:G-F-group} that we consider in this section are not necessarily of finite type for the alphabet~$F$, as the following example shows.
\end{remark}

\begin{ex}
Let $\Gamma$ be a group with a finite generating subset $F$ such that any reduced word in the alphabet $F\cup F^{-1}$ which is trivial in~$\Gamma$ involves at least three different elements of $F\cup F^{-1}$.
For instance, we can take $\Gamma$ to be the fundamental group of a closed surface of genus $g\geq 2$, with its standard generating subset $F=\{a_1,b_1,\dots,a_g,b_g\}$ subject to the relation $[a_1,b_1]\cdots [a_g,b_g]=e$.
We assume that $\Gamma$ is not the free group generated by~$F$.
Suppose by contradiction that the subshift $\cG_F$ is of finite type, \ie there is a square matrix $A =(a_{f,f'})_{f,f'\in F \cup F^{-1}}$ of $0$'s and $1$'s such that
$$\cG_F = \big\{ (f_k)_{k\in\ZZ} \in (F \cup F^{-1})^{\ZZ} \ : a_{f_k,f_{k+1}} = 1 \quad \forall k\in\ZZ\big\}.$$
Let $R=f_0\cdots f_{n-1}$ be a cyclically reduced word in the alphabet $F\cup F^{-1}$ which is trivial in~$\Gamma$.
Since any reduced word which is trivial in~$\Gamma$ involves at least three different elements of $F\cup F^{-1}$, for any $0 \leq j \leq n-1$ and $k\in\NN$ we have $|(f_j f_{j+1})^k|_F=2k$ (where we set $f_n:=f_0$) and so $a_{f_j,f_{j+1}}=1$. 
Therefore the element $(f'_k)\in (F \cup F^{-1})^{\ZZ}$ with $f'_{\ell n+m}=f_m$ for all $\ell\in\ZZ$ and $0\leq m\leq n-1$ belongs to~$\cG_F$.
But $R=f_0\cdots f_{n-1}=e$: contradiction.
\end{ex}

In general, as explained in \cite[\S\,5]{BPS}, for any word hyperbolic group~$\Gamma$ we can find a finite generating subset $F$ such that $\cG_F$ is what is called a \emph{sofic subshift}.
Sofic subshifts are a mild generalization of subshifts of finite type, which also have the specification property of Fact~\ref{fact-specification}.

\section{Anosov representations for semigroups} \label{sec:Ano-semig}

In this section, we consider the notion of an Anosov semigroup homomorphism from Definition~\ref{def:Anosov-semig}.
We investigate a few basic aspects of it, namely the existence of a boundary map (Lemma~\ref{lem:semig-Ano-bound-map}) defined on the so-called \emph{quasigeodesic boundary} of the semigroup (Section~\ref{subsec:semig-QG-bound}), the openness of the set of Anosov semigroup homomorphisms (Corollary~\ref{cor:Ano-open-semig}), and the link with eigenvalue gaps (Proposition~\ref{prop:eig-gap-semig}).
We point out some differences with the group case (Section~\ref{subsec:semig-vs-gp}), and focus in particular on one important class of semigroups, namely completely simple ones (Section~\ref{subsec:compl-simple-semig}).
Our goal here is not to be exhaustive but to lay some foundations; we hope to develop a more thorough theory of Anosov semigroup homomorphisms in further work.

\subsection{Definitions} \label{subsec:semig-def-Anosov}

Let $\semig$ be a semigroup with a finite generating subset~$F$, \ie any element of~$\semig$ can be written as a product of elements of~$F$.
The semigroup $\semig$ may admit an identity element~$e$ (in which case it is called a monoid) or not.
For $\gamma \in\nolinebreak \semig \smallsetminus\nolinebreak \{e\}$ we define the \emph{word length} of~$\gamma$ to be
$$|\gamma|_F = \min \{ k\in\NN^* \ : \ \exists f_0,\dots,f_{k-1}\in F\ \mathrm{s.t.}\ \gamma=f_0\cdots f_{k-1}\}$$
and the \emph{stable length} of~$\gamma$ to be $|\gamma|_{F,\infty} = \lim_n |\gamma^n|/n$; we also set $|e|_F = |e|_{F,\infty} = 0$.
Note that there is no clear a priori definition of translation length.
For any $\gamma,\gamma'\in\semig$, we have $|\gamma\gamma'|_F \leq |\gamma|_F + |\gamma'|_F$.
The following holds similarly to Remark~\ref{rem:change-gen-set}.

\begin{remark} \label{rem:change-gen-set-semig}
If $F'$ is another finite generating subset of~$\semig$, then there exists $M \geq 1$ such that $M^{-1} \, |\gamma|_{F'} \leq |\gamma|_F \leq M \, |\gamma|_{F'}$ and $M^{-1} \, |\gamma|_{F',\infty} \leq |\gamma|_{F,\infty} \leq M \, |\gamma|_{F',\infty}$ for all $\gamma\in\semig$.
\end{remark}

Similarly to Definition~\ref{def:gaps}, we introduce the following terminology.

\begin{defi} \label{def:gaps-semig}
For $1\leq i\leq d-1$, a semigroup homomorphism $\rho : \semig\to\GL(d,\RR)$ has
\begin{itemize}
  \item a \emph{uniform $i$-gap in singular values} if there exist $c,c'>0$ such that
  \begin{equation} \label{eqn:sv-gap-semig}
  (\mu_i - \mu_{i+1})(\rho(\gamma)) \geq c \, |\gamma|_F - c'  \ \ \ \ \ \ \forall \gamma \in \semig ,
  \end{equation}
  \item a \emph{uniform $i$-gap in eigenvalues} if there exist $c,c'>0$ such that
  \[(\lambda_i - \lambda_{i+1})(\rho(\gamma)) \geq c |\gamma|_{F,\infty} - c'  \ \ \ \ \ \ \forall \gamma \in \semig . \]
\end{itemize}
\end{defi}

These notions do not depend on the choice of finite generating set~$F$, by Remark~\ref{rem:change-gen-set-semig}.
As in the introduction (Definition~\ref{def:Anosov-semig}), we say that a semigroup homomorphism $\rho : \semig\to\GL(d,\RR)$ is \emph{$P_i$-Anosov} if it has a uniform $i$-gap in singular values.
By an \emph{Anosov representation} we mean a semigroup homomorphism which is $P_i$-Anosov for some $1\leq i\leq d-1$.

\begin{remark} \label{rem:Ano-discrete}
If $\rho : \semig\to\GL(d,\RR)$ is $P_i$-Anosov, then its image is discrete in $\GL(d,\RR)$, in the sense that its intersection with any compact subset of $\GL(d,\RR)$ is at most finite: indeed, the function $\mu_i-\mu_{i+1}$ is continuous on $\GL(d,\RR)$, hence bounded on compact sets, and the uniform $i$-gap in singular values implies that for any $R>0$ the set of $\gamma\in\semig$ with $(\mu_i-\mu_{i+1})(\rho(\gamma))\leq R$ is finite.
Moreover, if we set $C := c/\sqrt{2} > 0$ and $C' := c'/\sqrt{2} > 0$ with $(c,c')$ as in \eqref{eqn:sv-gap-semig}, then
\begin{equation} \label{eqn:QI-semig}
\Vert\mu(\rho(\gamma)\Vert \geq C \, |\gamma|_F - C'  \ \ \ \ \ \ \forall \gamma \in \semig .
\end{equation}
\end{remark}

\begin{remark} \label{rem:Ano-GL-SL}
A semigroup homomorphism $\rho : \semig\to\GL(d,\RR)$ is $P_i$-Anosov if and only if the following semigroup homomorphism is:
\begin{equation} \label{eqn:rho-1}
\begin{array}{cccc}
\rho_1 : & \semig & \longrightarrow & \SL^{\pm}(d,\RR)\\
& \gamma & \longmapsto &\rho(\gamma)/|\det(\rho(\gamma))|^{1/d}.
\end{array}
\end{equation}
Indeed, $(\mu_i - \mu_{i+1})(\rho(\gamma)) = (\mu_i - \mu_{i+1})(\rho_1(\gamma))$ for all $\gamma \in \semig$.
\end{remark}

\subsection{Anosov semigroup homomorphisms and dominated splittings}

Similarly to \eqref{eqn:G-F-group}, we set
$$\cG_F := \big\{ (f_k) \in F^{\NN} \ : \ |f_k \cdots f_{k+\ell}|_F = \ell+1 \quad \forall k,\ell\in\NN\big\}.$$
For $\kappa\geq 1$ and $\kappa'\geq 0$, we also set
\begin{equation} \label{eqn:def-QG-kappa-kappa'}
\cQ\cG_F^{\kappa,\kappa'} := \big\{ (f_k) \in F^{\NN} \ : \ |f_k \cdots f_{k+\ell}|_F \geq \kappa^{-1} (\ell+1) - \kappa' \quad \forall k,\ell\in\NN\big\}.
\end{equation}
Then $\cG_F$ and $\cQ\cG_F^{\kappa,\kappa'}$ are closed subsets of $F^{\NN}$, and for $\hat \kappa \geq \kappa$ and $\hat \kappa' \geq \kappa'$ we have $\cG_F = \cQ\cG_F^{1,0} \subset \cQ\cG_F^{\kappa, \kappa'} \subset \cQ\cG_F^{\hat \kappa, \hat \kappa'}$.
We set
\begin{equation} \label{eqn:def-QG-F}
\cQ\cG_F := \bigcup_{\kappa \geq 1,\,\kappa'\geq 0} \cQ\cG_F^{\kappa,\kappa'}.
\end{equation}
See Corollary~\ref{cor:QG-F-semigroup} below for a geometric interpretation of the set $\cQ\cG_F$ in terms of quasigeodesic rays in the Cayley graph $\mathrm{Cay}(\semig,F)$.

Let $\sigma : \cQ\cG_F\to\cQ\cG_F$ be the shift, sending $\sx=(f_k)_{k\in\NN}$ to $\sy=(f'_k)_{k\in\NN}$ where $f'_k=f_{k+1}$ for all $k\in\NN$.
For any semigroup homomorphism $\rho : \semig \to \GL(d,\RR)$, we define a map $\Phi_{\rho} : \cQ\cG_F \to \GL(d,\RR)$ by
\begin{equation} \label{eqn:cocycle}
 \Phi_{\rho}((f_k)_{k\in\NN}) := \rho(f_0)^{-1}. 
 \end{equation}
 As in Section~\ref{sec:proof-teoC}, for any $n\in\NN^*$ and $(f_k)_{k\in\NN}\in\cQ\cG_F$, using the notation \eqref{eqn:Phi-(n)} with $T=\sigma$, we have
\begin{equation} \label{eqn:Phi-rho-n-semig}
\Phi_{\rho}^{(n)}((f_k)_{k\in\NN}) = \rho(f_0 \cdots f_{n-1})^{-1}.
\end{equation}

\begin{prop} \label{prop:Ano-sv-gap-semig}
Let $\semig$ be a semigroup with finite generating subset~$F$, let $1\leq i<d$ be integers, and let $\rho : \semig\to\GL(d,\RR)$ be a $P_i$-Anosov semigroup homomorphism.
Then
\begin{enumerate}
  \item\label{item:Ano-sv-gap-semig-1} for any $\kappa \geq 1$ and $\kappa'\geq 0$ the linear cocycle $(\sigma, \Phi_{\rho})$ over $\cQ\cG_F^{\kappa,\kappa'}$ has a dominated splitting of index $d-i$;
  \item\label{item:Ano-sv-gap-semig-2} for any $\kappa \geq 1$ and $\kappa'\geq 0$, there exist $C'', C''' > 0$ such that the $P_i$-Anosov semigroup homomorphism $\rho_1 : \semig\to\SL^{\pm}(d,\RR)$ of Remark~\ref{rem:Ano-GL-SL} satisfies that for any $n,m\in\NN^*$ and any $(f_k)_{k\in\NN} \in \cQ\cG_F^{\kappa,\kappa'}$,
  $$\big\Vert \mu(\rho_1(f_0\cdots f_{n+m-1})) - \mu(\rho_1(f_0\cdots f_{n-1})) \big\Vert \geq C'' m - C''' \, ;$$
  \item\label{item:Ano-sv-gap-semig-3} the map
  \begin{eqnarray*}
  E^{cs} :\hspace{0.5cm} \cQ\cG_F\, & \longrightarrow & \hspace{1cm}\mathrm{Gr}_i(\RR^d)\\
  (f_k)_{k\in\NN} & \longmapsto & \lim_{n\to +\infty} \Xi_i\big(\rho(f_0 \cdots f_{n-1})\big)
  \end{eqnarray*}
  is well defined and $(\sigma,\Phi_{\rho})$-equivariant; its restriction to $\cQ\cG_F^{\kappa,\kappa'}$ is continuous for all $\kappa \geq 1$ and $\kappa'\geq 0$.
\end{enumerate}
\end{prop}

\begin{proof}
By assumption, there exist $c,c'>0$ such that $(\mu_i - \mu_{i+1})(\rho(\gamma)) \geq c \, |\gamma|_F - c'$ for all $\gamma \in \semig$.
Fix $\kappa\geq 1$ and $\kappa'\geq 0$.
Using \eqref{eqn:Phi-rho-n-semig} and \eqref{eqn:mu-lambda-inverse}, we see that for any $n\in\NN^*$ and any $(f_k)_{k\in\NN} \in \cQ\cG_F^{\kappa,\kappa'}$,
\begin{eqnarray*}
(\mu_{d-i} - \mu_{d-i-1})\big(\Phi_{\rho}^{(n)}(\sx)\big) & = & (\mu_{d-i} - \mu_{d-i-1})(\rho(f_0 \cdots f_{n-1})^{-1})\\
& = & (\mu_i - \mu_{i+1})(\rho(f_0 \cdots f_{n-1}))\\
& \geq & c \, |f_0 \cdots f_{n-1}|_F - c'\\
& \geq & c \kappa^{-1} \, n - (c \kappa' + c').
\end{eqnarray*}
By Fact~\ref{fact-BG}, the linear cocycle $(\sigma, \Phi_{\rho})$ over $\cQ\cG_F^{\kappa,\kappa'}$ has a dominated splitting of index $d-i$ and the map $E^{cs} : \cQ\cG_F^{\kappa,\kappa'} \to \mathrm{Gr}_i(\RR^d)$ sending $(f_k)_{k\in\NN} $ to $\lim_n \Xi_i(\rho(f_0 \cdots f_{n-1}))$ is well defined, continuous, and $(\sigma,\Phi_{\rho})$-equivariant.
We thus obtain \eqref{item:Ano-sv-gap-semig-1}, as well as \eqref{item:Ano-sv-gap-semig-3} since $\cQ\cG_F$ is the union of the $\cQ\cG_F^{\kappa,\kappa'}$ over all pairs $(\kappa,\kappa')$.
Finally, \eqref{item:Ano-sv-gap-semig-2} is a direct consequence of Lemma~\ref{lem:domin-CLI}.\eqref{item:domin-CLI} together with \eqref{eqn:Phi-rho-n-semig} and \eqref{eqn:mu-lambda-inverse}.
\end{proof}
                     
We shall see that for given $\kappa\geq 1$ and $\kappa'\geq 0$, a converse to Proposition~\ref{prop:Ano-sv-gap-semig}.\eqref{item:Ano-sv-gap-semig-1} holds (Proposition~\ref{prop:sv-gap-semig-Ano}) under some condition which we call property~(D).

\subsection{The quasigeodesic boundary} \label{subsec:semig-QG-bound}

We now introduce a space $\partial \semig$ on which boundary maps for Anosov representations of~$\semig$ will later be defined (see Section~\ref{subsec:semig-bound-map}).

\subsubsection{The Cayley graph}

Recall that the semigroup $\semig$ may admit an identity element~$e$ (in which case it is called a monoid) or not.
We denote by $\semig^e$ the monoid obtained from~$\semig$ by possibly adding an identity element~$e$.
As in \cite{dg04}, we define the Cayley graph $\mathrm{Cay}(\semig,F)$ to have vertices $\semig^e$, with a directed edge from $\gamma$ to $\gamma f$ for every $\gamma\in\semig^e$ and $f\in F$.
Each edge is given a local metric in which it has unit length, and $\mathrm{Cay}(\semig,F)$ is turned into a metric space by defining the distance $d_F(x,y)$ between two points $x,y$ to be the length of the shortest undirected path joining them.
Note that for any $\gamma,\gamma'\in\semig$,
\begin{equation} \label{eqn:bound-dF}
d_F(\gamma\gamma',\gamma\gamma'') \leq d_F(\gamma',\gamma'') \quad\quad\mathrm{and}\quad\quad d_F(\gamma\gamma',\gamma) \leq |\gamma'|_F.
\end{equation}
The inequalities may be strict, as in the following example.

\begin{ex} \label{ex:semig-strict-ineq}
Let $\semig$ be the monoid obtained by taking the quotient of the free monoid $\semig_{\{a,b\}}$ on the alphabet $F=\{a,b\}$ by the congruence generated by $\mathcal{R} := \{ (ab^3, a)\} \subset \semig_{\{a,b\}} \times \semig_{\{a,b\}}$ (in other words, $\semig$ is the largest quotient of $\semig_{\{a,b\}}$ in which $ab^3 = a$: see \cite[\S\,1.5]{how95}).
Then $d_F(ab^3,a) = 0 < d_F(b^3,e) = |b^3|_F = 3$.
\end{ex}

Remark~\ref{rem:change-gen-set-semig} implies the following.

\begin{remark} \label{rem:change-gen-set-semig-Cay}
If $F'$ is another finite generating subset of~$\semig$, then there exists $M \geq 1$ (given by Remark~\ref{rem:change-gen-set-semig}) such that $M^{-1} \, d_{F'}(\gamma_1,\gamma_2) \leq d_F(\gamma_1,\gamma_2) \leq M \, d_{F'}(\gamma_1,\gamma_2)$ for all $\gamma_1,\gamma_2\in\semig$.
In particular, $\mathrm{Cay}(\semig,F)$ and $\mathrm{Cay}(\semig,F')$ are quasi-isometric.
\end{remark}

\subsubsection{Quasigeodesic rays}

For $\kappa\geq 1$ and $\kappa'\geq 0$, a path $(\gamma_n) \in \semig^{\NN^*}$ is said to be a \emph{$(\kappa,\kappa')$-quasigeodesic ray in $\mathrm{Cay}(\semig,F)$} if for any $k,\ell\in\NN^*$,
\begin{equation} \label{eqn:quasigeod}
\kappa^{-1}\ell - \kappa' \leq d_F(\gamma_k,\gamma_{k+\ell}) \leq \kappa\ell + \kappa'.
\end{equation}
Note that a $(\kappa,\kappa')$-quasigeodesic ray is also a $(\hat \kappa, \hat \kappa')$-quasigeodesic ray for any $\hat \kappa \geq \kappa$ and $\hat \kappa' \geq \kappa'$.
In the sequel we will be interested in paths $(\gamma_n)_{n\in\NN^*}$ satisfying $\gamma_{n+1} \in \gamma_n F$ for all $n\in\NN^*$, \ie there is a directed edge from $\gamma_n$ to $\gamma_{n+1}$ in $\mathrm{Cay}(\semig,F)$ for all $n\in\NN^*$; the right-hand inequality in \eqref{eqn:quasigeod} is then automatically satisfied.

\subsubsection{The path map $\mathcal{P}$} \label{subsubsec:QG-rays-Cay}

Consider the map $\mathcal{P} : \semig^{\NN} \to \semig^{\NN^*}$ defined by
\begin{equation} \label{eqn:def-P}
\mathcal{P}((g_k)_{k\in\NN}) := (g_0 \cdots g_{n-1})_{n\in\NN^*}.
\end{equation}

\begin{remark} \label{rem:QG-group}
Suppose $\semig$ is a group and $F$ is symmetric (\ie $F = F^{-1}$), and view $\semig$ as a semigroup with finite generating subset~$F$.
Then for any $\kappa\geq 1$ and $\kappa'\geq 0$, and any $(f_k)_{k\in\NN}\in F^{\NN}$, the path $(\gamma_n)_{n\in\NN^*} = \mathcal{P}((f_k)_{k\in\NN})$ is a $(\kappa,\kappa')$-quasigeodesic ray in $\mathrm{Cay}(\semig,F)$ as in \eqref{eqn:quasigeod} if and only if $(f_k)_{k\in\NN} \in \cQ\cG_F^{\kappa,\kappa'}$ as in \eqref{eqn:def-QG-kappa-kappa'}.
Indeed, in this case $d_F(\gamma_k, \gamma_{k+\ell}) = d_F(e,\gamma_k^{-1} \gamma_{k+\ell}) = |f_k \cdots f_{k+\ell-1}|_F$ for all $k,\ell\in\NN$.
\end{remark}

For semigroups $\semig$ which are not groups, quasigeodesic rays are less immediate to describe, as the equality $d_F(\gamma_k, \gamma_{k+\ell}) = |f_k \cdots f_{k+\ell-1}|_F$ in Remark~\ref{rem:QG-group} is only an inequality in general (see \eqref{eqn:bound-dF}), which may be strict.
Still, the following holds.

\begin{lema} \label{lem:QG-semig-Cayley-graph}
Let $\semig$ be a semigroup with finite generating subset~$F$.
Let $\kappa \geq 1$ and $\kappa'\geq 0$.
\begin{enumerate}
  \item\label{item:QG-semig-1} For any $(f_k)_{k\in\NN}\in F^{\NN}$, if $\mathcal{P}((f_k)_{k\in\NN})$ is a $(\kappa,\kappa')$-quasigeodesic ray in $\mathrm{Cay}(\semig,F)$, then $(f_k)_{k\in\NN}\in\cQ\cG_F^{\kappa,\kappa'}$.
  \item\label{item:QG-semig-2} Conversely, suppose that $\semig$ admits an Anosov representation, or more generally (Proposition~\ref{prop:Ano-sv-gap-semig}.\eqref{item:Ano-sv-gap-semig-2}) that there exist a semigroup homomorphism\linebreak $\rho : \semig\to\GL(d,\RR)$ and constants $C'',C'''>0$ such that
  \begin{equation} \label{eqn:Ano-sv-gap-semig-mu}
  \big\Vert \mu(\rho(f_0\cdots f_{n+m-1})) - \mu(\rho(f_0\cdots f_{n-1})) \big\Vert \geq C'' m - C'''
  \end{equation}
  for all $n,m\in\NN^*$ and $(f_k)_{k\in\NN} \in \cQ\cG_F^{\kappa,\kappa'}$.
  Then there exist $\kappa''\geq 1$ and $\kappa'''\geq 0$ with the following property: for any $(f_k)_{k\in\NN} \in F^{\NN}$, if $(f_k)_{k\in\NN} \in \cQ\cG_F^{\kappa,\kappa'}$, then $\mathcal{P}((f_k)_{k\in\NN})$ is a $(\kappa'',\kappa''')$-quasigeodesic ray in $\mathrm{Cay}(\semig,F)$.
\end{enumerate}
\end{lema}

Here is an immediate consequence of Lemma~\ref{lem:QG-semig-Cayley-graph}.

\begin{cor} \label{cor:QG-F-semigroup}
Let $\semig$ be a semigroup with finite generating subset~$F$.
Suppose that $\semig$ admits an Anosov representation $\rho : \semig\to\GL(d,\RR)$.
Then $\mathcal{P}(\cQ\cG_F)$ is the set of quasigeodesic rays in $\mathrm{Cay}(\semig,F)$ of the form $(\gamma_n)_{n\in\NN^*} \in \semig^{\NN^*}$ with $\gamma_1 \in F$ and $\gamma_{n+1} \in \gamma_n F$ for all $n\in\NN^*$.
\end{cor}

Lemma~\ref{lem:QG-semig-Cayley-graph}.\eqref{item:QG-semig-2} relies on the following observation.

\begin{remark} \label{rem:mu-dist-Cayley-graph}
Let $\rho : \semig\to\GL(d,\RR)$ be a semigroup homomorphism.
For any $\gamma,\gamma'\in\nolinebreak\semig$, we have $\| \mu(\rho(\gamma)) - \mu(\rho(\gamma')) \| \leq d_F(\gamma,\gamma') \, \max_{f\in F} \|\mu(\rho(f))\|$.
Indeed, setting $R := d_F(\gamma,\gamma') \in \NN$, there exists a path $(\gamma_j)_{0\leq j\leq R}$ in $\mathrm{Cay}(\semig,F)$ with $\gamma_0 = \gamma$ and $\gamma_R = \gamma'$, and with $\gamma_{j+1} \in \gamma_j F$ or $\gamma_j \in \gamma_{j+1} F$ for all $0\leq j\leq R-1$.
By \eqref{eqn:mu-subadd} we have $\Vert\mu(\rho(\gamma_{j+1})) - \mu(\rho(\gamma_j))\Vert \leq \max_{f\in F} \|\mu(\rho(f))\|$ for all~$j$, and we conclude by triangle inequality.
\end{remark}

\begin{proof}[Proof of Lemma~\ref{lem:QG-semig-Cayley-graph}]
\eqref{item:QG-semig-1} Consider $(f_k)_{k\in\NN}\in F^{\NN}$ and $(\gamma_n)_{n\in\NN^*} := \mathcal{P}((f_k)_{k\in\NN}) \in \semig^{\NN^*}$.
By \eqref{eqn:bound-dF}, for any $k,\ell\in\NN^*$ we have
$$d_F(\gamma_k,\gamma_{k+\ell}) = d_F(\gamma_k, \gamma_k f_k\cdots f_{k+\ell-1}) \leq |f_k\cdots f_{k+\ell-1}|_F.$$
Therefore, if $(\gamma_n)_{n\in\NN^*}$ satisfies \eqref{eqn:quasigeod}, then $(f_k)_{k\in\NN}$ belongs to the set $\cQ\cG_F^{\kappa,\kappa'}$ of \eqref{eqn:def-QG-kappa-kappa'}.

\eqref{item:QG-semig-2} Suppose there exist $\rho : \semig\to\GL(d,\RR)$ and $C'',C'''>0$ such that \eqref{eqn:Ano-sv-gap-semig-mu} holds for all $n,m\in\NN^*$ and $(f_k)_{k\in\NN} \in \cQ\cG_F^{\kappa,\kappa'}$.
We may assume $C''\leq 1$.
Let $M:=\max_{f\in F} \Vert\mu(\rho(f))\Vert$.
By \eqref{eqn:Ano-sv-gap-semig-mu} and Remark~\ref{rem:mu-dist-Cayley-graph}, for any $(f_k)_{k\in\NN} \in \cQ\cG_F^{\kappa,\kappa'}$, setting $(\gamma_n)_{n\in\NN^*} := \mathcal{P}((f_k)_{k\in\NN}) \in \semig^{\NN^*}$, we have
$$C'' \ell - C''' \leq  \big\Vert\mu\big(\rho(\gamma_{k+\ell})\big) - \mu\big(\rho(\gamma_k)\big)\big\Vert \leq M \, d_F(\gamma_k, \gamma_{k+\ell})$$
for all $k,\ell\in\NN^*$, hence $(\gamma_n)_{n\in\NN^*}$ satisfies \eqref{eqn:quasigeod} with $(\kappa,\kappa')$ replaced by $(\kappa'', \kappa''') := ({C''}^{-1} M, C''' M^{-1})$.
\end{proof}

Lemma~\ref{lem:QG-semig-Cayley-graph}.\eqref{item:QG-semig-2} provides a simple obstruction for the existence of Anosov representations of semigroups as in the example below.
Other obstructions follow from Section~\ref{subsec:compl-simple-semig} for the class of completely simple semigroups. 

\begin{ex} \label{ex:semig-no-Ano-rep}
Let $\semig$ be the monoid obtained by taking the quotient of the free monoid $\semig_{\{a,b,c\}}$ on the alphabet $F=\{a,b,c\}$ by the congruence generated by $\mathcal{R} := \{ ((bc^n)^na, ac) \,:\, n\in\NN^*\} \subset \semig_{\{a,b,c\}} \times \semig_{\{a,b,c\}}$.
Note that the elements $b$ and~$c$ generate a free subsemigroup of~$\semig$.
Consider $(f_k)_{k\in\NN}\in\{b,c\}^{\NN}\subset F^{\NN}$ such that the elements $\gamma_k = f_{0}  \cdots f_{k-1}$ satisfy $\gamma_{\ell_n} = bc (bc^2)^2 (bc^3)^3 \cdots (bc^n)^n$ whenever $\ell_n = 2 + 3^2 + 4^3 + \ldots + (n+1)^n$.
Then $(f_k)_{k \in \NN} \in \cG_F = \cQ\cG^{1,0}_F$.
On the other hand, $\gamma_{\ell_n}a = a c^n$, and so the triangle inequality and \eqref{eqn:bound-dF} imply $d_F(e,\gamma_{\ell_n}) \leq  d_F(e, \gamma_{\ell_n} a) + d_F(\gamma_{\ell_n} a, \gamma_{\ell_n}) \leq d_F(e, a c^n) + d_F(a, e) \leq n+2$, which implies that $(\gamma_n)_{n\in\NN^*} := \mathcal{P}((f_k)_{k\in\NN}) \in \semig^{\NN^*}$ is not a quasigeodesic ray in $\mathrm{Cay}(\semig,F)$.
By Lemma~\ref{lem:QG-semig-Cayley-graph}.\eqref{item:QG-semig-2}, the semigroup $\semig$ cannot admit an Anosov representation.
In fact, $\semig$ cannot even admit a representation $\rho : \semig\to\GL(d,\RR)$ satisfying \eqref{eqn:QI-semig} for some $C,C'>0$, since $|\gamma_{\ell_n}a|_F = |ac^n|_F \leq n+1$ grows much slower with~$n$ than $|\gamma_{\ell_n}|_F = \ell_n$, and so the left-hand side of \eqref{eqn:QI-embed-ineq-length} in Remark~\ref{rem:QI-embed-ineq-length} below cannot be satisfied for $(\gamma',\gamma'') = (\gamma_{\ell_n},a)$ for all~$n$.
\end{ex}

\subsubsection{An equivalence relation on $\semig^{\NN^*}$}

We shall say that $(\gamma_n)_{n\in\NN^*}, (\gamma'_n)_{n\in\NN^*} \in \semig^{\NN^*}$ \emph{remain at bounded distance in the Cayley graph} if there exists $N>0$ such that for any $n\in\NN^*$ we can find $\varphi_n,\varphi'_n\in\NN^*$ with $d_F(\gamma_n, \gamma'_{\varphi_n})\leq N$ and $d_F(\gamma'_n, \gamma_{\varphi'_n})\leq N$. 
One readily checks that this defines an equivalence relation $\sim$ on $\semig^{\NN^*}$.
It does not depend on the choice of finite generating subset~$F$, by Remark~\ref{rem:change-gen-set-semig-Cay}.

\subsubsection{The quasigeodesic boundary}

Recall the path map $\mathcal{P} : \semig^{\NN} \to \semig^{\NN^*}$ from Section~\ref{subsubsec:QG-rays-Cay}.
We define the \emph{quasigeodesic boundary} of~$\semig$ to be
\begin{equation} \label{eqn:def-bound-semig}
\partial \semig := \mathcal{P}(\cQ\cG_F)/_{\sim}.
\end{equation}

\begin{lema} \label{lem:QG-bound-semig-F-F'}
The quasigeodesic boundary $\partial \semig$ does not depend on the choice of finite generating subset $F$ of~$\semig$.
\end{lema}

\begin{proof}
By symmetry, it is sufficient to check that if $F'$ is another finite generating subset of~$\semig$, then for any $(f_k)_{k\in\NN} \in \cQ\cG_F$, there exists $(f'_a)_{a\in\NN} \in \cQ\cG_{F'}$ such that $\mathcal{P}((f_k)_{k\in\NN}) \sim \mathcal{P}((f'_a)_{a\in\NN})$.
For this we choose, for each $f\in F$, an expression $f = f'_{(f,0)}\cdots f'_{(f,m_f-1)}$ of $f$ as a product of elements $f'_{(f,i)}$ of~$F'$.

We start with the following observation: for $(f_k)_{k\in\NN} \in F^{\NN}$, let $(f'_a)_{a\in\NN}\in {F'}^{\NN}$ be given by $f'_{m_{f_0}+\dots+m_{f_{k-1}}+i} := f'_{(f_k,i)}$ for all $k\in\NN$ and $0\leq i<m_{f_k}$.
Then for any $k,\ell\in\NN$,
\begin{equation} \label{eqn:f-f'}
f_k\cdots f_{k+\ell} = f'_{m_{f_0}+\dots+m_{f_{k-1}}} \cdots f'_{m_{f_0}+\cdots+m_{f_{k+\ell}}-1}.
\end{equation}
In particular, using \eqref{eqn:bound-dF} and the fact that the $m_f$ are uniformly bounded, we see that $\mathcal{P}((f_k)_{k\in\NN}) \sim \mathcal{P}((f'_a)_{a\in\NN})$.

It remains to check that if $(f_k)_{k\in\NN} \in \cQ\cG_F$, then $(f'_a)_{a\in\NN}\in\cQ\cG_{F'}$.
Let $M := \max(\max_{f\in F} m_f, \max_{f'\in F'} |f'|_F) \geq 1$.
We have $|\gamma|_F \leq M |\gamma|_{F'}$ for all $\gamma\in\semig$.
We claim that if $(f_k)_{k\in\NN} \in \cQ\cG_F^{\kappa,\kappa'}$, where $\kappa\geq 1$ and $\kappa'\geq 0$, then $(f'_a)_{a\in\NN}\in\cQ\cG_{F'}^{\kappa'',\kappa'''}$ where $\kappa'' = M^2\kappa$ and $\kappa''' = M^{-1}\kappa' + 2M$.
Indeed, let $a,b\in\NN$.
We can write $a = m_{f_0} + \dots + m_{f_{k-1}} + i$ and $a+b = m_{f_0} + \dots + m_{f_{k+\ell-1}} + j$ for some $k,\ell\in\NN$ and $0\leq i<m_{f_k}$ and $0\leq j<m_{f_{k+\ell}}$.
In particular,
$$b+1 = m_{f_k} + \dots + m_{f_{k+\ell-1}} + j+1 - i \leq (\ell+1)\,M.$$
Using \eqref{eqn:f-f'}, we see that
$$|f_k\cdots f_{k+\ell}|_{F'} \leq i + |f'_a\cdots f'_{a+b}|_{F'} + (m_{f_{k+\ell}} - j) \leq |f'_a\cdots f'_{a+b}|_{F'} + 2M,$$
hence
\begin{eqnarray*}
|f'_a\cdots f'_{a+b}|_{F'} & \geq & M^{-1} |f_k\cdots f_{k+\ell}|_F - 2M\\
& \geq & M^{-1} \big(\kappa^{-1} (\ell+1) - \kappa'\big) - 2M\\
& \geq & {\kappa''}^{-1} (b+1) - \kappa'''.
\end{eqnarray*}
This shows that $(f'_a)_{a\in\NN}\in\cQ\cG_{F'}^{\kappa'',\kappa'''}$.
\end{proof}

Recall from Remark~\ref{rem:change-gen-set-semig} that for $\gamma\in\semig$, the condition $|\gamma|_{F,\infty}>0$ does not depend on the choice of finite generating subset $F$ of~$\semig$ (even though the actual value of $|\gamma|_{F,\infty}$ does); we shall write this condition as $|\gamma|_{\infty}>0$.

\begin{lema} \label{lem:eta-gamma+}
Let $\semig$ be a finitely generated semigroup.
For any $\gamma\in\semig$ with $|\gamma|_{\infty}>0$, the element $\eta_{\gamma}^+ := [(\gamma^n)_{n\in\NN^*}] \in \semig^{\NN^*}\!/_{\sim}$ belongs to $\partial \semig$.
\end{lema}

\begin{proof}
Let $F$ be a finite generating subset of~$\semig$.
Consider $\gamma\in\semig$ with $|\gamma|_{F,\infty}>0$: there exist $\kappa \geq 1$ and $\kappa' \geq 0$ such that $|\gamma^n|_F \geq \kappa^{-1}\,n - \kappa'$ for all $n\in\NN$.
Write $\gamma=f^{\gamma}_0\cdots f^{\gamma}_{N-1}$ for some $f^{\gamma}_i\in F$, and consider $(f_k)_{k\in\NN} \in F^{\NN}$ such that $f_k := f^{\gamma}_i$ whenever $k \equiv i$ modulo~$N$.
Then $f_0\cdots f_{nN-1} = \gamma^n$ for all $n\in\NN^*$, hence, using \eqref{eqn:bound-dF}, we see that $\mathcal{P}((f_k)_{k\in\NN}) \sim (\gamma^n)_{n\in\NN^*}$.
It remains to show that $(f_k)_{k\in\NN} \in \cQ\cG_F$.
For this, for $k,\ell\in\NN$, write $k = mN+i$ and $k+\ell = (m+n)N+j$ where $m,n\in\NN$ and $0\leq i,j<N$.
We have $\ell+1 = nN + j+1 - i \leq (n+1)\,N$ and
$$|\gamma^{n+1}|_F = |f_{mN} \cdots f_{(m+n+1)N-1}|_F \leq i + |f_k \cdots f_{k+\ell}|_F + (N-1-j) \leq |f_k \cdots f_{k+\ell}|_F + 2N,$$
hence
$$|f_k\cdots f_{k+\ell}|_F \geq |\gamma^{n+1}|_F - 2N \geq \kappa^{-1}\,(n+1) - (\kappa'+2N) \geq {\kappa''}^{-1} (\ell+1) - \kappa'''$$
where $\kappa'' := \kappa N$ and $\kappa''' := \kappa'+2N$.
This shows that $(f_k)_{k\in\NN}\in\cQ\cG_{F'}^{\kappa'',\kappa'''}$.
\end{proof}

\subsubsection{Action of $\semig$ on the quasigeodesic boundary}

The semigroup $\semig$ naturally acts on $\semig^{\NN^*}$ by left multiplication: $\gamma \cdot (\gamma_n)_{n\in\NN^*} := (\gamma\gamma_n)_{n\in\NN^*}$.
By \eqref{eqn:bound-dF}, this action factors through the equivalence relation~$\sim$.

In order to obtain an action of $\semig$ on $\partial \semig$, we would like the set $\mathcal{P}(\cQ\cG_F)$ to be stable under the action of~$\semig$.
This is always the case if $\semig$ is a group $\Gamma$ with finite symmetric generating subset~$F$, because in this case $|\gamma' \gamma''|_F \geq |\gamma''|_F - |{\gamma'}^{-1}|_F = |\gamma''|_F - |\gamma'|_F$ for all $\gamma',\gamma''\in\Gamma$.
In a general semigroup $\semig$ with finite generating subset~$F$, the inequality $|\gamma' \gamma''|_F \geq |\gamma''|_F - |\gamma'|_F$ need not always hold (see Example~\ref{ex:semig-strict-ineq} with $\gamma' = a$ and $\gamma'' = b^3$).
Still, if $\semig$ admits an Anosov representation, or more generally (Remark~\ref{rem:Ano-discrete}) a representation satisfying \eqref{eqn:QI-semig} for some $C,C'>0$, then the following weaker form of the inequality holds.

\begin{remark} \label{rem:QI-embed-ineq-length}
Suppose that $\semig$ admits a representation $\rho : \semig\to\GL(d,\RR)$ such that there exist $C,C'>0$ with $\Vert\mu(\rho(\gamma)\Vert \geq C \, |\gamma|_F - C'$ for all $\gamma\in\semig$, as in \eqref{eqn:QI-semig}.
Let $M := \max_{f\in F} \Vert\mu(\rho(f))\Vert > 0$ and $(c,c') := (C M^{-1}, C' M^{-1})$.
Then for any $\gamma',\gamma''\in\semig$,
\begin{equation} \label{eqn:QI-embed-ineq-length}
|\gamma' \gamma''|_F \geq c |\gamma'|_F - c' - |\gamma''|_F \quad\quad\mathrm{and}\quad\quad |\gamma' \gamma''|_F \geq c |\gamma''|_F - c' - |\gamma'|_F.
\end{equation}
Indeed, let $(\gamma_1,\gamma_2)$ be $(\gamma',\gamma'')$ or $(\gamma'',\gamma')$.
By \eqref{eqn:mu-subadd} and \eqref{eqn:mu-lambda-inverse} we have $\Vert\mu(\rho(\gamma_2)^{-1})\Vert = \Vert\mu(\rho(\gamma_2))\Vert \leq M |\gamma_2|_F$ and $\Vert\mu(\rho(\gamma_1))\Vert -\nolinebreak \Vert\mu(\rho(\gamma_2)^{-1})\Vert \leq \Vert\mu(\rho(\gamma' \gamma''))\Vert \leq M |\gamma' \gamma''|_F$.
Therefore
$$M |\gamma' \gamma''|_F \geq \Vert\mu(\rho(\gamma_1))\Vert - \Vert\mu(\rho(\gamma_2)^{-1})\Vert \geq C |\gamma_1|_F - C' - M |\gamma_2|_F.$$
\end{remark}

We deduce the following.

\begin{lema} \label{lem:semig-bound-equiv}
Suppose that $\semig$ admits an Anosov representation, or more generally (Remark~\ref{rem:Ano-discrete}) a representation $\rho : \semig\to\GL(d,\RR)$ such that there exist $C,C'>0$ with $\Vert\mu(\rho(\gamma)\Vert \geq C \, |\gamma|_F - C'$ for all $\gamma\in\semig$, as in \eqref{eqn:QI-semig}.
Then the action of $\semig$ on $\semig^{\NN^*}\!/_{\sim}$ restricts to an action of $\semig$ on $\partial \semig$.
\end{lema}

\begin{proof}
By Remark~\ref{rem:QI-embed-ineq-length}, there exist $c,c'>0$ such that $|\gamma' \gamma''|_F \geq c |\gamma''|_F - c' - |\gamma'|_F$ for all $\gamma',\gamma''\in\semig$.
In particular, for any $\gamma' \in \semig$, any $\kappa\geq 1$, any $\kappa'\geq 0$, any $(f_k)_{k\in\NN} \in \cQ\cG_F^{\kappa,\kappa'}$, and any $n\in\NN^*$,
\begin{equation} \label{eqn:length-gamma-f-i}
|\gamma' f_0 \cdots f_{n-1}|_F \geq c \kappa^{-1} n - (c\kappa' + c' + |\gamma'|_F).
\end{equation}

Consider an element $\gamma\in\semig$ and a decomposition $\gamma=f^{\gamma}_0\cdots f^{\gamma}_{N-1}$ with $f^{\gamma}_i\in F$.
For $(f_k)_{k\in\NN}\in\cQ\cG_F^{\kappa,\kappa'}$ with $\kappa\geq 1$ and $\kappa'\geq 0$, let $(g_a)_{a\in\NN}\in F^{\NN}$ be given by $g_a=f^{\gamma}_a$ for $0\leq a\leq N-1$ and $g_a=f_{a-N}$ for $a\geq N$.
Then $g_0 \cdots g_{a-1} = \gamma f_0 \cdots f_{a-N-1}$ for all $a\geq N+1$, and so $\mathcal{P}((g_a)_{a\in\NN}) \sim \gamma\cdot\mathcal{P}((f_k)_{k\in\NN})$.
It remains to show that $(g_a)_{a\in\NN}\in\cQ\cG_F$.
Consider $a,b\in\NN$.
If $a+b\leq N-1$, then $|g_a \cdots g_{a+b}|_F \geq 0 \geq (b+1) - N$.
If $a\leq N-1<a+b$, then $g_a \cdots g_{a+b} = (f_a^{\gamma} \cdots f_{N-1}^{\gamma}) f_0 \cdots f_{a+b-N}$, and so \eqref{eqn:length-gamma-f-i} implies
$$|g_a \cdots g_{a+b}|_F \geq c \kappa^{-1} (a+b-N+1) - (c\kappa' + c' + (N-a)) \geq c \kappa^{-1} (b+1) - \kappa''$$
where $\kappa'' = c \kappa^{-1} N + c \kappa' + c' + N$.
If $a\geq N$, then $|g_a \cdots g_{a+b}|_F = |f_{a-N} \cdots f_{a+b-N}|_F \geq \kappa^{-1} (b+1) - \kappa'$.
We deduce $(g_a)_{a\in\NN}\in\cQ\cG_F$.
This shows that $\partial \semig = \mathcal{P}(\cQ\cG_F)/_{\sim}$ is stable under the action of~$\semig$.
\end{proof}

\begin{remarks} \label{rem:partial-semig-hyp-Cayley-graph}
\begin{enumerate}
  \item By construction, any $\gamma\in\semig$ with $|\gamma|_{\infty}>0$ fixes the point $\eta_{\gamma}^+ := [(\gamma^n)_{n\in\NN^*}] \in \partial \semig$ of Lemma~\ref{lem:eta-gamma+}.
  \item By Corollary~\ref{cor:QG-F-semigroup}, if $\semig$ admits an Anosov representation, then the quasigeodesic boundary $\partial\semig$ of~$\semig$, defined by \eqref{eqn:def-bound-semig}, can be seen as a subset of the set of equivalence classes of quasigeodesic rays in $\mathrm{Cay}(\semig,F)$ for the equivalence relation ``to remain at bounded distance''.
  It corresponds to the classes of the quasigeodesic rays `following directed edges', \ie quasigeodesic rays of the form $(\gamma_n)_{n\in\NN^*} \in \semig^{\NN^*}$ with $\gamma_1 \in F$ and $\gamma_{n+1} \in \gamma_n F$ for all $n\in\NN^*$.
  \item\label{item:partial-semig-hyp-3} If the proper geodesic metric space $\mathrm{Cay}(\semig,F)$ is Gromov hyperbolic, then the set of equivalence classes of quasigeodesic rays in $\mathrm{Cay}(\semig,F)$ identifies with the Gromov boundary of $\mathrm{Cay}(\semig,F)$ (see \cite[III.H.3]{bh99}).
  Thus, if $\semig$ admits an Anosov representation and $\mathrm{Cay}(\semig,F)$ is Gromov hyperbolic, then $\partial\semig$ can be seen as a subset of the Gromov boundary of $\mathrm{Cay}(\semig,F)$.
  \item\label{item:partial-semig-hyp-4} Suppose $\semig$ is a word hyperbolic group with a finite generating subset $F$ which is symmetric (\ie $F=F^{-1}$), and view $\semig$ as a semigroup with generating subset~$F$.
  Any point in the Gromov boundary of $\mathrm{Cay}(\semig,F)$ can be realized by a \emph{geodesic} (not only quasigeodesic) ray from~$e$ in $\mathrm{Cay}(\semig,F)$, hence of the form $(\gamma_n)_{n\in\NN} \in \semig^{\NN}$ with $\gamma_0 = e$ and $\gamma_{n+1} \in \gamma_n F$ for all $n\in\NN$.
  Using Remark~\ref{rem:QG-group}, we see that $\partial\semig$ identifies with the full Gromov boundary of $\mathrm{Cay}(\semig,F)$, \ie the boundary of~$\semig$.
\end{enumerate}
\end{remarks}

\subsection{The boundary map} \label{subsec:semig-bound-map}

Similar arguments to \cite[\S\,5]{GGKW} show the following.

\begin{lema} \label{lem:semig-Ano-bound-map}
Let $\semig$ be a finitely generated semigroup, let $1\leq i<d$ be integers, and let $\rho : \semig\to\GL(d,\RR)$ be a $P_i$-Anosov semigroup homomorphism.
The map $E^{cs} : \cQ\cG\to\mathrm{Gr}_i(\RR^d)$ of Proposition~\ref{prop:Ano-sv-gap-semig} induces a continuous $\rho$-equivariant map
$$\xi : \partial\semig \longrightarrow \mathrm{Gr}_i(\RR^d),$$
with the property that for any $\gamma\in\semig$ with $|\gamma|_{\infty}>0$, the image by $\xi$ of the fixed point $\eta_{\gamma}^+\in\partial\semig$ of Lemma~\ref{lem:eta-gamma+} is an attracting fixed point of $\rho(\gamma)$ in $\mathrm{Gr}_i(\RR^d)$.
\end{lema}

By ``$\xi(\eta_{\gamma}^+)$ is an attracting fixed point of $\rho(\gamma)$ in $\mathrm{Gr}_i(\RR^d)$'' we mean that there is an open neighborhood $\mathcal{U}$ of $\xi(\eta_{\gamma}^+)$ in $\mathrm{Gr}_i(\RR^d)$ such that $\rho(\gamma)^n\cdot z \to \xi(\eta_{\gamma}^+)$ as $n\to +\infty$ for all $z\in\mathcal{U}$, uniformly on compact sets.
For word hyperbolic groups, this property was referred to as ``$\xi$ is dynamics-preserving'' in \cite{GGKW}, since in that case $\eta_{\gamma}^+$ is an attracting fixed point of $\gamma$ in the boundary of the group; a similar situation arises in the setting of finitely generated completely simple semigroups, see Proposition~\ref{prop:main-Rees-semig} and Section~\ref{subsec:compl-simple-semig} below.
For more general semigroups, the fact that $\xi(\eta_{\gamma}^+)$ is an attracting fixed point of $\rho(\gamma)$ in $\mathrm{Gr}_i(\RR^d)$ does not fully describe the dynamics of $\semig$ on $\partial\semig$ since the map $\xi$ is not necessarily injective (see Section~\ref{subsec:semig-vs-gp}); it still imposes some restrictions on the dynamics, and could possibly lead to obstructions to the existence of Anosov representations.

\begin{proof}[Proof of Lemma~\ref{lem:semig-Ano-bound-map}]
By construction, the continuous map $E^{cs} : \cQ\cG_F\to\mathrm{Gr}_i(\RR^d)$ of Proposition~\ref{prop:Ano-sv-gap-semig} factors through the path map $\mathcal{P}$ of \eqref{eqn:def-P}.
Let us check that if $(f_k)_{k\in\NN},(g_k)_{k\in\NN}\in\cQ\cG_F$ satisfy $\mathcal{P}((f_k)_{k\in\NN}) \sim \mathcal{P}((g_k)_{k\in\NN})$, then $E^{cs}((f_k)_{k\in\NN}) = E^{cs}((g_k)_{k\in\NN})$.
Since $\rho$ is $P_i$-Anosov, there exist $c,c'>0$ such that $(\mu_i-\nolinebreak\mu_{i+1})(\rho(\gamma)) \geq c\,|\gamma|_F - c'$ for all $\gamma\in\semig$.
Consider $(f_k)_{k\in\NN},(g_k)_{k\in\NN}\in\cQ\cG_F$ such that $\mathcal{P}((f_k)_{k\in\NN}) \sim \mathcal{P}((g_k)_{k\in\NN})$: there exists $N>0$ such that for any $n\in\NN^*$ we can find $\varphi_n\in\NN^*$ with
$d_F(f_0\cdots f_{n-1}, g_0\cdots g_{\varphi_n-1}) \leq N$.
Note that $\varphi_n\to +\infty$ as $n\to +\infty$, for otherwise some subsequence of $(g_0\cdots g_{\varphi_n-1})_{n\in\NN^*}$ would be bounded, hence some subsequence of $(f_0\cdots f_{n-1})_{n\in\NN^*}$ would be bounded, contradicting the fact that $(f_k)_{k\in\NN}\in\cQ\cG_F$.
Consider the finite subset
$$\mathcal{M} := \big\{\rho(f) \,: f\in F\big\} \cup \big\{\rho(f)^{-1} \,: f\in F\big\} \cup \{\mathrm{id}\}$$
of $\GL(d,\RR)$.
Arguing as in Remark~\ref{rem:mu-dist-Cayley-graph}, we see that $\rho(g_0\cdots g_{\varphi_n-1}) \in \rho(f_0\cdots f_{n-1}) \, \mathcal{M}^N$ for all $n\in\NN^*$.
By \cite[Lem.\,5.8.(i)]{GGKW}, there are a metric $d_{\mathrm{Gr}_i}$ on $\mathrm{Gr}_i(\RR^d)$ and a constant $C_{\mathcal{M}^N}>\nolinebreak 0$ such that for any $g\in\GL(d,\RR)$ and $m\in\mathcal{M}^N$,
\begin{equation} \label{eqn:d-Xi-right-mult}
d_{\mathrm{Gr}_i}\big(\Xi_i(g), \Xi_i(gm)\big) \leq C_{\mathcal{M}^N} \, e^{-(\mu_i-\mu_{i+1})(g)}.
\end{equation}
We deduce
\begin{eqnarray*}
d_{\mathrm{Gr}_i}\big(\Xi_i\big(\rho(f_0\cdots f_{n-1})\big), \Xi_i\big(\rho(g_0\cdots g_{\varphi_n-1})\big)\big) & \leq & C_{\mathcal{M}^N} \, e^{-(\mu_i-\mu_{i+1})(\rho(f_0\cdots f_{n-1}))}\\
& \leq & C_{\mathcal{M}^N} \, e^{-c\,|f_0\cdots f_{n-1}|_F + c'} \ \underset{n\to +\infty}{\longrightarrow} 0.
\end{eqnarray*}
Therefore $E^{cs}((f_k)_k) = \lim_n \Xi_i(\rho(f_0 \cdots f_{n-1})) = \lim_n \Xi_i(\rho(g_0 \cdots g_{\varphi_n-1})) = E^{cs}((g_k)_k)$.
This shows that the continuous map $E^{cs} : \cQ\cG_F\to\mathrm{Gr}_i(\RR^d)$ of Proposition~\ref{prop:Ano-sv-gap-semig} induces a map $\xi : \partial\semig = \cP(\cQ\cG_F)/_\sim \to \mathrm{Gr}_i(\RR^d)$.
This map is continuous by definition of the quotient topology.

The map $\xi$ is $\rho$-equivariant.
Indeed, for any $\gamma\in\semig$, we can write $\gamma=f^{\gamma}_0\cdots f^{\gamma}_{N-1}$ for some $f^{\gamma}_i\in F$.
By  definition of the action of~$\semig$ (Lemma~\ref{lem:semig-bound-equiv}), for $(f_k)_{k\in\NN}\in\cQ\cG_F$ we have $\gamma\cdot [(f_k)_{k\in\NN}] := [(g_a)_{a\in\NN}]$ where $g_a=f^{\gamma}_a$ for $0\leq a\leq N-1$ and $g_a=f_{a-N}$ for $a\geq N$, and 
$$E^{cs}((g_a)_a) = \lim_n\ \Xi_i\big(\rho(\gamma)\,\rho(f_0 \cdots f_{n-1})\big).$$
By \cite[Lem.\,5.8.(ii)]{GGKW}, there are a metric $d_{\mathrm{Gr}_i}$ on $\mathrm{Gr}_i(\RR^d)$ and a constant $C_{\gamma}>0$ such that for any $g\in\GL(d,\RR)$,
$$d_{\mathrm{Gr}_i}\big(\Xi_i\big(\rho(\gamma)g\big), \rho(\gamma)\cdot\Xi_i(g)\big) \leq C_{\gamma} \, e^{-(\mu_i-\mu_{i+1})(g)}.$$
By taking $g=\rho(f_0\cdots f_{n-1})$ and letting $n\to +\infty$, we see that $E^{cs}((g_a)_a) = \rho(\gamma) \cdot\nolinebreak E^{cs}((f_k)_k)$, and so $\xi(\gamma\cdot [(f_k)_{k\in\NN}])=\rho(\gamma)\cdot\xi([(f_k)_{k\in\NN}])$ by passing to the quotient.

For $\gamma\in\semig$ with $|\gamma|_{\infty} > 0$, let us check that $\xi(\eta_{\gamma}^+)$ is an attracting fixed point of $\rho(\gamma)$ in $\mathrm{Gr}_i(\RR^d)$.
The fact that $\rho$ is $P_i$-Anosov implies that $\rho$ has a uniform $i$-gap in eigenvalues (Definition~\ref{def:gaps-semig}), by \eqref{eqn:lambda-mu-relation} and the definition of stable length.
In particular, $\lambda_i(\rho(\gamma))> \lambda_{i+1}(\rho(\gamma))$, hence $\rho(\gamma)$ has a unique attracting fixed point in $\mathrm{Gr}_i(\RR^d)$.
By \cite[Lem.\,5.11]{GGKW}, this attracting fixed point is the limit of $\Xi_i(\rho(\gamma^n))$ as $n$ tends to infinity.
This is equal to $\xi(\eta_{\gamma}^+)$ by definition of $\xi$ and~$\eta_{\gamma}^+$.
\end{proof}

\subsection{Openness of the space of Anosov representations} \label{subsec:semig-dom-splitting}

We consider the following condition.

\begin{defi} \label{def:property-D}
The semigroup $\semig$ has \emph{property (D)} (for `density of uniform quasigeodesic rays following directed edges') if there exist $\kappa\geq 1$ and $\kappa',N\geq 0$ such that for any $\gamma \in \semig$ we can find $(f_k)_{k\in\NN} \in \cQ\cG_F^{\kappa,\kappa'}$ and $n\in\NN^*$ with $d_F(\gamma,f_0 \cdots f_{n-1})\leq N$.
\end{defi}

\begin{remark}
Property~(D) is independent of the choice of finite generating set~$F$.
Indeed, let $F'$ be another finite generating subset of~$\semig$.
By the proof of Lemma~\ref{lem:QG-bound-semig-F-F'}, for any $\kappa\geq 1$ and $\kappa'\geq 0$ there exist $\kappa''\geq 1$ and $\kappa'''\geq 0$ such that for any $(f_k)_{k\in\NN} \in \cQ\cG_F^{\kappa,\kappa'}$ we can find $(f'_a)_{a\in\NN} \in \cQ\cG_{F'}^{\kappa'',\kappa'''}$ with $\{ f_0 \cdots f_{n-1} \,:\, n\in\NN^*\} \subset \{ f'_0 \cdots f'_{m-1} \,:\nolinebreak m\in\NN^*\}$.
On the other hand, by Remark~\ref{rem:change-gen-set-semig-Cay} there exists $M\geq 1$ such that $d_{F'}(\gamma_1,\gamma_2)\leq M\,d_F(\gamma_1,\gamma_2)$ for all $\gamma_1,\gamma_2\in\semig$.
\end{remark}

Property~(D) is related to \emph{automatic properties} of the language generated by the semigroup~$\semig$ (as is the case for word hyperbolic groups and the Cannon automaton, see \cite[\S\,5]{BPS} and references therein).
For instance, any nonabelian free semigroup or abelian free semigroup with free generating subset~$F$ has property~(D) with $\kappa=1$ and $\kappa'=N=0$.

Under property~(D), the following converse to Proposition~\ref{prop:Ano-sv-gap-semig}.\eqref{item:Ano-sv-gap-semig-1} holds.

\begin{prop} \label{prop:sv-gap-semig-Ano}
Let $\semig$ be a finitely generated semigroup, $1\leq i<d$ integers, and $\rho : \semig\to\GL(d,\RR)$ a semigroup homomorphism.
Suppose $\semig$ has property~(D) with constants $\kappa \geq 1$ and $\kappa'\geq 0$.
If the linear cocycle $(\sigma,\Phi_{\rho})$ over $\cQ\cG_F^{\kappa,\kappa'}$ has a dominated splitting of index $d-i$, then $\rho$ is $P_i$-Anosov.
\end{prop}

Propositions \ref{prop:Ano-sv-gap-semig}.\eqref{item:Ano-sv-gap-semig-1} and~\ref{prop:sv-gap-semig-Ano} together show that if $\semig$ has property~(D) with constants $\kappa \geq 1$ and $\kappa'\geq 0$ and if $(\sigma,\Phi_{\rho})$ has a dominated splitting of index $d-i$ as a linear cocycle over $\cQ\cG_F^{\kappa,\kappa'}$, then $(\sigma,\Phi_{\rho})$ actually has a dominated splitting of index $d-i$ as a linear cocycle over $\cQ\cG_F^{\kappa'',\kappa'''}$ for any $\kappa'' \geq 1$ and $\kappa'''\geq 0$.

\begin{proof}
If the linear cocycle $(\sigma, \Phi_\rho)$ over $\cQ\cG_F^{\kappa,\kappa'}$ has a dominated splitting of index $d-i$, then by Fact~\ref{fact-BG}, there exist $C,C'>0$ such that for any $n\in\NN$ and any $\sx\in\cQ\cG_F^{\kappa,\kappa'}$,
$$(\mu_{d-i} - \mu_{d-i+1})(\Phi_{\rho}^{(n)}(\sx)) \geq Cn - C'.$$
By property~(D), there exists $N\geq 0$ such that for any $\gamma \in \semig$ we can find $\sx = (f_k)_{k\in\NN} \in \cQ\cG_F^{\kappa,\kappa'}$ and $n\in\NN^*$ with $d_F(\gamma,f_0 \cdots f_{n-1})\leq N$.
By \eqref{eqn:mu-lambda-inverse} and \eqref{eqn:Phi-rho-n-semig},
\begin{eqnarray*}
(\mu_i - \mu_{i+1})\big(\rho(f_0 \cdots f_{n-1})\big) & = & (\mu_{d-i} - \mu_{d-i+1})\big(\rho(f_0 \cdots f_{n-1})^{-1}\big)\\
& = & (\mu_{d-i} - \mu_{d-i+1})\big(\Phi_{\rho}^{(n)}(\sx)\big)\ \geq\ Cn - C'.
\end{eqnarray*}
On the other hand, by Remark~\ref{rem:mu-dist-Cayley-graph}, the fact that $d_F(\gamma,f_0 \cdots f_{n-1})\leq N$ implies
$$ \| \mu(\rho(\gamma)) - \mu(\rho(f_0\cdots f_{n-1})) \| \leq C'' := N \, \max_{f\in F} \|\mu(\rho(f))\|, $$
and so $(\mu_i - \mu_{i+1})(\rho(\gamma)) \geq Cn - (C' + 2 C'')$.
This shows that $\rho$ has a uniform $i$-gap in singular values.
\end{proof}

Here is an immediate consequence of Fact~\ref{fact-dsopen} and Propositions \ref{prop:Ano-sv-gap-semig} and~\ref{prop:sv-gap-semig-Ano}.
We endow the space $\Hom(\semig,\GL(d,\RR))$ of semigroup homomorphisms from $\semig$ to $\GL(d,\RR)$ with the topology of uniform convergence on the finite generating subset~$F$; this topology does not depend on the choice of~$F$.

\begin{cor} \label{cor:Ano-open-semig}
Let $\semig$ be a finitely generated semigroup with property~(D).
For $1\leq i\leq d-1$, the space of $P_i$-Anosov semigroup homomorphisms is an open subset of $\Hom(\semig,\GL(d,\RR))$.
\end{cor}

\subsection{Anosov representations and eigenvalue gaps} \label{subsec:gaps-sv-ev-semig}

Similarly to Definition~\ref{def:prop-U-groups}, we consider the following property, which is again independent of the choice of finite generating subset~$F$, by Remark~\ref{rem:change-gen-set-semig}.

\begin{defi} \label{def:property-U}
The semigroup $\semig$ has \emph{property (U)} if there exist a finite subset $S'$ of $\semig^e$ and constants $c,c'>0$ such that for any $\gamma \in \semig$,
$$\max_{s\in S'} |s\gamma|_{F,\infty} \geq c \, |\gamma|_F - c'.$$
\end{defi}

For instance, any nonabelian free semigroup or abelian free semigroup with free generating subset~$F$ has property~(U) with $S'=\{e\}$ and $(c,c')=(1,0)$.
Similarly to Proposition~\ref{prop:eig-sing-gap-group}, we prove the following.

\begin{prop} \label{prop:eig-gap-semig}
Let $\semig$ be a finitely generated semigroup, let $1\leq i<d$ be integers, and let $\rho : \semig\to\GL(d,\RR)$ be a semigroup homomorphism.
\begin{itemize}
  \item If $\rho$ is $P_i$-Anosov, then it has a uniform $i$-gap in eigenvalues.
  \item The converse holds as soon as $\semig$ has property~(U) and the Zariski closure of $\rho(\semig)$ in $\GL(d,\RR)$ is reductive, or as soon as $\semig$ has properties (D) and~(U).
\end{itemize}
\end{prop}

\begin{proof}
The fact that a uniform $i$-gap in singular values implies a uniform $i$-gap in eigenvalues follows from \eqref{eqn:lambda-mu-relation} and from the definition of stable length, as we already observed in the proof of Lemma~\ref{lem:semig-Ano-bound-map}.

For the converse, we assume that $\semig$ has property~(U).
If we are not in the case that the Zariski closure of $\rho(\semig)$ in $\GL(d,\RR)$ is reductive, then we may reduce to this case as soon as $\semig$ has property~(D), using Corollary~\ref{cor:Ano-open-semig} and arguing as in the second paragraph of the proof of Proposition~\ref{prop:eig-sing-gap-group} in Section~\ref{subsec:proof-eig-sing-gap-group}.
So we now assume that $\semig$ has property~(U) and the Zariski closure of $\rho(\semig)$ in $\GL(d,\RR)$ is reductive.
Suppose that $\rho$ has a uniform $i$-gap in eigenvalues: there exist $c,c'>0$ such that $(\lambda_i - \lambda_{i+1})(\rho(\gamma)) \geq c \, |\gamma|_{F,\infty} - c'$ for all $\gamma\in\semig$.
Arguing exactly as in the fourth and fifth paragraphs of the proof of Proposition~\ref{prop:eig-sing-gap-group}, but replacing the translation length $\ell_F$ by the stable length $|\cdot|_{F,\infty}$ everywhere in the fourth paragraph (as we already did in the proof of Proposition~\ref{prop:strong-U-gap-hyp}), we see that $\rho$ has a uniform $i$-gap in singular values.
\end{proof}

\subsection{Differences with the group case} \label{subsec:semig-vs-gp}

Unlike in the group case, if a semigroup homomorphism $\rho : \semig\to\GL(d,\RR)$ is $P_i$-Anosov, then
\begin{itemize}
  \item the Cayley graph of $\semig$ is not necessarily Gromov hyperbolic;
  \item the continuous $\rho$-equivariant map $\xi : \partial\semig\to\mathrm{Gr}_i(\RR^d)$ of Lemma~\ref{lem:semig-Ano-bound-map} does not need to be injective: it can actually be constant;
  \item $\rho$ can be far from injective: there can exist two elements $\gamma,\gamma'\in\semig$, generating infinite cyclic subsemigroups with trivial intersection, such that $\rho(\gamma)=\rho(\gamma')$;
  \item although the semigroup $\rho(\semig)$ is discrete in $\GL(d,\RR)$ (Remark~\ref{rem:Ano-discrete}), the group generated by $\rho(\semig)$ is not necessarily discrete in $\GL(d,\RR)$;
  \item $\rho$ is not necessarily $P_{d-i}$-Anosov;
  \item when $\rho$ is $P_{d-i}$-Anosov, the continuous $\rho$-equivariant maps $\xi : \partial\semig\to\mathrm{Gr}_i(\RR^d)$ and $\xi' : \partial\semig\to\mathrm{Gr}_{d-i}(\RR^d)$ given by Lemma~\ref{lem:semig-Ano-bound-map} are not necessarily transverse: there may exist $\eta\neq\eta'$ in $\partial\semig$ such that the $i$-plane $\xi(\eta)$ and the $(d-i)$-plane $\xi'(\eta')$ intersect nontrivially.
\end{itemize}

\begin{ex} \label{ex:free-semig}
Let $\semig$ be either a nonabelian free semigroup or an abelian free semigroup on two generators $a,b$.
Let $\rho : \semig\to\GL(d,\RR)$ be a semigroup homomorphism such that $\rho(a) = \mathrm{diag}(e^{t_1},\dots,e^{t_d})$ is a diagonal matrix with $t_1>t_2\geq\dots\geq t_d$ and $\rho(b)\rho(a)^{-1}\in\mathrm{O}(d)$ fixes the vector $e_1 = (1,0,\dots,0) \in \RR^d$.
One easily checks that any element of $\rho(\semig)$ preserves $\RR e_1$ (which is an eigenline for its eigenvalue of maximal modulus) and $\mathrm{span}(e_2,\dots,e_d)$, and that $\mu_1(\rho(\gamma\gamma')) = \mu_1(\rho(\gamma)) + \mu_1(\rho(\gamma'))$ and $(\mu_1-\mu_2)(\rho(\gamma\gamma')) \geq (\mu_1-\mu_2)(\rho(\gamma)) + (\mu_1-\mu_2)(\rho(\gamma'))$ for all $\gamma,\gamma'\in\semig$.
It follows that $\rho$ is $P_1$-Anosov and that the continuous equivariant maps $E^{cs} : \cQ\cG_F\to\mathrm{Gr}_1(\RR^d)=\PP(\RR^d)$ and $\xi : \partial\semig\to\PP(\RR^d)$ of Proposition~\ref{prop:Ano-sv-gap-semig} and Lemma~\ref{lem:semig-Ano-bound-map} are constant with image~$[e_1]$.
\begin{enumerate}
  \item If $\rho(a)=\rho(b)$, then $\rho$ is not injective.
  \item If
  $\rho(b)\rho(a)^{-1}\in\mathrm{O}(d)$ has infinite order, then the group generated by $\rho(\semig)$ is not discrete in $\GL(d,\RR)$.
  \item If
  $t_2 = \dots = t_d$, then $\rho$ is not $P_i$-Anosov for any $2\leq i\leq d-1$.
  \item If
  $t_1 > t_2 = \dots = t_{d-1} > t_d$ and $\rho(b)$ fixes $e_d = (0,\dots,0,1) \in \RR^d$, then $\rho$ is $P_{d-1}$-Anosov and the continuous equivariant map $\xi' : \partial\semig\to\mathrm{Gr}_{d-1}(\RR^d)$ given by Lemma~\ref{lem:semig-Ano-bound-map} is constant; its image contains the image of~$\xi$, and so $\xi$ and~$\xi'$ are not transverse.
\end{enumerate}
\end{ex}

Despite these differences, one can still try to look for analogies between the semigroup case and the group case, at least for certain classes of semigroups (see \eg Section~\ref{subsec:compl-simple-semig})).
With this in mind, we ask the following:

\begin{quest} \label{qu:charact-Ano-semig-xi}
Can $P_i$-Anosov semigroup homomorphisms be characterized in terms of a boundary map $\xi : \partial\semig\to\mathrm{Gr}_i(\RR^d)$ as in \cite[Th.\,1.3]{GGKW} and \cite{klp-survey}?
\end{quest}

\begin{quest} \label{qu:D-U-semig}
If a finitely generated semigroup $\semig$ admits an Anosov representation into some $\GL(d,\RR)$, must $\semig$ have properties (D) and~(U)?
\end{quest}

\subsection{The case of completely simple semigroups} \label{subsec:compl-simple-semig}

We now focus on an important class of semigroups, namely that of \emph{completely simple} semigroups (see \cite[\S\,3.3]{how95}).
By definition, these are the semigroups that have no two-sided ideals other than themselves, but possess minimal one-sided ideals.
For such semigroups we provide affirmative answers to Questions \ref{qu:charact-Ano-semig-xi} and~\ref{qu:D-U-semig} and prove Proposition~\ref{prop:main-Rees-semig}.

\subsubsection{Rees semigroups}

In order to prove Proposition~\ref{prop:main-Rees-semig}, we use the well-known fact (see \eg \cite[Th.\,3.3.1]{how95}) that any completely simple semigroup $\semig$ isomorphic to a \emph{Rees semigroup} of the form $M(\Gamma,I,J,P) = I\times\Gamma\times J$ with multiplication given by
$$ (i, g, j) (i', g', j') = (i, gp_{j,i'}g', j'), $$
where $\Gamma$ is a group, $I,J$ are two nonempty sets, and $P = (p_{j,i})_{j\in J,\,i\in I}$ is a matrix with values in~$\Gamma$.
We shall assume that $\semig$ is finitely generated, which means that $I$ and~$J$ are finite and $\Gamma$ is finitely generated.
For any $(i,j)\in I\times J$, the set $\Gamma_{i,j}:=\{i\}\times\Gamma\times\{j\}$ is a subsemigroup of $M(\Gamma,I,J,P)$ and there is a semigroup isomorphism
\begin{eqnarray*}
\varphi_{i,j} :\ \Gamma & \longmapsto & \hspace{0.5cm}\Gamma_{i,j}.\\
g & \longmapsto & (i,g p_{j,i}^{-1},j)
\end{eqnarray*}
The semigroup $M(\Gamma,I,J,P)$ is the disjoint union of the $\Gamma_{i,j}=\varphi_{i,j}(\Gamma)$ for $(i,j)\in I\times J$, and so one may expect that certain properties of the group $\Gamma$ will transfer to $M(\Gamma,I,J,P)$.

Proposition~\ref{prop:main-Rees-semig} will be a consequence of the following proposition, where we say that $\rho : \semig\to\GL(d,\RR)$ is \emph{$P_k$-divergent} if $(\mu_k - \mu_{k+1})(\rho(\gamma_n)) \to +\infty$ for any sequence $(\gamma_n)_{n\in\NN}$ of pairwise distinct elements of~$\semig$.

\begin{prop} \label{prop:Rees-semig}
Let $\semig=M(\Gamma,I,J,P)$ be a finitely generated Rees semigroup.
Let $F_{\Gamma}$ be a symmetric (\ie $F_{\Gamma} = F_{\Gamma}^{-1}$) finite generating subset of~$\Gamma$, so that $F := \bigsqcup_{(i,j)\in I\times J} \varphi_{i,j}(F_{\Gamma})$ is a finite generating subset of~$\semig$.
\begin{enumerate}[(a)]
  \item\label{item:D-Rees} If $\Gamma$ has property~(D), then so does~$\semig$.
  \item\label{item:U-Rees} If $\Gamma$ has property~(U), then so does~$\semig$.
  \item\label{item:Cay-Rees} The Cayley graph $\mathrm{Cay}(\semig,F)$ is the disjoint union of a finite oriented subgraph, consisting of the vertex~$e$ and finitely many edges starting from it, and of finitely many connected oriented subgraphs $\mathrm{Cay}_i(\semig,F)$, $i\in I$, such that
  \begin{itemize}
    \item for $i\neq i'$, there is no edge between a vertex of $\mathrm{Cay}_i(\semig,F)$ and a vertex of $\mathrm{Cay}_{i'}(\semig,F)$; in particular, $d_F(x,x') = d_F(x,e) + d_F(e,x')$ for all $x\in\mathrm{Cay}_i(\semig,F)$ and $x'\in\mathrm{Cay}_{i'}(\semig,F)$ when $i\neq i'$;
    \item for any $(i,j)\in I\times J$ the semigroup homomorphism $\varphi_{i,j} : \Gamma\to\semig$ induces a quasi-isometry $(\mathrm{Cay}(\Gamma,F_{\Gamma}), d_{F_{\Gamma}}) \to (\mathrm{Cay}_i(\semig,F), d_F)$.
  \end{itemize}
  \item\label{item:boundary-Rees} For any $(i,j)\in I\times J$, the map $(g_k)_{k\in\NN} \mapsto (\varphi_{i,j}(g_k))_{k\in\NN}$ from $F_{\Gamma}^{\NN}$ to $F^{\NN}$ defines an embedding $\cQ\cG_{F_{\Gamma}}\hookrightarrow\cQ\cG_F$, which induces an embedding $\partial\varphi_i : \partial\Gamma\hookrightarrow\partial\semig$, independent of~$j$.
  The set $\partial\semig$ is the disjoint union of its subsets
  \begin{equation} \label{eqn:bound-phi-i}
  \partial\varphi_i(\partial\Gamma) = \big\{ [(f_k)_{k\in\NN}] \in \partial\semig \,: f_0 \in \{i\}\times\Gamma\times J\big\}
  \end{equation}
  for $i\in I$.
  Moreover, $\mathcal{P}(\cQ\cG_F)$ (see \eqref{eqn:def-P}) is the set of quasigeodesic rays in $\mathrm{Cay}(\semig,F)$ of the form $(\gamma_n)_{n\in\NN^*} \in \semig^{\NN^*}$ with $\gamma_1 \in F$ and $\gamma_{n+1} \in \gamma_n F$ for all $n\in\NN^*$.

  \item\label{item:action-Rees} The action of $\semig$ on $\semig^{\NN^*}\!/_{\sim}$ (see Section~\ref{subsec:semig-QG-bound}) restricts to an action of $\semig$ on $\partial\semig$.
  For any $(i,i',j)\in I\times I\times J$ and any $g\in\Gamma$,
    \begin{equation} \label{eqn:equiv-partial-varphi}
  \varphi_{i,j}(g) \circ \partial\varphi_{i'} = \partial\varphi_i \circ gp_{j,i}^{-1}p_{j,i'}.
  \end{equation}

  \item\label{item:hyp-Rees} The group $\Gamma$ is word hyperbolic if and only if $\mathrm{Cay}(\semig,F)$ is Gromov hyperbolic (as in \cite[Th.\,4.1]{fk04}).
  In this case, $\partial\semig$ is naturally homeomorphic to the Gromov boundary of $\mathrm{Cay}(\semig,F)$, and for any infinite-order element $\gamma\in\semig$ we have $|\gamma|_{\infty}>\nolinebreak 0$ and the point $\eta_{\gamma}^+\in\partial\semig$ of Lemma~\ref{lem:eta-gamma+} is an attracting fixed point of $\gamma$ in $\partial\semig$; more precisely, if $\gamma = \varphi_{i,j}(g)$ where $(i,j,g)\in I\times J\times\Gamma$, then $\eta_{\gamma}^+ = \partial\varphi_i(\eta_g^+)$ where $\eta_g^+$ is the attracting fixed point of $g$ in $\partial\Gamma$.
  \item\label{item:hom-Rees} Given $(i,j)\in I\times J$, any semigroup homomorphism $\rho : \semig\to\GL(d,\RR)$ induces a group homomorphism $\rho\circ\varphi_{i,j} : \Gamma\to\GL(d,\RR)$.
  \item\label{item:Ano-Rees} Let $(i,j)\in I\times J$ and $1\leq k\leq d-1$.
  A semigroup homomorphism $\rho : \semig\to\GL(d,\RR)$ is $P_k$-divergent (\resp has a uniform $k$-gap in singular values) if and only if the same holds for $\rho\circ\varphi_{i,j}$.
  In particular (using Fact~\ref{fact:Ano-sg-gap}), $\rho$ is $P_k$-Anosov if and only if $\Gamma$ is word hyperbolic and $\rho\circ\varphi_{i,j}$ is $P_k$-Anosov.
\end{enumerate}
\end{prop}

In~\eqref{item:hyp-Rees}, recall that if the group $\Gamma$ is word hyperbolic, then $\partial\Gamma$ identifies with the Gromov boundary of~$\Gamma$, see Remark~\ref{rem:partial-semig-hyp-Cayley-graph}.\eqref{item:partial-semig-hyp-4}.
By \cite{fk04}, for finitely generated Rees semigroups~$\semig$, the Gromov hyperbolicity of the Cayley graph of~$\semig$ is equivalent to a notion of hyperbolicity for~$\semig$ introduced by Gilman \cite{gil02} (see also \cite{dg04}) in language-theoretic terms.

\subsubsection{Length estimates}

Before proving Proposition~\ref{prop:Rees-semig}, we introduce some notation and establish useful estimates.
Let $\semig=M(\Gamma,I,J,P)$ be a Rees semigroup.
For any $i\in I$, we set $R_i := \{i\} \times \Gamma \times J = \bigsqcup_{j\in J} \Gamma_{i,j}$; it is a subsemigroup of~$\semig$, and in fact a right-ideal.
Let $F_{\Gamma}$ be a symmetric generating subset of~$\Gamma$.
As in Proposition~\ref{prop:Rees-semig},
$$F := \bigsqcup_{(i,j)\in I\times J} \varphi_{i,j}(F_{\Gamma})$$
is a generating subset of~$\semig$, and
$$F'_{\Gamma} := \big\{p_{j,i}^{-1} \, p_{j,i'} \,:\, (i,i',j)\in I\times I\times J\big\} \, F_{\Gamma} \, \big\{p_{j,i}^{-1} \, p_{j,i'} \,:\, (i,i',j)\in I\times I\times J\big\}$$
is again a generating subset of~$\Gamma$, containing~$F_{\Gamma}$.
We assume that $\semig$ is finitely generated, which means that $I$ and~$J$ are finite and $\Gamma$ is finitely generated, and we take $F_{\Gamma}$ (hence $F$ and~$F'_{\Gamma}$) to be finite.
We set
$$r := 1 + 4 \max_{(i,j)\in I\times J} |p_{j,i}|_{F_{\Gamma}} \geq 1.$$

\begin{lema} \label{lem:length-Rees-semig}
For any $g,g'\in\Gamma$, any $i,i'\in I$, any $j,j'\in J$, and any $\gamma,\gamma'\in\semig$,
\begin{enumerate}
  \item\label{eqn:word-length-Rees-semig} $r^{-1} |g^{-1}g'|_{F_{\Gamma}} \,\leq\, |g^{-1}g'|_{F'_{\Gamma}} \,\leq\, d_F(\varphi_{i,j}(g),\varphi_{i,j}(g'))$
  
  $\hspace{3.98cm} \leq |\varphi_{i,j}(g^{-1}g')|_F \,\leq\, |g^{-1}g'|_{F_{\Gamma}} =\nolinebreak d_{F_{\Gamma}}(g,g')$,
  \item\label{eqn:stable-length-Rees-semig} $r^{-1} |g|_{F_{\Gamma},\infty} \,\leq\, |g|_{F'_{\Gamma},\infty} \,\leq\, |\varphi_{i,j}(g)|_{F,\infty} \,\leq\, |g|_{F_{\Gamma},\infty}$,
  \item\label{eqn:phi-i-j-j'Rees-semig} $d_F(\varphi_{i,j}(g),\varphi_{i,j'}(g))\leq 2$,
  \item\label{eqn:phi-i-i'-jRees-semig} $d_F(\varphi_{i,j}(g)\,\varphi_{i',j}(g'),\varphi_{i,j}(g p_{j,i}^{-1} p_{j,i'} g'))\leq r$,
  \item\label{eqn:gamma-gamma'-d-Rees-semig} $d_F(\gamma, \gamma \gamma') \geq r^{-1} \, |\gamma'|_F - 3$,  %
  \item\label{eqn:gamma-gamma'-length-Rees-semig} $|\gamma\gamma'|_F \geq r^{-1} |\gamma'|_F - 1 - |\gamma|_F$.
  \end{enumerate}
\end{lema}

\begin{proof}
\eqref{eqn:word-length-Rees-semig} The first inequality is immediate: any element of~$F'_{\Gamma}$ can be written as a product of at most $r$ elements of~$F_{\Gamma}$ (using the fact that $|p_{j,i}^{-1}|_{F_{\Gamma}}=|p_{j,i}|_{F_{\Gamma}}$).
The third and fourth inequalities follow from \eqref{eqn:bound-dF} and from the fact that $\varphi_{i,j}$ is a semigroup homomorphism.
The right-hand equality comes from the fact that $\Gamma$ is a group.
To check the second inequality, we note that for any $(i_k,j_k,g_k) \in I \times J \times \Gamma$,
\begin{eqnarray} \label{eqn:prod-phi-i-j}
& & \varphi_{i_0,j_0}(g_0)\cdots\varphi_{i_{n-1},j_{n-1}}(g_{n-1})\nonumber\\
& = & \varphi_{i_0,j_{n-1}}\big((g_0 p_{j_0,i_0}^{-1} p_{j_0,i_1}) \, (g_1 p_{j_1,i_1}^{-1} p_{j_1,i_2}) \cdots (g_{n-1} p_{j_{n-1},i_{n-1}}^{-1} p_{j_{n-1},i_0})\big).
\end{eqnarray}
Suppose that there exist elements $\gamma_0,\dots,\gamma_n\in\semig$ such that $\gamma_0 = \varphi_{i,j}(g)$ and $\gamma_n = \varphi_{i,j}(g')$ and $\gamma_{k+1} \in \gamma_k F$ or $\gamma_k \in \gamma_{k+1} F$ for all~$k$.
If we write $\gamma_k = \varphi_{i_k,j_k}(g_k)$ for all~$k$, where $(i_k,j_k,g_k) \in I \times J \times \Gamma$, then \eqref{eqn:prod-phi-i-j} shows that $g_k^{-1} g_{k+1} \in F'_{\Gamma}$ for all~$k$.
In particular, $|g^{-1}g'|_{F'_{\Gamma}} \leq d_F(\varphi_{i,j}(g),\varphi_{i,j}(g'))$.

\eqref{eqn:stable-length-Rees-semig} This immediately follows from~\eqref{eqn:word-length-Rees-semig} by replacing $g$ by $g^n$, dividing by~$n$, and taking a limit.

\eqref{eqn:phi-i-j-j'Rees-semig} One readily checks that for $j_0\in J$ and $g'\in F_{\Gamma}$, we have $\varphi_{i,j'}(gg') =\linebreak \varphi_{i,j_0}(g)\,\varphi_{i,j'}(g')$, hence $d_F(\varphi_{i,j_0}(g),\varphi_{i,j'}(g g'))=1$.
We apply this to $j_0=j$ and $j_0=j'$, and conclude using the triangle inequality.

\eqref{eqn:phi-i-i'-jRees-semig} As in \eqref{eqn:prod-phi-i-j}, we have $\varphi_{i,j}(g)\,\varphi_{i',j}(g') = \varphi_{i,j}(g p_{j,i}^{-1} p_{j,i'} g' p_{j,i'}^{-1} p_{j,i})$.
On the other hand, using the fact that $\varphi_{i,j}$ is a semigroup homomorphism, \eqref{eqn:bound-dF}, and~\eqref{eqn:word-length-Rees-semig}, we get
$$d_F\big(\varphi_{i,j}\big(g p_{j,i}^{-1} p_{j,i'} g' p_{j,i'}^{-1} p_{j,i}\big), \varphi_{i,j}\big(g p_{j,i}^{-1} p_{j,i'} g'\big)\big) \leq \big|\varphi_{i,j}(p_{j,i'}^{-1}p_{j,i})\big|_F \leq |p_{j,i'}^{-1}p_{j,i}|_{F_{\Gamma}} \leq r.$$

\eqref{eqn:gamma-gamma'-d-Rees-semig} Write $\gamma = \varphi_{i,j}(g)$ and $\gamma' = \varphi_{i',j'}(g')$ where $(i,j,g), (i',j',g') \in I\times J\times\Gamma$.
As in \eqref{eqn:prod-phi-i-j}, we have $\gamma \gamma' = \varphi_{i,j'}(g p_{j,i}^{-1} p_{j,i'} g' p_{j',i'}^{-1}p_{j',i})$.
Using \eqref{eqn:phi-i-j-j'Rees-semig}, \eqref{eqn:word-length-Rees-semig}, and the fact that $|g_1 g_2|_{F_{\Gamma}} \leq |g_1|_{F_{\Gamma}} + |g_2|_{F_{\Gamma}}$ for all $g_1,g_2\in\Gamma$, we obtain
\begin{eqnarray*}
d_F(\gamma, \gamma \gamma') & \geq & d_F\big(\varphi_{i,j}(g), \varphi_{i,j}(g p_{j,i}^{-1}p_{j,i'} g p_{j',i'}^{-1} p_{j',i})\big) - 2\\
& \geq & r^{-1} \, \big|p_{j,i}^{-1} p_{j,i'} g p_{j',i'}^{-1} p_{j',i}\big|_{F_{\Gamma}} - 2\\
& \geq & r^{-1} \, |g|_{F_{\Gamma}} - 3 \ \geq\ r^{-1} \, |\varphi_{i',j'}(g)|_F - 3.
\end{eqnarray*}

\eqref{eqn:gamma-gamma'-length-Rees-semig} Write $\gamma = \varphi_{i,j}(g)$ and $\gamma' = \varphi_{i',j'}(g')$ where $(i,j,g), (i',j',g') \in I\times J\times\Gamma$.
As in \eqref{eqn:prod-phi-i-j}, we have $\gamma \gamma' = \varphi_{i,j'}(g p_{j,i}^{-1} p_{j,i'} g' p_{j',i'}^{-1}p_{j',i})$.
By \eqref{eqn:word-length-Rees-semig}, we have $|\gamma \gamma'|_F \geq r^{-1} |g p_{j,i}^{-1} p_{j,i'} g' p_{j',i'}^{-1}p_{j',i}|_{F_{\Gamma}}$.
On the other hand, using the fact that $|g_1|_{F_{\Gamma}} = |g_1 g_2 g_2^{-1}|_{F_{\Gamma}}\linebreak \leq |g_1 g_2|_{F_{\Gamma}} + |g_2^{-1}|_{F_{\Gamma}} \leq |g_1 g_2|_{F_{\Gamma}} + |g_2|_{F_{\Gamma}}$ for all $g_1,g_2\in\Gamma$ (since $F_{\Gamma}$ is symmetric), we have $|g p_{j,i}^{-1} p_{j,i'} g' p_{j',i'}^{-1}p_{j',i}|_{F_{\Gamma}} \geq |g'|_{F_{\Gamma}} - |g|_{F_{\Gamma}} - r$.
Using \eqref{eqn:word-length-Rees-semig} again, we obtain
$$|\gamma \gamma'|_F \geq r^{-1} |g'|_{F_{\Gamma}} - r^{-1} |g|_{F_{\Gamma}} - 1 \geq r^{-1} |\gamma'|_F - |\gamma|_F - 1. \hfill\qedhere$$
\end{proof}

\subsubsection{Proof of Proposition~\ref{prop:Rees-semig}}

\eqref{item:D-Rees} Suppose $\Gamma$ has property~(D) with constants $\kappa,\kappa',N$ for the generating subset~$F_{\Gamma}$.
For any $\gamma\in\semig$ we can write $\gamma=\varphi_{i,j}(g)$ for some $(i,j,g)\in I\times J\times\Gamma$.
By property~(D) for~$\Gamma$, we can find $(g_k)_{k\in\NN} \in \cQ\cG_{F_{\Gamma}}^{\kappa,\kappa'} \subset F_{\Gamma}^{\NN}$ and $n\in\NN^*$ with $d_{F_{\Gamma}}(g, g_0 \cdots g_{n-1})\leq N$.
Let $f_k := \varphi_{i,j}(g_k) \in\nolinebreak F$.
By Lemma~\ref{lem:length-Rees-semig}.\eqref{eqn:word-length-Rees-semig} and the fact that $\varphi_{i,j}$ is a semigroup homomorphism, we have $d_F(\gamma,f_0\cdots f_{n-1}) \leq N$ and $(f_k)_{k\in\NN} \in \cQ\cG_F^{r\kappa,r^{-1}\kappa'}$.
Thus $\semig$ has property~(D) with constants $r\kappa, r^{-1}\kappa', N$ for the generating subset~$F$.

\eqref{item:U-Rees} Suppose $\Gamma$ has property~(U): there exist a finite subset $S'_{\Gamma}$ of~$\Gamma$ and constants $c_{\Gamma},c'_{\Gamma}>0$ such that for any $g\in\Gamma$, we can find $s_{\Gamma}\in S'_{\Gamma}$ with $|s_{\Gamma} g|_{F_{\Gamma},\infty} \geq\nolinebreak c_{\Gamma} |g|_{F_{\Gamma}} -\nolinebreak c'_{\Gamma}$.
For any $\gamma\in\semig$ we can write $\gamma=\varphi_{i,j}(g)$ for some $(i,j,g)\in I\times J\times\Gamma$.
Consider $s_{\Gamma}\in S'_{\Gamma}$ such that $|s_{\Gamma} g|_{F_{\Gamma},\infty} \geq\nolinebreak c_{\Gamma} |g|_{F_{\Gamma}} -\nolinebreak c'_{\Gamma}$.
By Lemma~\ref{lem:length-Rees-semig}.\eqref{eqn:word-length-Rees-semig}--\eqref{eqn:stable-length-Rees-semig}, we have
$$|\varphi_{i,j}(s_{\Gamma}) \gamma|_{F,\infty} = |\varphi_{i,j}(s_{\Gamma} g)|_{F,\infty} \geq r^{-1} c_{\Gamma} |g|_{F_{\Gamma}} - r^{-1} c'_{\Gamma} \geq r^{-1} c_{\Gamma} |\gamma|_F - r^{-1} c'_{\Gamma}.$$
Thus $\semig$ has property~(U) with $S' := \bigsqcup_{(i,j)\in I\times J} \varphi_{i,j}(S'_{\Gamma})$ and constants $r^{-1}c_{\Gamma}, r^{-1}c'_{\Gamma}$.

\eqref{item:Cay-Rees} For any $i\in I$, let $\mathrm{Cay}_i(\semig,F)$ be the \emph{Sch\"utzenberger graph} of $R_i = \{i\} \times \Gamma \times J$, namely the oriented subgraph of $\mathrm{Cay}(\semig,F)$ with vertices $R_i$ and with a directed edge from $g$ to $g f$ for any $g\in R_i$ and $f\in F$.
The right-ideals $R_i$ and $R_{i'}$ of~$\semig$ are disjoint for $i\neq i'$ and so, inside $\mathrm{Cay}(\semig,F)$, there does not exist any edge between a vertex of $\mathrm{Cay}_i(\semig,F)$ and a vertex of $\mathrm{Cay}_{i'}(\semig,F)$ for $i\neq i'$.
In particular, the Cayley graph $\mathrm{Cay}(\semig,F)$ is the disjoint union of its oriented subgraphs $\mathrm{Cay}_i(\semig,F)$, for $i\in I$, and of a finite oriented subgraph, consisting of the vertex~$e$ and finitely many edges joining $e$ to the subgraphs $\mathrm{Cay}_i(\semig,F)$.
By Lemma~\ref{lem:length-Rees-semig}.\eqref{eqn:word-length-Rees-semig}, for any $(i,j)\in I\times J$ we have $d_F(\varphi_{i,j}(g),\varphi_{i,j}(g')) \geq r^{-1} \, d_{F_{\Gamma}}(g,g')$ for all $g,g'\in\Gamma$, hence the semigroup homomorphism $\varphi_{i,j} : \Gamma\to R_i$ induces a quasi-isometric embedding $(\mathrm{Cay}(\Gamma,F_{\Gamma}), d_{F_{\Gamma}}) \to (\mathrm{Cay}_i(\semig,F), d_F)$, which is in fact a quasi-isometry since any point of $R_i = \bigsqcup_{j'\in J} \varphi_{i,j'}(\Gamma)$ lies at uniformly bounded distance from a point of $\varphi_{i,j}(\Gamma)$ by Lemma~\ref{lem:length-Rees-semig}.\eqref{eqn:phi-i-j-j'Rees-semig}.

\eqref{item:boundary-Rees} Consider $(i,j)\in I\times J$.
By Lemma~\ref{lem:length-Rees-semig}.\eqref{eqn:word-length-Rees-semig}, for any $\kappa\geq 1$ and $\kappa'\geq 0$, and any $(g_k)_{k\in\NN} \in \cQ\cG_{F_{\Gamma}}^{\kappa,\kappa'}$, we have $(\varphi_{i,j}(g_k))_{k\in\NN} \in \cQ\cG_F^{r\kappa,r^{-1}\kappa'}$; moreover, $r^{-1} d_{F_{\Gamma}}(g,g') \leq d_F(\varphi_{i,j}(g),\varphi_{i,j}(g')) \leq d_{F_{\Gamma}}(g,g')$ for all $g,g'\in\Gamma$.
Therefore $(g_k)_{k\in\NN} \mapsto (\varphi_{i,j}(g_k))_{k\in\NN}$ defines an embedding $\cQ\cG_{F_{\Gamma}} \hookrightarrow \cQ\cQ_F$, which factors into an embedding $\partial\varphi_i : \partial\Gamma \hookrightarrow \partial\semig$.
By Lemma~\ref{lem:length-Rees-semig}.\eqref{eqn:phi-i-j-j'Rees-semig}, this embedding $\partial\varphi_i$ does not depend on~$j$.
By~\eqref{item:Cay-Rees}, we have $d_F(\varphi_{i,j}(g),\varphi_{i',j'}(g')) \geq d_F(\varphi_{i,j}(g),e)$ for all $(i,j,g), (i',j',g') \in I\times J\times\Gamma$ with $i\neq i'$, and so the images of $\partial\varphi_i$ and $\partial\varphi_{i'}$ are disjoint for $i\neq i'$.

Let us check the equality \eqref{eqn:bound-phi-i}.
By \eqref{eqn:prod-phi-i-j}, for any $(f_k)_{k\in\NN}\in\cQ\cG_F$, if we write $f_k = \varphi_{i_k,j_k}(g_k)$ with $(i_k,j_k,g_k) \in I\times J\times F_{\Gamma}$ and set $g'_k := g_k p_{j_k,i_k}^{-1} p_{j_k,i_{k+1}} \in F'_{\Gamma}$, then for any $k,\ell\in\NN$ we have

\begin{equation} \label{eqn:f-g-prime}
f_k \cdots f_{k+\ell} = \varphi_{i_k,j_{k+\ell}}\big(g'_k \cdots g'_{k+\ell} \, p_{j_{k+\ell},i_{k+\ell+1}}^{-1} \, p_{j_{k+\ell},i_k}\big).
\end{equation}
In particular, using Lemma~\ref{lem:length-Rees-semig}.\eqref{eqn:word-length-Rees-semig}, we see that for any $k,\ell\in\NN$,
$$ |g'_k \cdots g'_{k+\ell}|_{F'_{\Gamma}} \geq |g'_k \cdots g'_{k+\ell} \, p_{j_{k+\ell},i_{k+\ell+1}}^{-1} \, p_{j_{k+\ell},i_k}|_{F'_{\Gamma}} - 1 \geq |f_k \cdots f_{k+\ell}|_F - 1,$$
and so $(g'_k)_{k\in\NN} \in \cQ\cG_{F'_{\Gamma}}$.
On the other hand, taking $k=0$ in \eqref{eqn:f-g-prime} and using Lemma~\ref{lem:length-Rees-semig}.\eqref{eqn:word-length-Rees-semig}, we have $d_F(f_0 \cdots f_{n-1}, \varphi_{i_0,j_{n-1}}(g'_0\cdots g'_{n-1})) \leq r$ for all $n\in\NN^*$.
By Lemma~\ref{lem:length-Rees-semig}.\eqref{eqn:phi-i-j-j'Rees-semig} we have $d_F(\varphi_{i_0,j_{n-1}}(g'_0\cdots g'_{n-1}),\varphi_{i_0,j_0}(g'_0\cdots g'_{n-1})) \leq 2$, and so by the triangle inequality we obtain $d_F\big(f_0 \cdots f_{n-1}, \varphi_{i_0,j_0}(g'_0\cdots g'_{n-1})\big) \leq r+2$ for all $n\in\NN^*$.
This shows that $[(f_k)_{k\in\NN}] = [(\varphi_{i_0,j_0}(g'_k))_{k\in\NN}] \in \partial\varphi_{i_0}(\partial\Gamma)$.
Thus the right-hand side of \eqref{eqn:bound-phi-i} is included in the left-hand side.
The converse is clear.

Consider $(f_k)_{k\in\NN} \in F^{\NN}$ and let $(\gamma_n)_{n\in\NN^*} := \mathcal{P}((f_k)_{k\in\NN})$.
By Lemma~\ref{lem:length-Rees-semig}.\eqref{eqn:gamma-gamma'-d-Rees-semig}, for any $k,\ell\in\NN^*$,
$$d_F(\gamma_k, \gamma_{k+\ell}) = d_F(\gamma_k, \gamma_k f_k \cdots f_{k+\ell-1}) \geq r^{-1} |f_k \cdots f_{k+\ell+1}|_F - 3.$$
Therefore, if $(f_k)_{k\in\NN} \in \cQ\cG_F^{\kappa,\kappa'}$ for some $\kappa\geq 1$ and $\kappa'\geq 0$, then $\mathcal{P}((f_k)_{k\in\NN})$ is an $(r \kappa, r^{-1} \kappa + 3)$-quasigeodesic ray in $\mathrm{Cay}(\semig,F)$.
Given Lemma~\ref{lem:QG-semig-Cayley-graph}.\eqref{item:QG-semig-1}, this implies (as in Corollary~\ref{cor:QG-F-semigroup}) that $\mathcal{P}(\cQ\cG_F)$ is the set of quasigeodesic rays in $\mathrm{Cay}(\semig,F)$ of the form $(\gamma_n)_{n\in\NN^*} \in \semig^{\NN^*}$ with $\gamma_1 \in F$ and $\gamma_{n+1} \in \gamma_n F$ for all $n\in\NN^*$.

\eqref{item:action-Rees} 
By Lemma~\ref{lem:length-Rees-semig}.\eqref{eqn:gamma-gamma'-length-Rees-semig}, there exist $c,c'>0$ such that $|\gamma' \gamma''|_F \geq c |\gamma''|_F - c' - |\gamma'|_F$ for all $\gamma',\gamma''\in\semig$.
Arguing exactly as in the proof of Lemma~\ref{lem:semig-bound-equiv}, we obtain that the action of $\semig$ on $\semig^{\NN^*}\!/_{\sim}$ (as in Section~\ref{subsec:semig-QG-bound}) restricts to an action of $\semig$ on $\partial\semig$.
The equality \eqref{eqn:equiv-partial-varphi} follows immediately from Lemma~\ref{lem:length-Rees-semig}.\eqref{eqn:phi-i-i'-jRees-semig}.

\eqref{item:hyp-Rees} Fix $j\in J$.
By~\eqref{item:Cay-Rees}, the maps $\varphi_{i,j}$, for $i\in I$, define a quasi-isometry between the disjoint union of finitely many copies (indexed by~$I$) of $\mathrm{Cay}(\Gamma,F_{\Gamma})$, and $\mathrm{Cay}(\semig,F)$.
Therefore, $\mathrm{Cay}(\semig,F)$ is Gromov hyperbolic if and only if $\mathrm{Cay}(\Gamma,F_{\Gamma})$ is, which is equivalent to $\Gamma$ being word hyperbolic.
In this case, the maps $\varphi_{i,j}$ induce a homeomorphism between the disjoint union of finitely many copies (indexed by~$I$) of the Gromov boundary of $\mathrm{Cay}(\Gamma,F_{\Gamma})$, and the Gromov boundary of $\mathrm{Cay}(\semig,F)$.
By~\eqref{item:boundary-Rees}, the maps $\varphi_{i,j}$ also induce an identification between the disjoint union of finitely many copies (indexed by~$I$) of $\partial\Gamma$, and $\partial\semig$.
Since $\partial\semig$ identifies with some subset of the Gromov boundary of $\mathrm{Cay}(\semig,F)$ (see \eqref{item:boundary-Rees} and Remark~\ref{rem:partial-semig-hyp-Cayley-graph}.\eqref{item:partial-semig-hyp-3}) and $\partial\Gamma$ identifies with the full Gromov boundary of $\mathrm{Cay}(\Gamma,F_{\Gamma})$ (see Remark~\ref{rem:partial-semig-hyp-Cayley-graph}.\eqref{item:partial-semig-hyp-4}), we deduce that $\partial\semig$ identifies with the full Gromov boundary of $\mathrm{Cay}(\semig,F)$.

Consider an infinite-order element $\gamma = \varphi_{i,j}(g) \in \semig$, where $(i,j,g)\in I\times J\times\Gamma$.
Since $\varphi_{i,j}$ is a semigroup homomorphism, $g$ is an infinite-order element of~$\Gamma$; since $\Gamma$ is word hyperbolic, $g$ admits a unique attracting (\resp repelling) fixed point $\eta_g^+$ (\resp $\eta_g^-$) in $\partial\Gamma$ and $|g|_{F_{\Gamma},\infty} > 0$ (see \cite[Ch.\,9, \S\,3]{cdp90} and Remark~\ref{rem:partial-semig-hyp-Cayley-graph}.\eqref{item:partial-semig-hyp-4}).
The point $\eta_g^+$ (\resp $\eta_g^-$) is equal to the point $\eta_g^+$ (\resp $\eta_{g^{-1}}^+$) of Lemma~\ref{lem:eta-gamma+} .
By Lemma~\ref{lem:length-Rees-semig}.\eqref{eqn:stable-length-Rees-semig} we have $|\gamma|_{F,\infty} > 0$.
For any $\eta \in \partial\Gamma \smallsetminus \{\eta_g^-\}$ we have $g^n\cdot\eta \to \eta_g^+$ as $n\to +\infty$, and so, by \eqref{eqn:equiv-partial-varphi}, for any $i'\in I$ and any $\eta \in \partial\Gamma \smallsetminus \{ p_{j,i'}^{-1}p_{j,i}\cdot\eta_g^-\}$ we have $\varphi_{i,j}(g^n) \cdot \partial\varphi_{i'}(\eta) \to \partial\varphi_i(\eta_g^+)$.
This shows that $\partial\varphi_i(\eta_g^+)$ is an attracting fixed point of $\varphi_{i,j}(g)$ in $\partial\semig$.
It follows from the definitions of $\partial\varphi_i$ and $\eta_{\gamma}^+,\eta_g^+$ that $\eta_{\gamma}^+ = \partial\varphi_i(\eta_g^+)$.

\eqref{item:hom-Rees} Since $\Gamma$ and $\GL(d,\RR)$ are both groups, the semigroup homomorphism $\rho\circ\varphi_{i,j} : \Gamma\to\GL(d,\RR)$ is a group homomorphism: indeed, $(\rho\circ\varphi_{i,j})(g)(\rho\circ\varphi_{i,j})(g^{-1})=\mathrm{id}$ for all $g\in\Gamma$, hence $(\rho\circ\varphi_{i,j})(g^{-1})=(\rho\circ\varphi_{i,j})(g)^{-1}$.

\eqref{item:Ano-Rees} The property for $\rho$ (\resp $\rho\circ\varphi_{i,j}$) to be $P_k$-divergent is equivalent to the existence of a nondecreasing function $\vartheta : \NN\to (0,+\infty)$, going to infinity, such that $(\mu_k-\mu_{k+1})(\rho(\gamma))\geq\vartheta(|\gamma|_F)$ for all $\gamma\in\semig$ (\resp $(\mu_k-\mu_{k+1})(\rho\circ\varphi_{i,j}(g))\geq\vartheta(|g|_{F_{\Gamma}})$ for all $g\in\Gamma$); the property for $\rho$ (\resp $\rho\circ\varphi_{i,j}$) to have a uniform $k$-gap in singular values is equivalent to the existence of an \emph{affine} increasing function $\vartheta : \NN\to (0,+\infty)$ satisfying the same inequalities.
Since any $\gamma\in\semig$ can be written as $\gamma=\varphi_{i_0,j_0}(g)$ for some $(i_0,j_0)\in I\times J$ and $g\in\Gamma$, and in view of Lemma~\ref{lem:length-Rees-semig}.\eqref{eqn:word-length-Rees-semig}, we see that in order to prove that $\rho$ is $P_k$-divergent if and only if $\rho\circ\varphi_{i,j}$ is, and that $\rho$ has a uniform $k$-gap in singular values if and only if $\rho\circ\varphi_{i,j}$ does, it is sufficient to prove that there exists $m>0$ such that for any $(i_0,j_0)\in I\times J$ and $g\in\Gamma$,
\begin{equation} \label{eqn:mu-rho-phi-i-j}
\| \mu(\rho(\varphi_{i_0,j_0}(g))) - \mu(\rho(\varphi_{i,j}(g))) \| \leq m.
\end{equation}
Let us prove this.
By the triangle inequality and Lemma~\ref{lem:length-Rees-semig}.\eqref{eqn:phi-i-j-j'Rees-semig}--\eqref{eqn:phi-i-i'-jRees-semig},
\begin{eqnarray*}
& & d_F\big(\varphi_{i_0,j_0}(g), \varphi_{i_0,j}(p_{j,i}^{-1}p_{j,i_0})\,\varphi_{i,j}(g)\big)\\
& \leq & d_F(\varphi_{i_0,j_0}(g),\varphi_{i_0,j}(g)) + d_F\big(\varphi_{i_0,j}(g),\varphi_{i_0,j}(p_{j,i}^{-1}p_{j,i_0})\,\varphi_{i,j}(g)\big) \,\leq\, 2 + r,
\end{eqnarray*}
and so, by Remark~\ref{rem:mu-dist-Cayley-graph},
$$\| \mu(\rho(\varphi_{i_0,j_0}(g))) - \mu(\rho(\varphi_{i_0,j}(p_{j,i}^{-1}p_{j,i_0})\,\varphi_{i,j}(g))) \| \leq (2+r)\,\max_{f\in F} \|\mu(\rho(f))\|$$
is uniformly bounded.
On the other hand, by \eqref{eqn:mu-subadd},
$$\| \mu(\rho(\varphi_{i_0,j}(p_{j,i}^{-1}p_{j,i_0})\,\varphi_{i,j}(g))) - \mu(\rho(\varphi_{i,j}(g))) \| \leq \| \mu(\rho(\varphi_{i_0,j}(p_{j,i}^{-1}p_{j,i_0}))) \|$$
is uniformly bounded since $I$ and~$J$ are finite.
Applying the triangle inequality, we obtain the existence of a uniform constant $m>0$ such that \eqref{eqn:mu-rho-phi-i-j} holds.

This completes the proof of Proposition~\ref{prop:Rees-semig}.

\subsubsection{Proof of Proposition~\ref{prop:main-Rees-semig}}

The semigroup $\semig$ is isomorphic to a Rees semigroup $M(\Gamma,I,J,P)$ where $I$ and~$J$ are finite and $\Gamma$ is a finitely generated group.
Let $F_{\Gamma}$ be a symmetric finite generating subset of~$\Gamma$ and let $F:=\bigsqcup_{(i,j)\in I\times J}\,\varphi_{i,j}(F_{\Gamma})$.

If $\semig$ admits an Anosov representation, then $\Gamma$ is word hyperbolic by Proposition~\ref{prop:Rees-semig}.\eqref{item:Ano-Rees}, and so Proposition~\ref{prop:Rees-semig}.\eqref{item:hyp-Rees} ensures that the Cayley graph $\mathrm{Cay}(\semig,F)$ is Gromov hyperbolic, that the Gromov boundary of $\mathrm{Cay}(\semig,F)$ naturally identifies with $\partial\semig$, and that for any infinite-order element $\gamma\in\semig$ we have $|\gamma|_{\infty}>0$ and the point $\eta_{\gamma}^+\in\partial\semig$ of Lemmas \ref{lem:eta-gamma+} and~\ref{lem:semig-Ano-bound-map} is an attracting fixed point of $\gamma$ in $\partial\semig$.

We first check \eqref{item:Pk-Ano-Rees} $\Leftrightarrow$ \eqref{item:Pd-k-Ano-Rees}.
By Proposition~\ref{prop:Rees-semig}.\eqref{item:hom-Rees}--\eqref{item:Ano-Rees}, the semigroup homomorphism $\rho$ is $P_k$-Anosov (\resp $P_{d-k}$-Anosov) if and only if the group homomorphism $\rho\circ\varphi_{i,j}$ has a uniform $k$-gap (\resp $(d-k)$-gap) in singular values for all $(i,j)\in I\times J$.
On the other hand, for a group homomorphism, having a uniform $k$-gap is equivalent to having a uniform $(d-k)$-gap (Remark~\ref{rem:gaps}.\eqref{item:i-(d-i)-gap}).
Thus $\rho$ is $P_k$-Anosov if and only if it is $P_{d-k}$-Anosov.

We now check \eqref{item:Pk-Ano-Rees} $\Rightarrow$ \eqref{item:bound-maps-Rees}.
Suppose $\rho$ is $P_k$-Anosov; in particular, it is $P_k$-divergent.
Since $\rho$ is also $P_{d-k}$-Anosov by~\eqref{item:Pd-k-Ano-Rees}, Lemma~\ref{lem:semig-Ano-bound-map} yields two $\rho$-equivariant, continuous, dynamics-preserving boundary maps $\xi : \partial\semig\to\mathrm{Gr}_k(\RR^d)$ and $\xi' : \partial\semig\to\mathrm{Gr}_{d-k}(\RR^d)$.
Let us check that $\xi$ and~$\xi'$ are compatible and transverse.
For any $(i,j)\in I\times\nolinebreak J$, the embedding $\partial\varphi_i : \partial\Gamma \hookrightarrow \partial\semig$ of Proposition~\ref{prop:Rees-semig}.\eqref{item:boundary-Rees} is $\varphi_{i,j}$-equivariant by Proposition~\ref{prop:Rees-semig}.\eqref{item:action-Rees}.
The maps $\xi\circ\partial\varphi_i : \partial\Gamma\to\mathrm{Gr}_k(\RR^d)$ and $\xi'\circ\partial\varphi_i : \partial\Gamma\to\mathrm{Gr}_{d-k}(\RR^d)$ are continuous and $(\rho\circ\varphi_{i,j})$-equivariant.
They are dynamics-preserving because $\xi$ and~$\xi'$ are and because $\partial\varphi_i(\eta_g^+) = \eta_{\gamma}^+$ where $\eta_g^+$ is the attracting fixed point of $g$ in $\partial\Gamma$ (Proposition~\ref{prop:Rees-semig}.\eqref{item:hyp-Rees}).
By Proposition~\ref{prop:Rees-semig}.\eqref{item:hom-Rees}--\eqref{item:Ano-Rees}, the group homomorphism $\rho\circ\varphi_{i,j} : \Gamma\to\GL(d,\RR)$ is $P_k$-Anosov; the $(\rho\circ\varphi_{i,j})$-equivariant, continuous, dynamics-preserving maps $\xi\circ\partial\varphi_i$ and $\xi'\circ\partial\varphi_i$ must be its boundary maps, hence they are compatible and transverse (see \eg \cite[\S\,2.5.2]{GGKW}).
By Proposition~\ref{prop:Rees-semig}.\eqref{item:boundary-Rees}, any point of $\partial\semig$ can be written as $\eta = \partial\varphi_i(\zeta)$ for some $i\in I$ and $\zeta\in\partial\Gamma$; since $\xi\circ\partial\varphi_i : \partial\Gamma\to\mathrm{Gr}_k(\RR^d)$ and $\xi'\circ\partial\varphi_i : \partial\Gamma\to\mathrm{Gr}_{d-k}(\RR^d)$ are compatible, the $k$-plane $\xi(\eta)$ and the $(d-k)$-plane $\xi'(\eta)$ intersect in a plane of dimension $\min(i,d-i)$.
This shows that $\xi$ and~$\xi'$ are compatible.
Consider $\eta,\eta'\in\partial\semig$ such that $\gamma\cdot\eta \neq \gamma\cdot\eta'$ for some $\gamma\in\semig$.
By \eqref{eqn:equiv-partial-varphi}, there exist $i\in I$ and $\zeta\neq\zeta'$ in $\partial\Gamma$ such that $\gamma\cdot\eta = \partial\varphi_i(\zeta)$ and $\gamma\cdot\eta' = \partial\varphi_i(\zeta')$.
By transversality of $\xi\circ\partial\varphi_i$ and $\xi'\circ\partial\varphi_i$, the $k$-plane $\xi(\gamma\cdot\eta)$ and the $(d-k)$-plane $\xi'(\gamma\cdot\eta')$ intersect trivially; but $\xi(\gamma\cdot\eta) = \rho(\gamma)\cdot\xi(\eta)$ and $\xi'(\gamma\cdot\eta') = \rho(\gamma)\cdot\xi'(\eta')$ where $\rho(\gamma)\in\GL(d,\RR)$, hence $\xi(\eta)$ and $\xi'(\eta')$ intersect trivially.
This shows that $\xi$ and~$\xi'$ are transverse.

Conversely, let us check \eqref{item:bound-maps-Rees} $\Rightarrow$ \eqref{item:Pk-Ano-Rees}.
Suppose that there exist $\rho$-equivariant, continuous, dynamics-preserving, compatible, transverse boundary maps $\xi : \partial\semig\to\nolinebreak\mathrm{Gr}_k(\RR^d)$ and $\xi' : \partial\semig\to\mathrm{Gr}_{d-k}(\RR^d)$ and that $\rho$ is $P_k$-divergent.
Consider $(i,j)\in I\times J$.
As in the proof of \eqref{item:Pk-Ano-Rees} $\Rightarrow$ \eqref{item:bound-maps-Rees} above, using Proposition~\ref{prop:Rees-semig}.\eqref{item:boundary-Rees}--\eqref{item:action-Rees}--\eqref{item:hyp-Rees}, we see that the maps $\xi\circ\partial\varphi_i : \partial\Gamma\to\mathrm{Gr}_k(\RR^d)$ and $\xi'\circ\partial\varphi_i : \partial\Gamma\to\mathrm{Gr}_{d-k}(\RR^d)$ are $(\rho\circ\varphi_{i,j})$-equivariant, continuous, and dynamics-preserving.
They are compatible because $\xi$ and~$\xi'$ are.
They are transverse because $\xi$ and~$\xi'$ are and because if $\zeta,\zeta'\in\partial\Gamma$ are distinct, then $\varphi_{i,j}(g)\cdot\partial\varphi_i(\zeta) \neq \varphi_{i,j}(g)\cdot\partial\varphi_i(\zeta')$ for all $g\in\Gamma$, by $\varphi_{i,j}$-equivariance and injectivity of $\partial\varphi_i$.
Moreover, $\rho\circ\varphi_{i,j}$ is $P_k$-divergent by Proposition~\ref{prop:Rees-semig}.\eqref{item:Ano-Rees}.
Therefore $\rho\circ\varphi_{i,j}$ is $P_k$-Anosov by \cite{GGKW} or \cite{KLP}, and so $\rho$ is $P_k$-Anosov by Proposition~\ref{prop:Rees-semig}.\eqref{item:Ano-Rees}.

The equivalence \eqref{item:Pk-Ano-Rees} $\Leftrightarrow$ \eqref{item:eig-gap-Rees} follows from Proposition~\ref{prop:eig-gap-semig}.
Indeed, since $\Gamma$ is word hyperbolic it has property~(U) (see Example~\ref{ex:prop-U-groups}), and so $\semig$ has property~(U) by Proposition~\ref{prop:Rees-semig}.\eqref{item:U-Rees}.

The fact that $P_k$-Anosov representations form an open subset of $\Hom(\semig,\GL(d,\RR))$ follows from Corollary~\ref{cor:Ano-open-semig}.
Indeed, since $\Gamma$ is word hyperbolic it has property~(D) (see \eg \cite[Lem.\,5.8]{BPS}), and so $\semig$ has property~(D) by Proposition~\ref{prop:Rees-semig}.\eqref{item:D-Rees}.

This completes the proof of Proposition~\ref{prop:main-Rees-semig}.

\section{General reductive Lie groups} \label{sec:general-reductive}

In this section, we fix a noncompact real reductive Lie group $G$ which is a finite union of connected components (for the real topology) of $\mathbf{G}(\RR)$ for some algebraic group~$\mathbf{G}$ defined over~$\RR$.

\subsection{Lie-theoretic reminders}

Recall that $G$ admits a \emph{Cartan decomposition} $G=K\exp(\mathfrak{a}^+)K$ where $K$ is a maximal compact subgroup of~$G$ and $\mathfrak{a}^+$ a closed Weyl chamber in a Cartan subspace $\mathfrak{a}$ of the Lie algebra $\mathfrak{g}$ of~$G$.
Any $g\in G$ can be written $g=k\exp(\mu(g))k'$ for some $k,k'\in K$ and a unique $\mu(g)\in\mathfrak{a}^+$; this defines a map $\mu : G\to\mathfrak{a}^+$ (\emph{Cartan projection}) which is continuous, proper, and surjective.

Any element $g\in G$ can be written uniquely as the commuting product $g = g_h g_e g_u$ of a hyperbolic, an elliptic, and a unipotent element (\emph{Jordan decomposition}).
By definition, the conjugacy class of $g_h$ intersects $\exp(\aaa^+)$ in a unique element $\exp(\lambda(g))$; this defines a map $\lambda : G\to\aaa^+$ called the \emph{Jordan projection} or \emph{Lyapunov projection}.

Let $\Sigma\subset\aaa^*$ be the set of restricted roots of $\aaa$ in~$\mathfrak{g}$, let $\Delta\subset\Sigma$ be the set of simple roots with respect to the choice of~$\aaa^+$, so that
$$\aaa^+ = \{Y \in \aaa \,:\, \langle{\alpha,Y}\rangle \geq 0\;\; \forall \alpha \in \Delta\},$$
and let $\Sigma^+=\Sigma\cap\R^{>0}\text{-span}(\Delta)$ be the set of positive roots.
For any $\alpha\in\Sigma\cup\{0\}$, we set $\mathfrak{g}_{\alpha}:=\{Z\in\mathfrak{g} \,:\, \mathrm{ad}(Y)Z=\alpha(Y)\,Z \;\; \forall Y\in\aaa\}$.
Given a subset $\theta \subset \Delta$, we define $P_\theta$ (\resp $P_{\theta}^*$) to be the parabolic subgroup of~$G$ with Lie algebra $\mathfrak{g}_0 \oplus \bigoplus_{\alpha\in\Sigma^+} \mathfrak{g}_{\alpha} \oplus \bigoplus_{\alpha\in\Sigma^+\cap\mathrm{span}(\Delta\smallsetminus\theta)} \mathfrak{g}_{-\alpha}$ (\resp $\mathfrak{g}_0 \oplus \bigoplus_{\alpha\in\Sigma^+} \mathfrak{g}_{-\alpha} \oplus \bigoplus_{\alpha\in\Sigma^+\cap\mathrm{span}(\Delta\smallsetminus\theta)} \mathfrak{g}_{\alpha}$).
Then $P_{\emptyset} = G$ and $P_{\Delta}$ is a minimal parabolic subgroup of~$G$.

For any word hyperbolic group~$\Gamma$ and any nonempty $\theta\subset\Delta$, there is a notion of \emph{$P_{\theta}$-Anosov representation} $\rho : \Gamma\to G$, with continuous $\rho$-equivariant boundary maps $\xi : \partial\Gamma\to G/P_{\theta}$ and $\xi^* : \partial\Gamma\to G/P_{\theta}^*$ satisfying a transversality condition and~a~uni\-form contraction/expansion condition: see \eg \cite{GGKW} where the same notation~is~used.

\begin{ex}
For $G=\GL(d,\RR)$ we can take $K=\mathrm{O}(d)$ and
$$\mathfrak{a}^+ = \{ \mathrm{diag}(t_1,\dots,t_d) \ : \ t_1\geq\dots\geq t_d\} \ \subset\ \mathfrak{a} = \{ \mathrm{diag}(t_1,\dots,t_d) \ : \ t_1,\dots t_d\in\RR\} \simeq \RR^d.$$
With this choice the Cartan projection $\mu$ (\resp the Lyapunov projection $\lambda$) identifies with the map $\mu=(\mu_1,\dots,\mu_d)$ (\resp $\lambda=(\lambda_1,\dots,\lambda_d)$) of Section~\ref{subsec:limit-cones}, giving the list of logarithms of singular values (\resp of moduli of eigenvalues) of a matrix.
We have $\Sigma = \{\varepsilon_i-\varepsilon_j \,:\, 1\leq i\neq j\leq d\}$ and $\Delta = \{\varepsilon_i-\varepsilon_{i+1} \,:\, 1\leq i\leq d-1\}$.
The parabolic subgroup $P_i:=P_{\{\varepsilon_i-\varepsilon_{i+1}\}}$ is the stabilizer of the $i$-plane $\mathrm{span}(e_1,\dots,e_i)$ of~$\RR^d$.
\end{ex}

We shall use the following fact.

\begin{fact}[{see \cite[Lem.\,3.2 \& 3.7, Prop.\,3.5]{GGKW}}] \label{fact:from-G-to-GL}
For any nonempty $\theta\subset\Delta$, there exist $d\in\NN^*$ and a representation $\tau : G\to\GL(d,\RR)$ with the following properties:
\begin{itemize}
  \item $\min_{\alpha \in \theta}\, \langle\alpha,\mu(g)\rangle = (\mu_1-\mu_2)(\tau(g))$ for all $g\in G$,
  \item $\min_{\alpha \in \theta}\, \langle\alpha,\lambda(g)\rangle = (\lambda_1-\lambda_2)(\tau(g))$ for all $g\in G$,
  \item a representation $\rho : \Gamma\to G$ is $P_{\theta}$-Anosov if and only if $\tau\circ\rho : \Gamma\to\GL(d,\RR)$ is $P_1$-Anosov.
\end{itemize}
\end{fact}

\subsection{Gaps in the Lyapunov projection for representations of finitely generated groups}

Let $\theta$ be a nonempty subset of~$\Delta$ and $\Gamma$ a group with a finite generating subset~$F$.
Generalizing Definition~\ref{def:gaps}, we shall say that a representation $\rho : \Gamma\to G$ has a \emph{uniform $\theta$-gap in the Cartan projection} (\resp a \emph{weak uniform $\theta$-gap in the Lyapunov projection}, \resp a \emph{strong uniform $\theta$-gap in the Lyapunov projection}) if there exist $c,c'>0$ such that for any $\alpha\in\theta$ and any $\gamma\in\Gamma$, we have $\langle\alpha,\mu(\rho(\gamma))\rangle \geq c \, |\gamma|_F - c'$ (\resp $\langle\alpha,\lambda(\rho(\gamma))\rangle \geq c \, |\gamma|_{F,\infty} - c'$, \resp $\langle\alpha,\lambda(\rho(\gamma))\rangle \geq c \, \ell_F(\gamma) - c'$).
These notions do not depend on the choice of finite generating set~$F$, by Remark~\ref{rem:change-gen-set}.

Using Fact~\ref{fact:from-G-to-GL}, we see that the following hold similarly to Fact~\ref{fact:Ano-sg-gap}, Remark~\ref{rem:gaps}.\eqref{item:strong-weak-gap-hyp}, Proposition~\ref{prop:eig-sing-gap-group}, and Corollary~\ref{cor:char-Ano-eigenv}.

\begin{fact}[{\cite{KLP,BPS}}]
A representation $\rho : \Gamma\to G$ has a uniform $\theta$-gap in the Cartan projection if and only if $\Gamma$ is word hyperbolic and $\rho$ is $P_{\theta}$-Anosov.
\end{fact}

\begin{remark}
If $\Gamma$ is word hyperbolic, then having a weak uniform $\theta$-gap is equivalent to having a strong uniform $\theta$-gap in the Lyapunov projection.
\end{remark}

\begin{prop}
Suppose $\Gamma$ is word hyperbolic.
A representation $\rho : \Gamma\to G$ has a uniform $\theta$-gap in the Lyapunov projection if and only if it has a uniform $\theta$-gap in the Cartan projection.
\end{prop}

\begin{cor}
Suppose $\Gamma$ is word hyperbolic.
A representation $\rho : \Gamma\to G$ is $P_{\theta}$-Anosov if and only if it has a uniform $\theta$-gap in the Lyapunov projection.
\end{cor}

\subsection{Anosov representations for semigroups}

Let $\semig$ be a finitely generated semigroup and $\theta$ a nonempty subset of~$\Delta$.
Similarly to Definition~\ref{def:Anosov-semig}, we propose the following.

\begin{defi}
A semigroup homomorphism $\rho : \semig\to G$ is \emph{$P_{\theta}$-Anosov} if it has a uniform $\theta$-gap in the Cartan projection.
\end{defi}

For $g\in G$ with $\langle\alpha,\mu(g)\rangle>0$ for all $\alpha\in\theta$, we set $\Xi_{\theta}(g):=kP_{\theta}\in G/P_{\theta}$ where $k\in K$ satisfies $g\in k\exp(\mu(g))K$; this is well-defined.
Let $F$ be a finite generating subset of~$\semig$.
As in Section~\ref{subsec:semig-dom-splitting}, to any semigroup homomorphism $\rho : \semig\to G$ we associate a locally constant cocycle $(\sigma,\Phi_{\rho})$ over $\cQ\cG_F$ given by $\Phi_{\rho}((f_k)_{k\in\NN}):=\rho(f_0)^{-1}\in G$.
Recall properties (D) and~(U) from Definitions \ref{def:property-D} and~\ref{def:property-U}.
Using Fact~\ref{fact:from-G-to-GL}, we see that the following holds similarly to Proposition~\ref{prop:Ano-sv-gap-semig} and Lemma~\ref{lem:semig-Ano-bound-map}, Corollary~\ref{cor:Ano-open-semig}, Proposition~\ref{prop:eig-gap-semig}, and Proposition~\ref{prop:main-Rees-semig}.

\begin{prop}
If $\rho : \semig\to G$ is $P_{\theta}$-Anosov, then there is a continuous $(\sigma,\Phi_{\rho})$-equivariant map $E^{cs} : \cQ\cG_F\to G/P_{\theta}$ given by 
$$E^{cs}\big((f_k)_{k\in\NN}\big) = \lim_{n\to +\infty} \Xi_{\theta}\big(\rho(f_0 \cdots f_{n-1})\big),$$
which induces a continuous $\rho$-equivariant map $\xi : \partial\semig\to G/P_{\theta}$, with the property that for any $\gamma\in\semig$ with $|\gamma|_{\infty}>0$, the image by $\xi$ of the fixed point $\eta_{\gamma}^+\in\partial\semig$ of~$\gamma$ of Lemma~\ref{lem:eta-gamma+} is an attracting fixed point of $\rho(\gamma)$ in $G/P_{\theta}$.
\end{prop}

\begin{prop}
If $\semig$ has property~(D), then the space of $P_{\theta}$-Anosov semigroup homomorphisms is an open subset of $\Hom(\semig,G)$.
\end{prop}

\begin{prop}
If $\rho : \semig\to G$ is $P_{\theta}$-Anosov, then it has a uniform $\theta$-gap in the Jordan projection.
The converse holds as soon as $\semig$ has property~(U) and the Zariski closure of $\rho(\semig)$ in~$G$ is reductive, or as soon as $\semig$ has properties (D) and~(U).
\end{prop}

\begin{prop}
Suppose the semigroup $\semig$ is completely simple.
If $\semig$ admits an Anosov representation, then the Cayley graph of~$\semig$ is Gromov hyperbolic, with boundary $\partial\semig$, and any infinite-order element $\gamma\in\semig$ has a unique attracting fixed point $\eta_{\gamma}^+$ in $\partial\semig$.
In this case, for any semigroup homomorphism $\rho : \semig\to G$, the following are equivalent:
\begin{enumerate}
  \item $\rho$ is $P_{\theta}$-Anosov;
  \item $\rho$ is $(P_{\Delta\smallsetminus\theta})$-Anosov;
  \item there exist continuous $\rho$-equivariant dynamics-preserving boundary maps\linebreak $\xi : \partial\semig\to G/P_{\theta}$ and $\xi' : \partial\semig\to G/P_{\theta}^*$ which are compatible and transverse, and $\langle\alpha,\mu(\rho(\gamma_n))\rangle \to +\infty$ for any $\alpha\in\theta$ and any sequence $(\gamma_n)_{n\in\NN}$ of pairwise distinct elements of~$\semig$;
  \item $\rho$ has a uniform $\theta$-gap in the Lyapunov projection.
\end{enumerate}
Moreover, $P_{\theta}$-Anosov representations form an open subset of $\Hom(\semig,G)$.
\end{prop}

Here the notions of compatibility and transversality are defined as follows.
View $G/P_{\theta}$ (\resp $G/P_{\theta}^*$) as the set of parabolic subgroups of~$G$ conjugate to $P_{\theta}$ (\resp $P_{\theta}^*$).
We say that $\xi$ and~$\xi'$ are \emph{compatible} if for any $\eta\in\partial\semig$ the intersection of $\xi(\eta)\in G/P_{\theta}$ and $\xi'(\eta)\in G/P_{\theta}^*$ is a parabolic subgroup of~$G$, and that $\xi$ and~$\xi'$ are \emph{transverse} if for any $\eta\neq\eta'$ in $\partial\semig$ the intersection of $\xi(\eta)\in G/P_{\theta}$ and $\xi'(\eta')\in G/P_{\theta}^*$ is a reductive subgroup of~$G$.

\vspace{1cm}
{\small {\bf Acknowledgements:}
We thank Konstantinos Tsouvalas for pointing out the simplified proof of Proposition~\ref{prop:eig-sing-gap-group} which we reproduce here.
R.P. is particularly grateful to Jairo Bochi for many enlightening discussions; he also thanks Le\'on Carvajales and Andr\'es Sambarino for several discussions.
We thank the anonymous referees, Sami Douba and Ilia Smilga for useful comments and suggestions.
This work was initiated during a visit of R.P. at the Laboratoire de Math\'ematiques d'Orsay in the spring 2017; the hospitality of Orsay and of the Institut des Hautes \'Etudes Scientifiques is gratefully acknowledged.
Part of this work was completed during the workshop \emph{Dynamics Beyond Uniform Hyperbolicity} in Provo, UT, in June 2017; we are grateful to the scientific committee and organizers of the workshop for inviting us.}



\begin{thebibliography}{GGKW}

\bibitem[AMS]{AMS}
\textsc{H. Abels, G. A. Margulis, G. A. Soifer}, \textit{Semigroups containing proximal linear maps}, Israel J. Math.~91 (1995), p.~1--30.

\bibitem[AB]{Avila-Bochi-Trieste} 
\textsc{A. Avila, J. Bochi}, \textit{Lyapunov exponents}, Trieste Lecture notes (2008), available at \texttt{http://www.mat.uc.cl/}$\sim$\texttt{jairo.bochi/docs/trieste.pdf}

\bibitem[ABY]{ABY}
\textsc{A. Avila, J. Bochi, J-C. Yoccoz}, \textit{Uniformly hyperbolic finite-valued ${\rm SL}(2,\Bbb R)$-cocycles.} Comment. Math. Helv.~85 (2010), p.~813--884. 

\bibitem[Be]{Benoist97}
\textsc{Y. Benoist}, \textit{Propri\'et\'es asymptotiques des groupes lin\'eaires}, Geom. Funct. Anal.~7 (1997), p.~1--47.

\bibitem[Be$_2$]{Benoist-notes}
\textsc{Y. Benoist}, \textit{Sous-groupes discrets des groupes de Lie}, notes from the European Summer School in Group Theory, Luminy, 1997, available at \texttt{https://www.math.u-psud.fr/~benoist/prepubli/0097luminy.pdf}

\bibitem[BQ]{BenoistQuint-livre}
\textsc{Y. Benoist, J. F. Quint}, \textit{Random walks on reductive groups}, Ergebnisse der Mathematik und ihrer Grenzgebiete, 3. Folge, vol.~62, Springer, Cham, 2016.

\bibitem[Bo]{Bochi-ICM}
\textsc{J. Bochi}, \textit{Ergodic optimization of Birkhoff averages and Lyapunov exponents}, in \emph{Proceedings\,of\,the International Congress of Mathematicians 2018 (ICM 2018)}, p.~1821--1842, World Scientific,~2019.

\bibitem[Bo$_2$]{Bochi-ICMtalk}
\textsc{J. Bochi}, \textit{Ergodic optimization of Birkhoff averages and Lyapunov exponents}, slides of a talk at the ICM 2018, available at \texttt{http://www.mat.uc.cl/} $\sim$\texttt{jairo.bochi/docs/optimization\_LE\_talk\_ICM.pdf}

\bibitem[BGa]{BochiGaribaldi} 
\textsc{J. Bochi, E. Garibaldi}, \textit{Extremal norms for fiber bunched cocycles}, J. \'Ec. polytech. Math.~6 (2019), p.~947--1004.	

\bibitem[BGo]{BG}
\textsc{J. Bochi, N. Gourmelon}, \textit{Some characterizations of domination}, Math. Z.~263 (2009), p.~221--231.

\bibitem[BM]{BochiMorris}
\textsc{J. Bochi, I. Morris}, \textit{Equilibrium states of generalised singular value potentials and applications to affine iterated function systems}, Geom. Funct. Anal.~28 (2018), p.~995--1028.

\bibitem[BPS]{BPS}
\textsc{J. Bochi, R. Potrie, A. Sambarino}, \textit{Anosov representations and dominated splittings}, J. Eur. Math. Soc.~21 (2019), p.~3343--3414.

\bibitem[BR]{BR}
\textsc{J. Bochi, M. Rams}, \textit{The entropy of Lyapunov-optimizing measures of some matrix cocycles}, J. Mod. Dyn.~10 (2016), p.~255--286.

\bibitem[BT]{BorelTits}
\textsc{A. Borel, J. Tits}, \textit{Groupes r\'eductifs}, Publ. Math. Inst. Hautes \'Etudes Sci.~27 (1965), p.~55--151.

\bibitem[BS]{BS}
\textsc{E. Breuillard, C. Sert}, \textit{The joint spectrum}, J. Lond. Math. Soc.~103 (2021), p.~943--990.

\bibitem[BCLS]{BCLS}
\textsc{M. Bridgeman, R. D. Canary, F. Labourie, A. Sambarino}, \textit{The pressure metric for Anosov representations}, Geom. Funct. Anal.~25 (2015), p.~1089--1179.

\bibitem[BH]{bh99}
\textsc{M. R. Bridson, A. Haefliger}, \textit{Metric spaces of non-positive curvature}, Grundlehren der mathematischen Wissenschaften, vol.~319, Springer-Verlag, Berlin, 1999.

\bibitem[BE]{be15}
\textsc{J. Burillo, M. Elder}, \textit{Metric properties of Baumslag--Solitar groups}, Internat. J. Algebra Comput.~25 (2015), p.~799--811.

\bibitem[Bu]{Butler}
\textsc{C. Butler}, \textit{Characterizing symmetric spaces by their Lyapunov spectra}, arXiv:1709.08066v4.

\bibitem[CDP]{cdp90}
\textsc{M. Coornaert, T. Delzant, A. Papadopoulos}, \textit{G\'eom\'etrie et th\'eorie des groupes. Les groupes hyperboliques de Gromov}, Lecture Notes in Mathematics, vol.~1441, Springer-Verlag, Berlin, 1990.

\bibitem[CP]{CP}
\textsc{S. Crovisier, R. Potrie} \textit{Introduction to partial hyperbolicity}, Trieste Lecture notes (2015), available at \texttt{http://www.cmat.edu.uy/}$\sim$\texttt{rpotrie/documentos/pdfs/} \texttt{Crovisier-Potrie-PH.pdf}

\bibitem[DGK]{DGK}
\textsc{J. Danciger, F. Gu\'eritaud, F. Kassel}, \textit{Convex cocompact actions in real projective geometry}, arXiv:1704.08711.

\bibitem[DGLM]{DGLM}
\textsc{M. Delzant, O. Guichard, F. Labourie, S. Mozes}, \textit{Displacing representations and orbit maps}, p.~494--514, in \textit{Geometry, rigidity, and group actions}, University of Chicago Press, 2011.

\bibitem[DG]{dg04}
\textsc{A. J. Duncan, R. H. Gilman}, \textit{Word hyperbolic semigroups}, Math. Proc. Camb. Phil. Soc.~136 (2004), p.~513--524.

\bibitem[EF]{ef97}
\textsc{D. B. A. Epstein, K. Fujiwara}, \textit{The second bounded cohomology of word-hyperbolic groups}, Topology~36 (1997), p.~1275--1289.

\bibitem[F]{Feng}
\textsc{D. J. Feng}, \textit{Equilibrium states for factor maps between subshifts}, Adv. Math.~226 (2011), p.~2470--2502.

\bibitem[FK]{fk04}
\textsc{J. Fountain, M. Kambites}, \textit{Hyperbolic groups and completely simple semigroups}, in \emph{Semigroups and languages}, Proceedings of a 2002 workshop in Lisbon, World Scientific, 2004.

\bibitem[Gi]{gil02}
\textsc{R. H. Gilman}, \textit{On the definition of word hyperbolic groups}, Math. Z.~242 (2002), p.~529--541.

\bibitem[Go]{Gogolev}
\textsc{A. Gogolev}, \textit{Diffeomorphisms H\"{o}lder conjugate to Anosov}, Ergodic Theory Dynam. Systems~30 (2010), p.~441--456.

\bibitem[GGKW]{GGKW}
\textsc{F. Gu\'eritaud, O. Guichard, F. Kassel, A. Wienhard}, \textit{Anosov representations and proper actions}, Geom. Topol.~21 (2017), p.~485--584. 

\bibitem[GW]{gw12}
\textsc{O. Guichard, A. Wienhard}, \textit{Anosov representations : Domains of discontinuity and applications}, Invent. Math.~190 (2012), p.~357--438.

\bibitem[He]{Herman}
\textsc{M. Herman}, \textit{Une méthode pour minorer les exposants de Lyapounov et quelques exemples montrant le caract\`ere local d'un th\'eor\`eme d'Arnold et de Moser sur le tore de dimension~2}, Comment. Math. Helv.~58 (1983), p.~453--502.

\bibitem[Ho]{how95}
\textsc{J. M. Howie}, \textit{Fundamentals of semi-group theory}, London Mathematical Society Monographs, New Series, vol.~12, Oxford Science Publications, the Clarendon Press, Oxford University Press, New York, 1995.

\bibitem[Kal]{Kalinin}
\textsc{B. Kalinin}, \textit{Livsic theorem for matrix cocycles}, Ann. of Math.~173 (2011), p.~1025--1042.

\bibitem[KB]{KB}
\textsc{I. Kapovich, N. Benakli}, \textit{Boundaries of hyperbolic groups}, p.~39--94, in \textit{Combinatorial and geometric group theory}, Contemporary Mathematics, vol.~296, American Mathematical Society, Providence, RI, 2002.

\bibitem[KLP]{klp-survey}
\textsc{M. Kapovich, B. Leeb, J. Porti}, \textit{Some recent results on Anosov representations}, Transform. Groups~21 (2016), p.~1105--1121.

\bibitem[KLP$_2$]{KLP}
\textsc{M. Kapovich, B. Leeb, J. Porti}, \textit{A Morse lemma for quasigeodesics in symmetric spaces and euclidean buildings}, Geom. Topol.~22 (2018), p.~3827--3923.

\bibitem[Kas]{kas08}
\textsc{F. Kassel}, \textit{Proper actions on corank-one reductive homogeneous spaces}, J. Lie Theory~18 (2008), p.~961--978.

\bibitem[Kas$_2$]{Kassel-ICM}
\textsc{F. Kassel}, \textit{Geometric structures and representations of discrete groups}, in \emph{Proceedings of the International Congress of Mathematicians 2018 (ICM 2018)}, p.~1113--1150, World Scientific, 2019.

\bibitem[KH]{KH} 
\textsc{A. Katok, B. Hasselblatt} \textit{Introduction to the modern theory of dynamical systems}, Cambridge University Press, Cambridge, 1995.

\bibitem[Kn]{Knapp}
\textsc{A. Knapp}, \textit{Lie groups beyond an introduction}, second edition, Progress in Mathematics, vol.~140, Birkh\"auser Boston Inc., Boston, MA, 2002.

\bibitem[L]{lab06}
\textsc{F. Labourie}, \textit{Anosov flows, surface groups and curves in projective space}, Invent. Math.~165 (2006), p.~51--114.

\bibitem[LM]{LindMarcus}
\textsc{D. Lind, B. Marcus}, \textit{An introduction to symbolic dynamics and coding}, Cambridge University Press, Cambridge, 1995.

\bibitem[M]{mos56}
\textsc{G. D. Mostow}, \textit{Fully reducible subgroups of algebraic groups}, Amer. J. Math.~78 (1956), p.~200--221.

\bibitem[O]{osi10}
\textsc{D. Osin}, \textit{Small cancellations over relatively hyperbolic groups and embedding theorems}, Ann. of Math.~172 (2010), p.~1--39.

\bibitem[Pa]{Park}
\textsc{K. Park}, \textit{Quasi-multiplicativity of typical cocycles},  Comm. Math. Phys.~376 (2020), p.~1957--2004.

\bibitem[Po]{Potrie-ICM}
\textsc{R. Potrie}, \textit{Robust dynamics, invariant structures and topological classification}, in \emph{Proceedings of the International Congress of Mathematicians 2018 (ICM 2018)}, p.~2057--2080, World Scientific,~2019.

\bibitem[Ve]{Velozo}
\textsc{R. Velozo Ruiz}, \textit{Characterization of uniform hyperbolicity for fibre-bunched cocycles}, Dyn. Syst.~35 (2020), p.~124--139.

\bibitem[Vi]{Viana}
\textsc{M. Viana}, \textit{Lectures on Lyapunov exponents}, Cambridge Studies in Advanced Mathematics, vol.~145, Cambridge University Press, Cambridge, 2014.

\bibitem[We]{Weiss}
\textsc{B. Weiss}, \textit{Subshifts of finite type and sofic systems}, Monatsh. Math.~77 (1973), p 462--474.

\bibitem[Wi]{wie-icm}
\textsc{A. Wienhard}, \textit{An invitation to higher Teichm\"uller theory}, in \emph{Proceedings of the International Congress of Mathematicians 2018 (ICM 2018)}, p.~1007--1034, World Scientific, 2019.

\bibitem[Y]{Yoccoz}
\textsc{J-C. Yoccoz}, \textit{Some questions and remarks on $\SL(2,\RR)$ cocycles}, in \textit{Modern dynamical systems and applications}, p.~447--458, Cambridge University Press, Cambridge, 2004. 

\bibitem[Z]{zim}
\textsc{A. Zimmer}, \textit{Projective Anosov representations, convex cocompact actions, and rigidity}, J. Differential Geom.~119 (2021), p.~513--586.

\end{thebibliography}
\end{document}